\documentclass[a4paper,11pt,twoside]{article} 

\usepackage[utf8]{inputenc}
\usepackage[T1]{fontenc}
\usepackage{indentfirst}
\usepackage{fancyhdr}
\usepackage{graphicx}
\usepackage{physics}
\usepackage{color}
\usepackage{xcolor}
\usepackage[scaled=0.875]{helvet}

\usepackage[english]{babel}
\usepackage{emptypage}
\usepackage{amsmath}
\usepackage{amssymb}
\usepackage{amsfonts}
\usepackage{amsthm, amsbsy}
\usepackage{latexsym}
\usepackage{verbatim}
\usepackage{tikz}
\usepackage{tkz-euclide}
\usetikzlibrary{math}
\usepackage{empheq}
\usepackage[center,bf,small]{caption}
\usepackage{rotating}
\usepackage{lscape}
\usepackage{geometry}
\usepackage{todonotes}
\usepackage{hyperref}
\newtheorem{theorem}{Theorem}
\newtheorem{remark}{Remark}
\newtheorem{proposition}[theorem]{Proposition}

\newcommand\myeq[1]{\mathrel{\overset{\makebox[0pt]{\mbox{\normalfont\tiny #1}}}{=}}}
\newcommand\myleq[1]{\mathrel{\overset{\makebox[0pt]{\mbox{\normalfont\tiny #1}}}{\leq}}}
\newcommand\mygeq[1]{\mathrel{\overset{\makebox[0pt]{\mbox{\normalfont\tiny #1}}}{\geq}}}

\allowdisplaybreaks

\title{On the discrete equation model\\ for compressible multiphase fluid flows.}
\author{Marco Petrella$^1$ \and R. Abgrall$^2$ \and S. Mishra$^1$}
\date{
    $^1$Seminar of Applied Mathematics, ETH Zurich, Switzerland.\\%
    $^2$Department of Mathematics, University of Zurich, Switzerland.\\ %
\today
}

\begin{document}
\maketitle

\begin{abstract}
The modeling of multi-phase flow is very challenging, given the range of scales as well as the diversity of flow regimes that one encounters in this context. We revisit the discrete equation method (DEM) for two-phase flow in the absence of heat conduction and mass transfer. We analyze the resulting probability coefficients and prove their local convexity, rigorously establishing that our version of DEM can model different flow regimes ranging from the disperse to stratified (or separated) flow. Moreover, we reformulate the underlying mesoscopic model in terms of an one-parameter family of PDEs that interpolates between different flow regimes. We also propose two sets of procedures to enforce relaxation to equilibrium.
We perform several numerical tests to show the flexibility of the proposed formulation, as well as to interpret different model components. The one-parameter family of PDEs provides an unified framework for modeling mean quantities for a multiphase flow, while at the same time identifying two key parameters that model the inherent uncertainty in terms of the underlying microstructure. 
\end{abstract}

\section{Introduction}

The dynamical evolution of two (or more) distinct phases (of matter) is often referred to as \emph{multiphase flow}  and it is a very important topic of study in a broad variety of engineering systems, even though it is by no means limited to modern industrial design and can be observed in many natural/biological phenomena. A very limited list of references for multiphase flow include \cite{Faghri, Ishii, Drew&Passman, Stewart, BaerNunziato, Bdzil, Saurel&Abgrall, Abgrall&Saurel} and references therein.

The simplest, yet very representative, form of multiphase flow is two-phase flow. The mathematical modeling of two-phase flow arguably originated in the so-called multi-fluid models. Herein, one assumes that the dynamics of compressible inviscid fluid mixtures is modelled by the Euler equations \cite{Saurel&Abgrall}, where the characteristic middle field (contact discontinuity) consists of a material interface, if the adjacent data belong to different phases. Different parameters in the equations of state (EOS) are introduced in these models to represent the inherent heterogeneities in terms of the discontinuous variation of the pressure-density relations. Finally, additional conservation laws are inlucded to model species advection \cite{Saurel1999, Larrouturou, Abgrall1988,Abgrall1994,Abgrall2001,Toro2002,Karni, Cocchi}. Despite the inherent simplicity and flexibility of this approach, such models are often marred by spurious velocity and pressure oscillations near material interfaces \cite{Abgrall1988,Karni,Abgrall1994,Abgrall2001}, excessive numerical diffusion \cite{Saurel2018}, when approximated via classical schemes and negative mass fractions \cite{Larrouturou}.

An alternative and more popular approach, based on the theory of multiphase flows \cite{Ishii, Drew&Passman}, assumes each phase to be distinct and described by its own set of equations, typically the Euler equations. Pioneering works in this direction include those of Stewart and Wendroff \cite{Stewart} and Bear and Nunziato \cite{BaerNunziato}, see also \cite{Saurel&Abgrall}. This approach has now been extended into a wide variety of possible models. Following the observation that different phases interacts through the interface up to reaching uniform conditions \cite{Bdzil} (i.e. they move with
approximately the same pressure and velocity), the resulting set of equations is classified according to the set of independent variable they consider \cite{Zein}. Restricting the discussion here to one space dimension, we start with the so called four equation models \cite{Kapila, Yi2019}, which essentially resemble the reactive Euler equations and lead to similar difficulties as those experienced with the multifluid approach described above. 

Next, one considers the so-called five equations models \cite{Bdzil, Kapila, Murrone, Kreeft2010,Saurel2007}, where one  assumes a fully mechanical equilibrium between the phases,implying that the mixture is macroscopically moving with one-pressure and one-velocity. In \cite{Murrone}, it is shown how to derive the  five equation models from the Baer and Nunziato one by a formal asymptotic expansion assuming that the relaxation parameters tends together toward infinity, while their ratio stay bounded. In the case of non smooth solutions, a set of jump relations for the five equation model was provided in  \cite{Saurel2007}.

One can follow \cite{Saurel&Abgrall} and relax the assumption of mechanical equilibrium across phases. The resulting seven equation model requires the introduction of stiff source terms to model the underlying thermodynamics and leads to the removal of spurious oscillations around material discontinuities. Moreover, the source terms force a \emph{relaxation} to a single pressure and velocity recovering an experimentally observed fact in two-phase flows. Moreover, the zero relaxation limit of these models results in the five-equation model of  \emph{Kapila et al.}  \cite{Kapila}. 

Inspite of the tremendous progress made with regards to the modeling of two-phase flows as described above, several pressing issues remain. To start with, these mathematical models involve non-conservative products which make conservation of energy potentially difficult. Moreover, a mathematically sound solution concept, together with rigorous proofs of well-posedness, even in one space dimension, is extremely challenging. Notable exceptions are presented in \cite{Jin2019,Kwon2020,Novotny2020} where the authors provide a rigorous mathematical treatment of a simplified version of the Baer-Nunziato equations.

Furthermore, from a modelling perspective, a stark shortcoming of the many of the afore-mentioned models lies in the fact that the interfacial velocity and pressure are difficult to determine, see \cite{Hantke2021,Mueller2016,Bresch2018,Saurel&Abgrall, Liou2008, Coquel2017, Perrier2021,
BaerNunziato,Coquel2002,Enaux2006,Delhaye1982,Saurel2003,Glimm1996, Saurel2001} and references therein for a discussion of this issue as well as possible remedies. 

Given these shortcomings of the afore-mentioned models, one can see that there is no consensus on what constitutes a suitable modelling framework for two-phase flows. In particular, an uniform description of the vast range of flow regimes, ranging from isolated interfaces to fogs and microbubbles, within the purview of a single predictive model is extremely challenging. The search for such a framework brings us to the so-called \emph{Discrete Equation Method} (DEM) of \cite{Abgrall&Saurel}, see also \cite{Abgrall1994}. Inspired by the Godunov method and well-established theories of ensemble averaging \cite{Drew&Passman}, DEM entails the statistical description of each phase in terms of its own equation of state and allows for, in principle, all possible flow regimes. A multiscale formulation allows one to incorporate information from finer scales. One can think of DEM as a mesoscopic model as its does not require an explicit description of the underlying microstructure. 

Despite its promise as a suitable modeling framework for multiphase flows, DEM still requires user-defined ansatz (closure relations) on the probability coefficients that arise in course of the ensemble averaging procedure. Although many papers such as \cite{Saurel2017} suggest modifications for overcome this issue, for instance in the case of simulating dense-to-dilute transitions by coupling the underlying Euler equations with an evolution equation on the number of dispersed particles, it is fair to say the design of a flexible general purpose DEM type model, which can describe various flow regimes is still outstanding. 

These limitations of the DEM approach constitute the starting point of the current paper. Herein, we will carefully develop and analyze the DEM approach for describing two-phase flows in one space dimension, while neglecting heat and mass transfer. Our main aim would be to characterize the probability coefficients that arise in the DEM framework of \cite{Abgrall&Saurel} such that all possible flow regimes can be described by DEM. This will allow us to encapsulate all phase interactions in terms of a single parameter that interpolates between disperse and stratified flows. Moreover, simple relaxation procedures will also be investigated. This will allow us to study numerically, how different choices of parameters leads to the recovery of different flow regimes, enabling a thorough analysis of the expressivity as well as limitations of DEM for different regimes of multiphase flow.  

The rest of the paper is organized as follows: 
In Section $2$ we summarize the DEM procedure, highlighting the modelling assumptions related to such procedure. Section $3$ is dedicated to the analysis of the probability coefficients resulting from the previous section, and Section $4$ derives the corresponding one-parameter limit along with the numerical strategy to solve it. Finally, Section $5$ include the numerical experiments we have performed on such models, and discussion of the outcomes is carried out in Section $6$.

\section{The Discrete Equation Method}
\label{sec:dem}
In this section, we will present the discrete equation method for modeling two-phase flows in one space dimension. We start with a succinct presentation of the ensemble averaging theory on which DEM is based.

\subsection{The ensemble averaging theory}

In the following we recall the procedure of \cite{Abgrall&Saurel} for a biphasic Eulerian flow without mixing. Phase transition is excluded from the present study and we suppose that heat transfer is too slow compared to mechanical relaxation \cite{Kapila}.\\
We consider two phases $\Sigma_1$ and $\Sigma_2$, each governed by the Euler equations
\begin{equation}
\label{Euler}
\partial_t \textbf{U}^{(k)} + \partial_x \textbf{F}^{(k)}(\textbf{U}) = \textbf{0}
\end{equation}
where $\textbf{U}^{(k)} = [\rho^{(k)}, \rho^{(k)} u^{(k)}, \rho^{(k)} E^{(k)}]^T$ 
and 
$\textbf{F}^{(k)} = [\rho^{(k)}{u^{(k)}}, \rho^{(k)} {u^{(k)}}^2 + p^{(k)}, \left(\rho^{(k)} E^{(k)} + p^{(k)}\right)u^{(k)}]^T$. 
The notation is classical: $\rho^{(k)}, u^{(k)}, p^{(k)}$ denote the density, velocity and pressure of the phase $k\in \lbrace 1, 2 \rbrace$. The total energy $E^{(k)} = \frac{1}{2} {u^{(k)}}^2 + e^{(k)}$, where $e^{(k)}$ denotes the internal energy. Different choices of equation of state (EOS) have severe implications on the flow regime and a typical issues in multiphase flow is the determination of a methodology that handles different EOS.

As it is well-known \cite{Drew&Passman}, a prime characteristic of 
multiphase mixtures is that there is uncertainty in the exact location of the particular constituents at any particular time.
In turn, from the practical point of view, for a given set of initial and boundary conditions, a single measurement of such experiment carries limited information about the mean and distribution of dispersed particles that generated such results. For this reason, modern multiphase flow theory is described in averaged sense.
In our case, we aim at considering both the spatial rearrangement of disperse particles and the statistical description of repeated sampling for a fixed set of initial and boundary condition.

\subsubsection{Notation}

We hereby introduce some notations.
Let $\left(\Omega, \mathcal{F}, \mathbb{P}\right)$ be a probability space on $\mathbb{R}$. We denote the physical space of interest (i.e. domain) by an open set $D\subseteq\mathbb{R}^d$, where $d\in\mathbb{N}$ is the spatial dimension. The time horizon is denoted by $T>0$, and the any time considered for our simulations is denoted by $t \in [0,T]$.
We aim at including the randomized dependency of quantities of interest by taking random fields between the spaces $(\Omega, \mathbb{F}, \mathbb{P})$ and the space of $p$-integrable functions $L^p\left(D\times \mathbb{R}_{+}; U\right)$, with $U\subset\mathbb{R}^N$. Here $N\in\mathbb{N}$ is the number of quantities of interest of the system under consideration.

Existence and uniqueness (well-posedness) of solutions for systems of hyperbolic conservation laws is restricted to one-dimensional ($d = 1$) and for sufficiently small initial data \cite{MR1816648}. More sophisticated solution paradigma \cite{Fjordholm2017} are also available but they are out of the scope of this work. We therefore restrict our description to the cases $d=1$.

In such a case, \emph{weak}-solution are typically found in the subspace $BV \left( D\times \mathbb{R}_+; U\right)$. We will consider random variables between the spaces $(\Omega,\mathcal{F})$ and $(\mathcal{X}, \mathcal{B}(\mathcal{X}))$, where the topological space $\mathcal{X} = L^{1}$ is endowed with the Borel-sigma algebra $\mathcal{B}(\mathcal{X})$, as to make each continuous function measurable.
Let $\omega\in \Omega$ be a fixed realization. At each time level $t\in [0,T]$ we will assume that there exist a pair of open sets $D_1(t;\omega),D_2(t;\omega)$ affected by only one phase, namely $D_k(t;\omega) := \lbrace x\in D\, \vert\, \textit{phase k is present at }(x,t)\rbrace$ such that

\begin{enumerate}
\item (\underline{Non-mixing condition}) Only one phase is present at each space-time location: 
\[
D_1(t;\omega)\cap D_2(t;\omega) = \emptyset
\]
\item (\underline{Saturation condition}) No vacuum is generated at any space-time location:
\[
D = \overline{D_1(t;\omega)}\cup\overline{D_2(t;\omega)} \setminus \partial D.
\]
where $\partial D$ denotes the frontier of $D$.
\end{enumerate}

\noindent
The interface between the two-phases is then defined according to the following relation:
\begin{equation*}
I(t;\omega) = \overline{D_1(t;\omega)} \cap \overline{D_2(t;\omega)}\setminus \partial D
\end{equation*}
We introduce the characteristic function $X^{(k)}\,:\,\Omega\rightarrow\mathcal{X}$ associated to phase $k$ as the indicator function over the points of the domain $D$ affected by phase $k$, namely 
\begin{equation}
\label{eq:characteristi}
X^{(k)}\,:\, \omega\in\Omega \longmapsto X^{(k)}(x,t;\omega) =
\begin{cases}
1 & \textit{if}\quad x\in D_k(t; \omega)\\
0 & \textit{otherwise}
\end{cases} 
\qquad
\forall (x,t) \in D\times \mathbb{R}_+
\end{equation}
Using standard theory of distribution, the characteristic function can be shown to satisfy the following topological equation (suppressing $\omega$-dependence for notational convenience)  \cite{Drew&Passman}
\begin{equation}
\label{CharEq}
\partial_t X^{(k)} + \sigma \partial_x X^{(k)} = 0
\end{equation}
where $\sigma$ is the interface velocity of the realization highlighted by $X^{(k)}(\cdot;\omega)$. Hence, one can also show that upon multiplication of (\ref{Euler}) by the characteristic function it holds
\begin{equation}\label{Euler_chi}
\partial_t X^{(k)} \textbf{U}^{(k)} + \partial_x X^{(k)}\textbf{F}^{(k)} = {\textbf{F}^{(k)}}^{lag}\partial_x X^{(k)}
\end{equation}
where the Lagrangian flux ${\textbf{F}^{(k)}}^{lag} := \textbf{F}^{(k)}_I - \sigma \textbf{U}^{(k)}_I$ and the subindex $I$ denotes the interfacial value from the $k$-th side.
We introduce the ensemble average operator $\mathcal{E}$ \cite{Drew&Passman} that is assumed to commute with time and space derivative operators (these are commonly referred as Gauss and Leibniz Rules, which hold for well-behaved input functions). Taking ensemble average on (\ref{Euler_chi}) and (\ref{CharEq}), one obtains the following equation
\begin{equation}
\label{Euler_ensamble}
\begin{cases}
\partial_t \mathcal{E}\left[X^{(k)} \textbf{U}^{(k)}\right] + \partial_x \mathcal{E}\left[X^{(k)}\textbf{F}^{(k)}\right] = \mathcal{E}\left[\left(\textbf{F}^{(k)}_I - \sigma \textbf{U}^{(k)}_I\right)\partial_x X^{(k)}\right]\\
\partial_t \mathcal{E}\left[X^{(k)}\right] + \mathcal{E}\left[\sigma\partial_x X^{(k)}\right] = 0
\end{cases}
\end{equation}
We thus introduce the notation that will be used throughout this paper: let
\begin{equation}\label{EA_vars}
\textbf{U}_k := \mathcal{E}\left[X^{(k)}\textbf{U}^{(k)}\right] = [\alpha_k \rho_k, \alpha_k \rho_k u_k, \alpha_k \rho_k E_k]^T
\end{equation}
where the ensemble average quantities are defined via
\begin{equation}
\alpha_k :=\mathcal{E}\left[X^{(k)}\right],\, 
\rho_k :=\frac{\mathcal{E}\left[X^{(k)}\rho^{(k)}\right]}{\alpha_k},\,
u_k :=\frac{\mathcal{E}\left[X^{(k)}\rho^{(k)}u^{(k)}\right]}{\alpha_k\rho_k},\,
p_k :=\frac{\mathcal{E}\left[X^{(k)}p^{(k)}\right]}{\alpha_k},\,
e_k :=\frac{\mathcal{E}\left[X^{(k)}\rho^{(k)}e^{(k)}\right]}{\alpha_k\rho_k}
\end{equation}
so that $E_k := \frac{1}{2}u_k^2 + e_k$. Using this notation the ensemble-average flux can be written as
\begin{equation}\label{EA_flux}
\begin{split}
 \mathcal{E}\left[ X^{(k)}\textbf{F}^{(k)} \right] & =\underbrace{
\begin{bmatrix}
\alpha_k\rho_k u_k\\
\alpha_k\rho_ku_k^2 + \alpha_kp_k\\
\alpha_k u_k(\rho_k E_k + p_k)
\end{bmatrix}
} 
+
\underbrace{
\begin{bmatrix}
0\\
\mathcal{E}[X^{(k)}\rho^{(k)}{u^{(k)}}^2]-\alpha_k\rho_ku_k^2\\
\mathcal{E}\left[X^{(k)}u^{(k)}\left(\rho^{(k)}E^{(k)} + p^{(k)}\right)\right]-\alpha_ku_k(\rho_k E_k + p_k)
\end{bmatrix}
}\\
&\qquad\qquad\quad =: \alpha_k\textbf{F}_k \qquad\qquad\qquad\qquad\qquad\qquad  =: \textbf{F}_{k}^0
\end{split}
\end{equation}
where $\textbf{F}_k^0$ denotes the kinetic fluctuation of momentum and energy, that will be neglected in the following.

\subsection{The DEM for Eulerian biphasic flow}

Using the notation introduced in the previous section, the DEM method applies to the discrete setting: we consider a computational mesh $(x_i)_{i=1,\ldots, M}\subset \mathbb{R}$ and the associated control volume $\mathcal{C}_i = \left[x_{i-\frac{1}{2}}, x_{i+\frac{1}{2}}\right]$.\\
According to the definition of Lagrangian Fluxes, one needs to identify/be able to compute the speed of the interface separating different components. This translates at the numerical level to the necessity of considering Riemann Solvers able to compute a contact-discontinuity $\sigma$. Given two initial states $\textbf{U}_L,\textbf{U}_R$ we assume the solution of a Riemann Problem with possibly different phases at each side of the discontinuity to generate three waves (shocks or rarefactions separated by a contact discontinuity), in complete analogy to the single-phase theory.\\
Given $\textbf{U}_L, \textbf{U}_R \in \mathbb{R}^m$, the speed of the contact-discontinuity/material interface is denoted by $\sigma_{LR} := \sigma(\textbf{U}_L,\textbf{U}_R)$, while $F(\textbf{U}_L,\textbf{U}_R)$ and $U(\textbf{U}_L,\textbf{U}_R)$ denote the numerical flux and the numerical solution generated by solving the Riemann Problem with initial states $\textbf{U}_L,\textbf{U}_R$. The concrete forms of the numerical operators $F, U$ depend on the Riemann Solver under consideration, for which popular choices are the HLLC or the Roe Riemann Solvers \cite[Chapter 10-11]{ToroRS}.\\
At each time step the preliminary stages of the method proceed as follows: at each time level $t = t^n$, we have 
\begin{enumerate}
\item Subdivide randomly the computational cell $x_{i-\frac{1}{2}}=\xi_0 < \xi_1 < \ldots < \xi_{N(\omega)} = x_{i+\frac{1}{2}}$, where $\omega$ aims at indexing the specific realization of $X^{(k)}$.
\item Assign randomly in each subcell $[\xi_j,\xi_{j+1}]$ the phases $\Sigma_1$ or $\Sigma_2$ with the state $\textbf{U}^{(1)}$ or $\textbf{U}^{(2)}$. Up to merging adjacent subcells affected by the same phase, we have that within a volume two adjacent subcells contain different phases. We denote the interface velocity originating at the subnode $\xi_j$ as $\sigma_j$, see Fig. \ref{Fig:schematic}.

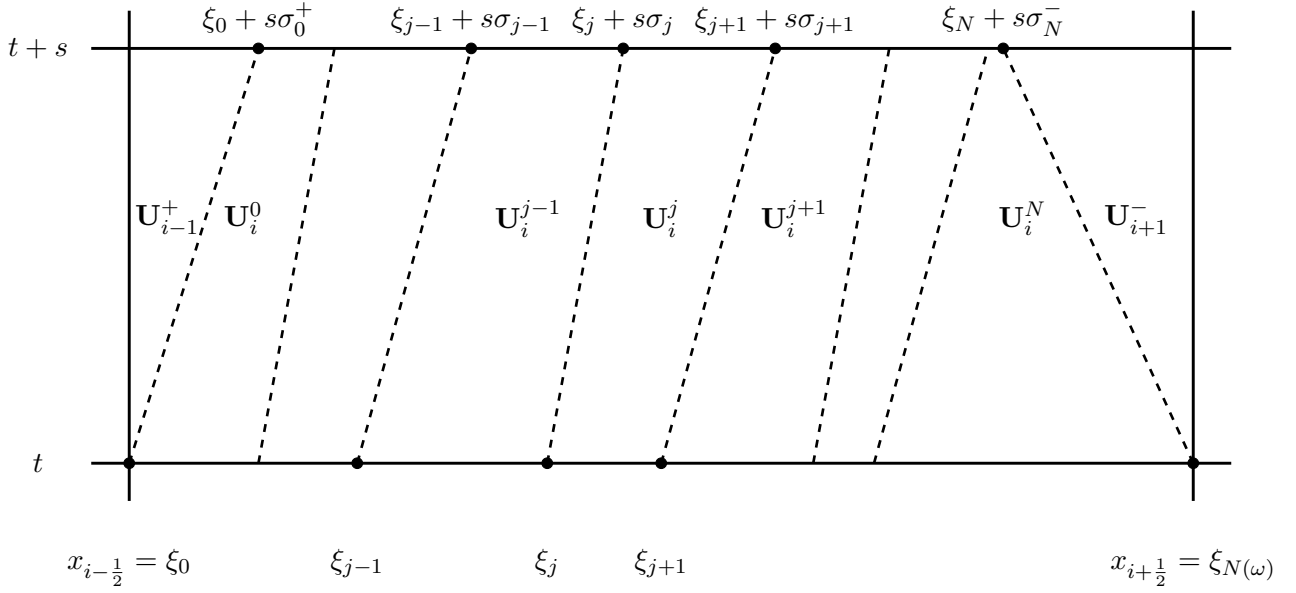
\begin{figure}
	\begin{center}
		\begin{tikzpicture}
		
		\pgfmathsetmacro{\xL}{-0.5};
		\pgfmathsetmacro{\yL}{-0.5};
		\pgfmathsetmacro{\xR}{14.5};
		\pgfmathsetmacro{\yR}{6.0};
		\pgfmathsetmacro{\marginBox}{0.5};
			
		\draw[line width = 1.1pt] (\xL, \yL + \marginBox) -- (\xR, \yL+\marginBox);
		\draw[line width = 1.1pt] (\xL +\marginBox,\yL) -- (\xL + \marginBox,\yR);
		\draw[line width = 1.1pt] (\xR -\marginBox,\yL) -- (\xR - \marginBox,\yR);
		\draw[line width = 1.1pt] (\xL,\yR - \marginBox) -- (\xR,\yR - \marginBox);
		
		\node at (\xL -0.7, \yL+\marginBox) {$t$};
		\node at (\xL -0.7, \yR-\marginBox) {$t+s$};
		
		\tkzDefPoint(\xL+\marginBox,\yL+\marginBox){x0};
		\tkzLabelPoint[below, yshift=-1cm](x0){$x_{i-\frac{1}{2}} = \xi_0$}		
		\tkzDefPoint(\xR - \marginBox,\yL+\marginBox){xN};
		\tkzLabelPoint[below, yshift=-1cm](xN){$x_{i+\frac{1}{2}} = \xi_{N(\omega)}$};
			
		\tkzDefPoint(1.7,\yR - \marginBox){y0};
		\tkzLabelPoint[above, yshift = 1pt](y0){$\xi_0 + s \sigma^+_0$};
		\tkzDefPoint(11.5,\yR - \marginBox){yN};
		\tkzLabelPoint[above, yshift = 1pt](yN){$\xi_N + s \sigma^-_N$};		
		
		\draw[line width = 1pt, dashed] (x0) -- (y0);
				
		\draw[line width = 1pt, dashed] (1.7,0) -- (2.7,\yR-\marginBox);
		
		\tkzDefPoint(3,\yL + \marginBox){xjm};
		\tkzLabelPoint[below, yshift = -1cm](xjm){$\xi_{j-1}$};
		\tkzDefPoint(4.5,\yR - \marginBox){yjm};
		\tkzLabelPoint[above, yshift = 1pt](yjm){$\xi_{j-1} + s \sigma_{j-1}$};			
		\draw[line width = 1pt, dashed] (xjm) -- (yjm);
		
		\tkzDefPoint(5.5,\yL + \marginBox){xj};
		\tkzLabelPoint[below, yshift = -1cm](xj){$\xi_{j}$};
		\tkzDefPoint(6.5,\yR - \marginBox){yj};
		\tkzLabelPoint[above, yshift = 1pt](yj){$\xi_{j} + s \sigma_{j}$};			
		\draw[line width = 1pt, dashed] (xj) -- (yj);
		
		\tkzDefPoint(7,\yL + \marginBox){xjp};
		\tkzLabelPoint[below, yshift = -1cm](xjp){$\xi_{j+1}$};
		\tkzDefPoint(8.5,\yR - \marginBox){yjp};
		\tkzLabelPoint[above, yshift = 1pt](yjp){$\xi_{j+1} + s \sigma_{j+1}$};			
		\draw[line width = 1pt, dashed] (xjp) -- (yjp);
			
		\draw[line width = 1pt, dashed] (9,0) -- (10, \yR-\marginBox);
		\draw[line width = 1pt, dashed] (9.8,0) -- (11.3, \yR-\marginBox);

		\draw[line width = 1pt, dashed] (xN) -- (yN);		
		
		\foreach \n in {y0,yjm,yj,yjp,yN, x0,xjm,xj,xjp,xN} { 
			\filldraw[black] (\n) circle (2pt); 
		}		
		
		\pgfmathsetmacro{\ym}{0.5*\yL + 0.5*\yR + \marginBox};		
		
		\node at (\xL+2*\marginBox, \ym ) {$\mathbf{U}^{+}_{i-1}$};
		\node at (\xL+4*\marginBox, \ym ) {$\mathbf{U}^{0}_{i}$};
		\node at (\xL+11.5*\marginBox, \ym ) {$\mathbf{U}^{j-1}_{i}$};
		\node at (\xL+15*\marginBox, \ym ) {$\mathbf{U}^{j}_{i}$};
		\node at (\xL+18.5*\marginBox, \ym ) {$\mathbf{U}^{j+1}_{i}$};
		\node at (\xL+24.5*\marginBox, \ym ) {$\mathbf{U}^{N}_{i}$};
		\node at (\xL+27.5*\marginBox, \ym ) {$\mathbf{U}^{-}_{i+1}$};
		\end{tikzpicture}
	\end{center}
	\caption{Schematic representation of the prototypical generation of interfaces in the control volume $\mathcal{C}_i\times [t,t+s]$.}\label{Fig:schematic}
\end{figure}

Notice that the evolution of phase $k\in\lbrace 1, 2 \rbrace$
\begin{equation}
\int_t^{t+s}\int_{\mathcal{C}_i} X^{(k)}\left(\partial_t \textbf{U}^{(k)} +\partial_x \textbf{F}^{(k)}\right)\, dxdr=0
\end{equation}
can be written as
\begin{equation}
\label{scheme_exact}
\begin{split}
&\frac{1}{\Delta x} \int_{x_{i-\frac{1}{2}}}^{x_{i+\frac{1}{2}}} (X^{(k)}\textbf{U}^{(k)})(x,t+s)\,dx - \frac{1}{\Delta x}\int_{x_{i-\frac{1}{2}}}^{x_{i+\frac{1}{2}}} (X^{(k)}\textbf{U}^{(k)})(x,t)\,dx\\
& + \frac{1}{\Delta x}\left(\int_t^{t+s}(X^{(k)}\textbf{F}^{(k)})(x_{i+\frac{1}{2}},r)\,dr- \int_t^{t+s}(X^{(k)}\textbf{F}^{(k)})(x_{i-\frac{1}{2}},r)\,dr\right)\\
& -\int_t^{t+s} \left(F^{lag}_0\partial_x X^{(k)}\right) (x_{i-\frac{1}{2}}+(r-t)\sigma_0^+,r)\, dr-\int_t^{t+s} \left(F^{lag}_{N(\omega)}\partial_x X^{(k)}\right)(x_{i+\frac{1}{2}} + (r-t)\sigma^-_{N(\omega)},r)\,dr\\
& - \frac{1}{\Delta x}\sum_{j=1}^{N(\omega)-1} \int_t^{t+s} \left(F^{lag}_{j}\partial_x X^{(k)}\right)(\xi_{j}+(r-t)\sigma_{j},r)\, dr=0
\end{split}
\end{equation}
where the Lagrangian fluxes $F^{lag}_{j}:=\textbf{F}^{(k)}_{I_j}-\sigma_j \textbf{U}^{(k)}_{I_j}$ are evaluated at the only side affected by phase $k$ of the interface moving with velocity $\sigma_j$.
\item Obtain a semi-discrete approximation of the realization according to a Godunov type scheme: we approximate the flux integrals and the Lagrangian flux integrals by means of a Godunov type scheme 
\begin{equation}\label{approximation}
\begin{split}
&X^{(k)}\textbf{F}^{(k)} (x_{i+\frac{1}{2}},r) \approx X^{(k)}(x_{i+\frac{1}{2}},t^n)F(U^{n}_i,U^{n}_{i+1})\\
&F^{lag}_j\partial_x X^{(k)} (\xi_{j}+(r-t)\sigma_{j},r) \approx \left[X^{(k)}\right]_j\left(F(U^{j}_i,U^{j+1}_i) -\sigma(U^{j}_i,U^{j+1}_i) U(U^{j}_i,U^{j+1}_i)\right)
\end{split}
\end{equation}
for any $r \in [t^n, t^{n}+s]$. The notation $\left[X^{(k)}\right]_j$ stands for the jump across the $j$-th interface moving with velocity $\sigma_j$. Under the above assumptions and upon division by $s$ in (\ref{scheme_exact}), the scheme reads
\begin{equation}
\begin{split}
&\frac{d}{dt}\left(\frac{1}{\Delta x} \int_{x_{i-\frac{1}{2}}}^{x_{i+\frac{1}{2}}} (X^{(k)}\textbf{U}^{(k)})(x,t)\,dx\right)\\
&\qquad + \frac{1}{\Delta x}\left[X^{(k)}(x_{i+\frac{1}{2}},t^n)F(U^{n}_{i-1},U^{n}_i) - X^{(k)}(x_{i-\frac{1}{2}},t^n)F(U^{n}_{i},U^{n}_{i+1})\right]=\\
&\qquad +\frac{1}{\Delta x}\sum_{j=1}^{N(\omega)-1} F^{lag}(U^{j-1}_i,U^j_i)\left[X^{(k)}\right]_j\\
&\qquad +\frac{1}{\Delta x}\Big(F^{lag}(U^+_{i-1},U^0_i)\left[X^{(k)}\right]_0+F^{lag}(U^{N(\omega)-1}_i,U^-_{i+1})\left[X^{(k)}\right]_{N(\omega)}\Big).
\end{split}
\end{equation}
Due to the alternate character of the distribution of data in the interior of the volume $\mathcal{C}_i$, one obtains the following relations: let us define the number of interior interfaces $N_{int} = N -1 \geq 0$, then
\begin{enumerate}
\item \underline{$N_{int}$ is even} : one can rearrange the summation as to arrive to (see Table  \ref{Tab:istances})
$$
\sum_{j=1}^{N(\omega)-1} F^{lag}(U^{j-1}_i,U^j_i)\left[X^{(k)}\right]_j 
= 
\frac{N_{int}}{2}\left(
F^{lag}(U^{(l)}_i,U^{(k)}_i)-F^{lag}(U^{(k)}_i,U^{(l)}_i)
\right)
$$
\item \underline{$N_{int}$ is odd} : then, $N_{int}-1$ is even (if $N_{int} > 0$), thus
\begin{equation*}
\begin{split}
\sum_{j=1}^{N(\omega)-1} 
&
F^{lag}(U^{j-1}_i,U^j_i)\left[X^{(k)}\right]_j 
= 
\frac{N_{int}-1}{2}\left(
F^{lag}(U^{(l)}_i,U^{(k)}_i)-F^{lag}(U^{(k)}_i,U^{(l)}_i)
\right)\\
&+ \chi_{\lbrace X^{(k)}(x_{i+\frac{1}{2}}^{-},t) = 1 \rbrace} F^{lag}\left(U^{(l)}_i,U^{(k)}_i\right) - \chi_{\lbrace X^{(k)}(x_{i+\frac{1}{2}}^{-},t) = 0 \rbrace} F^{lag}\left(U^{(k)}_i,U^{(l)}_i\right) 
\end{split}
\end{equation*}
where the characteristic function $\chi$ over the even $\lbrace Y=1\rbrace$ is defined as
\begin{equation*}
\chi_{\lbrace Y = 1 \rbrace} = \begin{cases}
1 & \textit{ if }\, Y = 1\\
0 & \textit{otherwise}
\end{cases}.
\end{equation*}

\end{enumerate}

Hence, by putting together the two instances that may occur, one ends up with
\begin{equation*}
\begin{split}
\sum_{j=1}^{N(\omega)-1} 
&
F^{lag}(U^{j-1}_i,U^j_i)\left[X^{(k)}\right]_j 
= 
\frac{N_{int}}{2}\left(
F^{lag}(U^{(l)}_i,U^{(k)}_i)-F^{lag}(U^{(k)}_i,U^{(l)}_i)
\right) + \theta^{(k)}(\omega)
\end{split}
\end{equation*}
where the perturbation variable $\theta^{(k)}$ is defined as
\begin{equation*}
\begin{split}
\theta^{(k)}(\omega) &:= 
\chi_{\lbrace N_{int} \textit{ odd } \rbrace}
\Bigg[
\chi_{\lbrace X^{(k)}(x_{i+\frac{1}{2}}^{-},t) = 1 \rbrace} F^{lag}\left(U^{(l)}_i,U^{(k)}_i\right)
- \chi_{\lbrace X^{(k)}(x_{i+\frac{1}{2}}^{-},t) = 0 \rbrace} F^{lag}\left(U^{(k)}_i,U^{(l)}_i\right)\\
& 
\qquad\qquad
-\frac{1}{2}\left(F^{lag}(U^{(l)}_i,U^{(k)}_i)-F^{lag}(U^{(k)}_i,U^{(l)}_i)\right)
\Bigg]
\end{split}
\end{equation*}
We thus assume that $\mathcal{E}\left[\theta^{(k)}\right] = 0$ for each $k$, thus implying that the perturbation with respect to the first term $\frac{N_{int}}{2}\left(
F^{lag}(U^{(l)}_i,U^{(k)}_i)-F^{lag}(U^{(k)}_i,U^{(l)}_i)
\right)$ generated by an odd number of internal contributions is negligible in mean. Such an assumption is clearly not verified for a low number of interfaces.\\
Under such assumption, we end up with
\begin{equation}
\frac{1}{\Delta x}\sum_{j=1}^{N(\omega)-1} F^{lag}(U^{j-1}_i,U^j_i)\left[X^{(k)}\right]_j \approx \frac{N_{int}(\omega)}{2\Delta x}\left[F^{lag}(U^{(l)}_i,U^{(k)}_i)-F^{lag}(U^{(k)}_i,U^{(l)}_i)\right].
\end{equation}
So the semi discrete scheme reads
\begin{equation}
\label{semi-discrete}
\begin{split}
&\frac{d}{dt}\left(\frac{1}{\Delta x} \int_{x_{i-\frac{1}{2}}}^{x_{i+\frac{1}{2}}} (X^{(k)}\textbf{U}^{(k)})(x,t)\,dx\right) + \frac{1}{\Delta x}\left[X^{(k)}(x_{i+\frac{1}{2}},t)F(U^{n}_{i-1},U^{n}_i) - X^{(k)}(x_{i-\frac{1}{2}},t)F(U^{n}_{i},U^{n}_{i+1})\right]=\\
&\qquad +\frac{N_{int}(\omega)}{2\Delta x}
\left(
F^{lag}(U^{(l)}_i,U^{(k)}_i)-F^{lag}(U^{(k)}_i,U^{(l)}_i)
\right)
\\
&\qquad+\frac{1}{\Delta x}\Big(F^{lag}(U^+_{i-1},U^0_i)\left[X^{(k)}\right]_0+F^{lag}(U^{N(\omega)-1}_i,U^-_{i+1})\left[X^{(k)}\right]_{N(\omega)}\Big).
\end{split}
\end{equation}

\begin{table}
\begin{tabular}{c||c|c|c|c|c|c|c}
\textbf{Cases}&
\multicolumn{3}{c|}{\textbf{Cell Phase}} &
\multicolumn{2}{|c|}{\textbf{Jumps}} &
\multicolumn{2}{|c}{\textbf{Lagrangian Fluxes}}\\
\hline
& $\mathbf{[\xi_{j-1},\xi_{j}]}$ & $\mathbf{[\xi_{j},\xi_{j+1}]}$ & $\mathbf{[\xi_{j+1},\xi_{j+2}]}$ & $\mathbf{[X^{(k)}]_j}$ & $\mathbf{[X^{(k)}]_{j+1}}$ & $\mathbf{F^{lag}(U^{j-1}_i,U^{j}_i)}$ & $\mathbf{F^{lag}(U^{j}_i,U^{j+1}_i)}$ \\
\hline
\hline
$\mathbf{1}$ & $\Sigma_l$&$\Sigma_k$&$\Sigma_l$ & $1$ & $-1$ & $F^{lag}(U^{(l)}_i,U^{(k)}_i)$ & $F^{lag}(U^{(k)}_i,U^{(l)}_i)$ \\
\hline
$\mathbf{2}$ & $\Sigma_k$&$\Sigma_l$&$\Sigma_k$ & $-1$ & $1$ & $F^{lag}(U^{(k)}_i,U^{(l)}_i)$ & $F^{lag}(U^{(l)}_i,U^{(k)}_i)$ \\
\hline
\end{tabular}
\caption{Possible configuration for the subcell $\left[\xi_j,\xi_{j+1}\right]$ and relative jumps across discontinuity, as well as Lagrangian fluxes. Integers $k\neq l\in \lbrace 1,2\rbrace$ denote phase indexes. }\label{Tab:istances}
\end{table}
\item Ensemble average of all realizations: taking ensemble average in (\ref{semi-discrete}) and with reference to the notation (\ref{EA_vars}), we obtain,
\begin{equation}
\label{ensamble_semi-discrete}
\begin{split}
&\frac{d}{dt}\left(\alpha_k\textbf{U}_k\right)_i + \frac{1}{\Delta x}\Bigg[\mathcal{E}\left[X^{(k)}(x_{i+\frac{1}{2}},t)F(U^{n}_{i-1},U^{n}_i) \right]-\mathcal{E}\left[ X^{(k)}(x_{i-\frac{1}{2}},t)F(U^{n}_{i},U^{n}_{i+1})\right]\Bigg]=\\
&\qquad +\mathcal{E}\left[\frac{N_{int}(\omega)}{2\Delta x}\right]\left( F^{lag}(U^{(l)}_i,U^{(k)}_i)-F^{lag}(U^{(k)}_i,U^{(l)}_i)\right)\\
&\qquad+\frac{1}{\Delta x}\Bigg(\mathcal{E}\left[ F^{lag}(U^+_{i-1},U^0_i)\left[X^{(k)}\right]_0\right]+\mathcal{E}\left[ F^{lag}(U^{N(\omega)-1}_i,U^-_{i+1})\left[X^{(k)}\right]_{N(\omega)}\right]\Bigg).
\end{split}
\end{equation}
\end{enumerate}

\section{The one-parameter mesoscopic scheme}
\label{sec:onePar}

\noindent
In order to be of practical use, the scheme (\ref{ensamble_semi-discrete}) requires the specification of four different terms:
\begin{itemize}
\item[$\bullet$] $\mathcal{E}\left[\frac{N_{int}(\omega)}{2\Delta x}\right]$ : the average number of internal components of the dispersed phase in cell $\mathcal{C}_i$;
\item[$\bullet$] $\mathcal{E}\left[X^{(k)}(x_{i+\frac{1}{2}},t^n)F(U^{n}_{i-1},U^{n}_i) \right]$: the conservative numerical flux;
\item[$\bullet$] $\mathcal{E}\left[ F^{lag}(U^+_{i-1},U^0_i)\left[X^{(k)}\right]_0\right]$ the left non-conservative term;
\item[$\bullet$] $\mathcal{E}\left[ F^{lag}(U^{N(\omega)-1}_i,U^-_{i+1})\left[X^{(k)}\right]_{N(\omega)}\right]$ : the right non-conservative term.
\end{itemize}
Building upon the work of Abgrall and Saurel \cite{Abgrall&Saurel}, the aforementioned ensemble averages can be simplified by noticing that the random variable $X^{(k)}$ is in fact discrete, and its average can be written as the sum of all the instances multiplied by their probability of occurrence. In the following we will make use of the following notation
\begin{equation*}
\begin{split}
\mathcal{P}_{i+\frac{1}{2}}\left[\Sigma_p,\Sigma_p\right] &:= \mathcal{P}_{i+\frac{1}{2}}\left[\lbrace X^{(p)}(x_{i+\frac{1}{2}}^+,t^n) = 1, X^{(p)}(x_{i+\frac{1}{2}}^-,t^n) = 1   \rbrace\right]\\
\mathcal{P}_{i+\frac{1}{2}}\left[\Sigma_p,\Sigma_q\right] &:= \mathcal{P}_{i+\frac{1}{2}}\left[\lbrace X^{(p)}(x_{i+\frac{1}{2}}^+,t^n) = 1, X^{(q)}(x_{i+\frac{1}{2}}^-,t^n) = 0  \rbrace\right]
\end{split}
\end{equation*}
for each phase index $p\neq q \in \lbrace 1, 2\rbrace$, with the notation $X^{(p)}(x_{i+\frac{1}{2}}^\pm,t^n) = \lim_{x\rightarrow x_{i+\frac{1}{2}}^\pm} X^{(p)}(x,t^n)$, for a prescribed time level $t=t^n$. Notice that, these probabilities are defined in terms of different characteristic functions $X^{(p)}$. Nevertheless, fixing the phase $k\neq l \in \lbrace 1,2 \rbrace$, one can equivalently rewrite these latter probabilities in terms of \emph{one} characteristic function
\begin{equation}
\label{Prob_def}
\begin{split}
\mathcal{P}_{i+\frac{1}{2}}\left[\Sigma_k,\Sigma_k\right] &= \mathcal{P}_{i+\frac{1}{2}}\left[\lbrace X^{(k)}(x_{i+\frac{1}{2}}^+,t^n) = 1, X^{(k)}(x_{i+\frac{1}{2}}^-,t^n) = 1   \rbrace\right]\\
\mathcal{P}_{i+\frac{1}{2}}\left[\Sigma_k,\Sigma_l\right] &= \mathcal{P}_{i+\frac{1}{2}}\left[\lbrace X^{(k)}(x_{i+\frac{1}{2}}^+,t^n) = 1, X^{(l)}(x_{i+\frac{1}{2}}^-,t^n) = 0  \rbrace\right]\\
\mathcal{P}_{i+\frac{1}{2}}\left[\Sigma_l,\Sigma_k\right] &= \mathcal{P}_{i+\frac{1}{2}}\left[\lbrace X^{(k)}(x_{i+\frac{1}{2}}^+,t^n) = 0, X^{(k)}(x_{i+\frac{1}{2}}^-,t^n) = 1  \rbrace\right]\\
\mathcal{P}_{i+\frac{1}{2}}\left[\Sigma_l,\Sigma_l\right] &= \mathcal{P}_{i+\frac{1}{2}}\left[\lbrace X^{(k)}(x_{i+\frac{1}{2}}^+,t^n) = 0, X^{(k)}(x_{i+\frac{1}{2}}^-,t^n) = 0  \rbrace\right]\\
\end{split}
\end{equation}
Moreover, we define the flux indicator function
\begin{equation}
\label{beta}
\beta_{i+\frac{1}{2}}^{(p,q)} := \mathrm{sign}\left(\sigma\left(U_{i}^{(p)}, U_{i+1}^{(l)}\right)\right) = \begin{cases}
1 & \textit{if}\quad\sigma\left(U_{i}^{(p)}, U_{i+1}^{(l)}\right) \geq 0\\
-1 & \textit{if}\quad\sigma\left(U_{i}^{(p)}, U_{i+1}^{(l)}\right) \leq 0
\end{cases}
\end{equation}
and the notation $a^+ := \max(a,0)$, $a^- := \min(a,0)$.\\ 
Estimation of three of the above quantities is accomplished as follows:
\begin{enumerate}
\item[$\bullet$] \underline{Conservative Terms}: We require that the Godunov state $U^*_{i+\frac{1}{2}}(0)$ \cite{ToroRS} (i.e. the solution of the Riemann Problem at the right cell interface at time $t=0$) belongs to the phase $k$ or not - see Fig. \ref{Fig:RP}. Hence,
\begin{equation}
\label{Cons_terms_Est}
\begin{split}
\mathcal{E}&\left[X^{(k)}(x_{i+\frac{1}{2}},t^n)F(U^{n}_{i-1},U^{n}_i) \right] = \mathcal{P}_{i+\frac{1}{2}}\left[\Sigma_k,\Sigma_k\right]F\left(U^{(k)}_{i}, U^{(k)}_{i+1}\right)\\
&\qquad + \left(\beta_{i+\frac{1}{2}}^{(k,l)}\right)^+ \mathcal{P}_{i+\frac{1}{2}}\left[\Sigma_k,\Sigma_l\right] F\left(U^{(k)}_{i},U^{(l)}_{i+1}\right) + \left(-\beta_{i+\frac{1}{2}}^{(l,k)}\right)^+ \mathcal{P}_{i+\frac{1}{2}}\left[\Sigma_l,\Sigma_k\right]F\left(U^{(l)}_{i},U^{(k)}_{i+1}\right)
\end{split}
\end{equation}

\begin{figure}
	\begin{center}
		\begin{tikzpicture}
		
		\pgfmathsetmacro{\xL}{-2.5};
		\pgfmathsetmacro{\len}{7};
		\pgfmathsetmacro{\xLR}{\xL + \len};
		\pgfmathsetmacro{\sp}{1};
		\pgfmathsetmacro{\xRL}{\xL + \len + \sp};
		\pgfmathsetmacro{\xRR}{\xRL + \len};		
		\pgfmathsetmacro{\hei}{5};
		\pgfmathsetmacro{\seg}{\hei-1};
			
		\tkzDefPoint(\xL,0){x0};
		\tkzDefPoint(\xLR, 0){x1};
		\tkzLabelPoint[below](x1){$x$}		
		\tkzDefPoint(\xRL,0){x2};
		\tkzDefPoint(\xRR, 0){x3};
		\tkzLabelPoint[below](x3){$x$}
		
		\tkzDefPoint(\xL + 0.5*\len, 0){xm1};
		\tkzLabelPoint[below](xm1){$x_{i+\frac{1}{2}}$};
		\tkzDefPoint(\xRL + 0.5*\len,0){xm2};
		\tkzLabelPoint[below](xm2){$x_{i+\frac{1}{2}}$};
		
		\tkzDefPoint(\xL + 0.5*\len, \hei){y1};
		\tkzLabelPoint[left](y1){$t$}
		\tkzDefPoint(\xL + \len + \sp + 0.5*\len,\hei){y2};
		\tkzLabelPoint[left](y2){$t$}
		
		\tkzDefPoint(\xL + 0.5*\len, \seg){ym1};
		\tkzDefPoint(\xRL+ 0.5*\len,\seg){ym2};
			
		\draw [-stealth] (x0) -- (x1);
		\draw [-stealth] (x2) -- (x3);		
		
		\draw [-stealth] (xm1) -- (y1);
		\draw [-stealth] (xm2) -- (y2);
		
		\draw [thick, color=red] (xm1) -- (ym1);
		\draw [thick, color=red] (xm2) -- (ym2);	
		
		
		\tkzDefPoint(\xL + 0.75*\len,\seg){xs1};
		\tkzLabelPoint[above](xs1){$\sigma_{i+\frac{1}{2}}$};		
		\tkzDefPoint(\xRL + 0.25*\len,\seg){xs2};
		\tkzLabelPoint[above](xs2){$\sigma_{i+\frac{1}{2}}$};
		
		\tkzDefPoint(\xL + 0.25*\len,\seg){xs1a};
		\tkzLabelPoint[above, color=red, font=\small](xs1a){$X^{(k)}(x_{i+\frac{1}{2}},t)\equiv 0$};
		\tkzDefPoint(\xRL + 0.75*\len,\seg){xs2a};
		\tkzLabelPoint[above, color=red, font=\small](xs2a){$X^{(k)}(x_{i+\frac{1}{2}},t)\equiv 1$};
			
		\draw [dashed] (xm1) -- (xs1);
		\draw [dashed] (xm2) -- (xs2);
		
		\draw [-stealth, color=red] (\xL + 0.5*\len - 0.1, 0.5*\seg) -- (xs1a);
		\draw [-stealth, color=red] (\xRL + 0.5*\len + 0.1, 0.5*\seg) -- (xs2a);
		
		\node at (\xL+0.25*\len, 0.5*\seg) {$\mathbf{U}^{(l)}_{i}$};
		\node at (\xL+0.75*\len, 0.5*\seg) {$\mathbf{U}^{(k)}_{i+1}$};
		\node at (\xRL+0.25*\len, 0.5*\seg) {$\mathbf{U}^{(l)}_{i}$};
		\node at (\xRL+0.75*\len, 0.5*\seg) {$\mathbf{U}^{(k)}_{i+1}$};
		\end{tikzpicture}
	\end{center}
	\caption{Schematic representation of the Godunov state $U_{i+\frac{1}{2}}(t)\vert_{t=0}$ at the right cell interface.}\label{Fig:RP}
\end{figure}
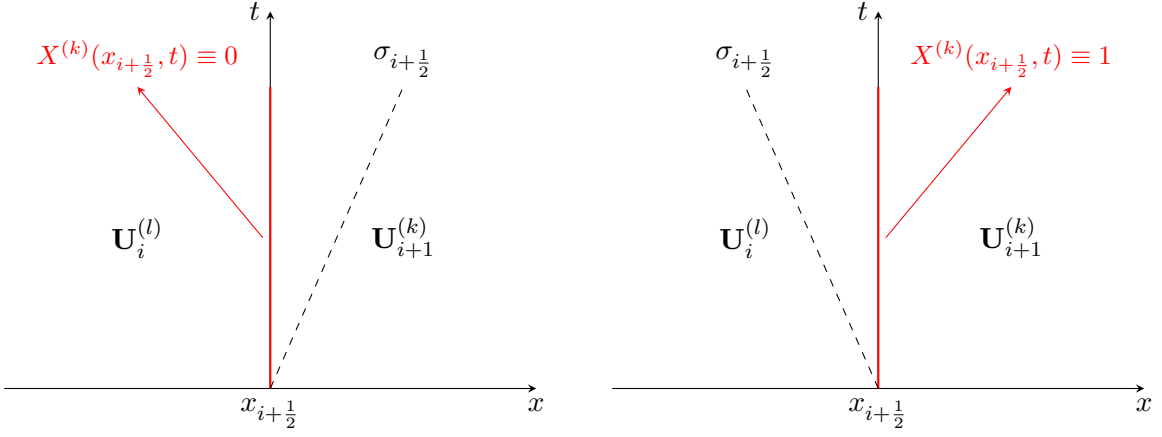

\item[$\bullet$] \underline{Right Non-Conservative Terms}: We need to make sure that a Lagrangian flux exists at the right interface, i.e. inflow is occurring.
\begin{equation}
\label{RNC_terms_Est}
\begin{split}
\mathcal{E}&\left[ F^{lag}(U^{N(\omega)-1}_i,U^-_{i+1})\left[X^{(k)}\right]_{N(\omega)}\right] = \left(-\beta_{i+\frac{1}{2}}^{(l,k)}\right)^+ \mathcal{P}_{i+\frac{1}{2}}\left[\Sigma_l, \Sigma_k\right] F^{lag}\left(U^{(l)}_{i},U^{(k)}_{i+1}\right)\\
&\qquad
-\left(-\beta_{i+\frac{1}{2}}^{(k,l)}\right)^+ \mathcal{P}_{i+\frac{1}{2}}\left[\Sigma_k, \Sigma_l\right] F^{lag}\left(U^{(k)}_{i},U^{(l)}_{i+1}\right)
\end{split}
\end{equation}
\item[$\bullet$] \underline{Left Non-Conservative Terms}: We need to make sure that a Lagrangian flux exists at the left interface, i.e. inflow is occurring.
\begin{equation}
\label{LNC_terms_Est}
\begin{split}
\mathcal{E}&\left[ F^{lag}(U^+_{i-1},U^0_i)\left[X^{(k)}\right]_0\right]= \left(\beta_{i-\frac{1}{2}}^{(l,k)}\right)^+ \mathcal{P}_{i-\frac{1}{2}}\left[\Sigma_l, \Sigma_k\right] F^{lag}\left(U^{(l)}_{i-1},U^{(k)}_{i}\right)\\
&\qquad
-\left(\beta_{i-\frac{1}{2}}^{(k,l)}\right)^+ \mathcal{P}_{i-\frac{1}{2}}\left[\Sigma_k, \Sigma_l\right] F^{lag}\left(U^{(k)}_{i-1},U^{(l)}_{i}\right)
\end{split}
\end{equation}
\end{enumerate}

Using formulas (\ref{Cons_terms_Est})-(\ref{LNC_terms_Est}), the final scheme reads
\begin{equation}
\label{ensamble_semi-discrete_prob}
\begin{split}
&\frac{d}{dt}\left(\alpha_k\textbf{U}_k\right)_i + \frac{1}{\Delta x}\Bigg[ \mathcal{P}_{i+\frac{1}{2}}\left[\Sigma_k,\Sigma_k\right]F\left(U^{(k)}_{i}, U^{(k)}_{i+1}\right) + \left(\beta_{i+\frac{1}{2}}^{(k,l)}\right)^+ \mathcal{P}_{i+\frac{1}{2}}\left[\Sigma_k,\Sigma_l\right] F\left(U^{(k)}_{i},U^{(l)}_{i+1}\right)\\
&\qquad + \left(-\beta_{i+\frac{1}{2}}^{(l,k)}\right)^+ \mathcal{P}_{i+\frac{1}{2}}\left[\Sigma_l,\Sigma_k\right]F\left(U^{(l)}_{i},U^{(k)}_{i+1}\right) -  \mathcal{P}_{i-\frac{1}{2}}\left[\Sigma_k,\Sigma_k\right]F\left(U^{(k)}_{i-1}, U^{(k)}_{i}\right)\\
&\qquad - \left(\beta_{i-\frac{1}{2}}^{(k,l)}\right)^+ \mathcal{P}_{i-\frac{1}{2}}\left[\Sigma_k,\Sigma_l\right] F\left(U^{(k)}_{i-1},U^{(l)}_{i}\right) - \left(-\beta_{i-\frac{1}{2}}^{(l,k)}\right)^+ \mathcal{P}_{i-\frac{1}{2}}\left[\Sigma_l,\Sigma_k\right]F\left(U^{(l)}_{i-1},U^{(k)}_{i}\right)\Bigg]=\\
&\qquad +\mathcal{E}\left[\frac{N_{int}(\omega)}{\Delta x}\right]\left( F^{lag}(U^{(l)}_i,U^{(k)}_i)-F^{lag}(U^{(k)}_i,U^{(l)}_i)\right)\\
&\qquad+\frac{1}{\Delta x}\Bigg(\left(\beta_{i-\frac{1}{2}}^{(l,k)}\right)^+ \mathcal{P}_{i-\frac{1}{2}}\left[\Sigma_l, \Sigma_k\right] F^{lag}\left(U^{(l)}_{i-1},U^{(k)}_{i}\right)\\
&\qquad
-\left(\beta_{i-\frac{1}{2}}^{(k,l)}\right)^+ \mathcal{P}_{i-\frac{1}{2}}\left[\Sigma_k, \Sigma_l\right] F^{lag}\left(U^{(k)}_{i-1},U^{(l)}_{i}\right) + \left(-\beta_{i+\frac{1}{2}}^{(l,k)}\right)^+ \mathcal{P}_{i+\frac{1}{2}}\left[\Sigma_l, \Sigma_k\right] F^{lag}\left(U^{(l)}_{i},U^{(k)}_{i+1}\right)\\
&\qquad
-\left(-\beta_{i+\frac{1}{2}}^{(k,l)}\right)^+ \mathcal{P}_{i+\frac{1}{2}}\left[\Sigma_k, \Sigma_l\right] F^{lag}\left(U^{(k)}_{i},U^{(l)}_{i+1}\right)\Bigg).
\end{split}
\end{equation}

Notice that the topological equation for the volume fraction is then recovered from (\ref{ensamble_semi-discrete_prob}) by formally choosing $F\equiv 0$ and $\textbf{U}_k\equiv 1$, so that the Lagrangian flux reduces to $F^{lag} = 0 -\sigma \cdot 1 = -\sigma$.\\ Hence, the numerical scheme is then of practical use, once the probability coefficients are defined.\\
\textbf{ }\\
\noindent
Originally such probability coefficients were given by means of an ansatz, leading to a limited model, even if thermodynamically consistent \cite{Saurel2017,Saurel2018}. We are going to close the model by proving convexity of such probability coefficients. The following proposition summarizes the properties each probability coefficient has to verify \cite{Abgrall&Saurel}.

\begin{proposition}
Let $\mathfrak{P}_{i+\frac{1}{2}}=\Bigg(\mathcal{P}_{i+\frac{1}{2}}\left[\Sigma_p, \Sigma_p\right], \mathcal{P}_{i+\frac{1}{2}}\left[\Sigma_p, \Sigma_q\right]\Bigg)$ be a pair of probabilities coefficients defined in  (\ref{Prob_def}). 
Assume that
\begin{subequations}\label{ConsistencyCondsA}
\begin{align}
&\mathcal{P}_{i+\frac{1}{2}}\left[\Sigma_p, \Sigma_p\right] + \mathcal{P}_{i+\frac{1}{2}}\left[\Sigma_p, \Sigma_q\right] = \alpha_{i}^p\label{Condi}\\
&\mathcal{P}_{i+\frac{1}{2}}\left[\Sigma_p, \Sigma_p\right] + \mathcal{P}_{i+\frac{1}{2}}\left[\Sigma_q, \Sigma_p\right] = \alpha_{i+1}^p\label{Condi+1}
\end{align}
\end{subequations}
Then the following consistency conditions must hold: for each $p\neq q \in\lbrace 1, 2 \rbrace$ 
\begin{subequations}\label{ConsistencyConds}
\begin{align}
&\mathcal{P}_{i+\frac{1}{2}}\left[\Sigma_p, \Sigma_p\right]\leq \min\left(\alpha_{i}^p,\alpha_{i+1}^p\right)\label{Condmin}\\
&\mathcal{P}_{i+\frac{1}{2}}\left[\Sigma_p, \Sigma_q\right]\geq \max\left(\alpha_{i}^p-\alpha_{i+1}^p,0\right)\label{Condmax}
\end{align}
\end{subequations}
where the two neighbouring volume fractions verify the $\emph{saturation condition}$
\begin{equation}
\label{vol}
\alpha^p_j \in [0,1]\qquad \textit{ and } \qquad \alpha^p_j + \alpha^q_j = 1 \qquad\forall j\in\lbrace i, i+1\rbrace.
\end{equation}
\end{proposition}

A probability pair $\mathfrak{P}$ will be termed a \emph{consistent} probability pair if it verify (\ref{ConsistencyCondsA})-(\ref{ConsistencyConds}), under the assumption that (\ref{vol}) holds.

\begin{remark}\label{Rem:Strat}
Notice that the pair consisting of the two bounds in (\ref{ConsistencyConds})
$$\mathfrak{P}^0_{i+\frac{1}{2}} := \Bigg(\mathcal{P}^0_{i+\frac{1}{2}}\left[\Sigma_p,\Sigma_p\right],\mathcal{P}^0_{i+\frac{1}{2}}\left[\Sigma_p,\Sigma_q\right]\Bigg) := \Bigg(\min\left(\alpha^p_{i},\alpha^p_{i+1}\right), \max\left(\alpha^p_{i}-\alpha^p_{i+1},0\right) \Bigg)$$
is itself a consistent probability pair.
\end{remark}

\noindent
Due to this remark, Abgrall and Saurel proposed the following approximation for the probability coefficients
\begin{equation}
\label{AS_prob}
\begin{split}
\mathcal{P}_{i+\frac{1}{2}}\left[\Sigma_k, \Sigma_k\right] \approx  \min(\alpha_{i}^k,\alpha_{i+1}^k),\qquad \mathcal{P}_{i+\frac{1}{2}}\left[\Sigma_k, \Sigma_l\right] \approx \max(\alpha_{i}^k-\alpha_{i+1}^k,0)\\
\mathcal{P}_{i+\frac{1}{2}}\left[\Sigma_l, \Sigma_l\right] \approx  \min(\alpha_{i}^l,\alpha_{i+1}^l),\qquad \mathcal{P}_{i+\frac{1}{2}}\left[\Sigma_l, \Sigma_k\right] \approx \max(\alpha_{i}^l-\alpha_{i+1}^l,0)
\end{split}
\end{equation}
\noindent
An interesting features of this choice is that it has both mathematical and physical implications. First, from the mathematical point of view, it can be shown that fixing the probability coefficients $\mathcal{P}_{i+\frac{1}{2}}\left[\Sigma_p,\Sigma_p\right] = \mathcal{P}^0_{i+\frac{1}{2}}\left[\Sigma_p,\Sigma_p\right]$ then, as to form a consistent probability pair, we have no other choice but $\mathcal{P}_{i+\frac{1}{2}}\left[\Sigma_p,\Sigma_q\right]=\mathcal{P}^0_{i+\frac{1}{2}}\left[\Sigma_p,\Sigma_q\right]$. The viceversa also holds. Furthermore, the pair $\mathfrak{P}_{i+\frac{1}{2}}^0$ constitutes an upper-lower bound for any pair of probability coefficients $\mathfrak{P}_{i+\frac{1}{2}}$, respectively, according to (\ref{Condi}) and (\ref{Condi+1}). Thus, such probability pair is an extreme point in the space of consistent probability pairs.

On the other hand, there is an interesting example that helps understanding the physical implication of choosing $\mathfrak{P}_{i+\frac{1}{2}} = \mathfrak{P}^0_{i+\frac{1}{2}}$:
consider a tube filled with two different fluids one surrounded by the other with no dispersion of one phase into the complementary one. 
We will term this physical regime as \emph{stratified flow}.
Let us consider if the DEM scheme with $\mathfrak{P}_{i+\frac{1}{2}}=\mathfrak{P}^0_{i+\frac{1}{2}}$ yields reasonable approximations of such flow regime, see Fig. \ref{Fig:Strat}.
First notice that each entry in the probability pair $\mathfrak{P}^0_{i+\frac{1}{2}}$ is not zero, that is, it is not zero the coefficient of each flux of the type $F(\Sigma_m,\Sigma_m)$ and $F(\Sigma_m, \Sigma_n)$ with $m\neq n\in\lbrace 1,2\rbrace$ in (\ref{ensamble_semi-discrete}).\\

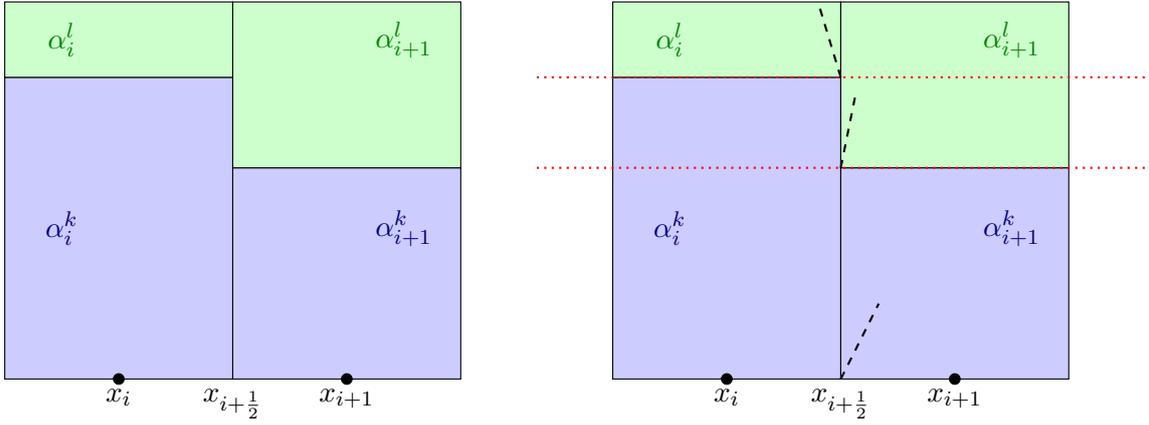
\begin{figure}[h!]
\begin{center}
		\begin{tikzpicture}
		
		\pgfmathsetmacro{\Dx}{3};
		\pgfmathsetmacro{\hei}{4};		
		
		\pgfmathsetmacro{\xL}{-5};
		\pgfmathsetmacro{\xR}{\xL + \Dx};
		
		\pgfmathsetmacro{\yL}{0};
		\pgfmathsetmacro{\yR}{0.7*\hei};
		
		\pgfmathsetmacro{\tp}{1};
		
		\tkzDefPoint(\xL + 0.5*\Dx, 0){xi};
		\tkzLabelPoint[below](xi){$x_{i}$};
		\tkzDefPoint(\xL + \Dx, 0){xih};
		\tkzLabelPoint[below](xih){$x_{i+\frac{1}{2}}$};
		\tkzDefPoint(\xL + 1.5*\Dx, 0){xip};
		\tkzLabelPoint[below](xip){$x_{i+1}$};	
				
		\draw [fill=blue!20] (\xL,\yL) rectangle (\xR,\hei);
		\draw [fill=blue!20] (\xR,\yL) rectangle (\xR+\Dx,\yR);
		\draw [fill=green!20] (\xL,\hei) rectangle (\xR,\hei + \tp);
		\draw [fill=green!20] (\xR,\yR) rectangle (\xR+\Dx,\tp + \hei);
		
		\node [color=green!50!black] at (\xL+0.25*\Dx, \hei + 0.5*\tp) {$\alpha^l_{i}$};
		\node [color=blue!50!black] at (\xL+0.25*\Dx, 0.5*\hei) {$\alpha^k_{i}$};
		\node [color=green!50!black] at (\xR+0.75*\Dx, \hei + 0.5*\tp) {$\alpha^l_{i+1}$};
		\node [color=blue!50!black] at (\xR+0.75*\Dx, 0.5*\hei) {$\alpha^k_{i+1}$};	
		
		\foreach \n in {xi,xip}{
			\filldraw[black] (\n) circle (2pt); 
		}	
		
		\pgfmathsetmacro{\sp}{2};
		\pgfmathsetmacro{\xLR}{\xR + \Dx + \sp};
		\pgfmathsetmacro{\xRR}{\xLR + \Dx};
	
		\tkzDefPoint(\xLR + 0.5*\Dx, 0){xiR};
		\tkzLabelPoint[below](xiR){$x_{i}$};
		\tkzDefPoint(\xLR + \Dx, 0){xihR};
		\tkzLabelPoint[below](xihR){$x_{i+\frac{1}{2}}$};
		\tkzDefPoint(\xLR + 1.5*\Dx, 0){xipR};
		\tkzLabelPoint[below](xipR){$x_{i+1}$};

		\draw [fill=blue!20] (\xLR,\yL) rectangle (\xRR,\hei);
		\draw [fill=blue!20] (\xRR,\yL) rectangle (\xRR+\Dx,\yR);
		\draw [fill=green!20] (\xLR,\hei) rectangle (\xRR,\hei + \tp);
		\draw [fill=green!20] (\xRR,\yR) rectangle (\xRR+\Dx,\tp + \hei);
		
		\node [color=green!50!black] at (\xLR+0.25*\Dx, \hei + 0.5*\tp) {$\alpha^l_{i}$};
		\node [color=blue!50!black] at (\xLR+0.25*\Dx, 0.5*\hei) {$\alpha^k_{i}$};
		\node [color=green!50!black] at (\xRR+0.75*\Dx, \hei + 0.5*\tp) {$\alpha^l_{i+1}$};
		\node [color=blue!50!black] at (\xRR+0.75*\Dx, 0.5*\hei) {$\alpha^k_{i+1}$};	
		
		\foreach \n in {xiR,xipR}{
			\filldraw[black] (\n) circle (2pt); 
		}	
		
		\draw [dotted, thick, color=red] (\xLR - 1, \yR) -- (\xRR + \Dx + 1, \yR);
		\draw [dotted, thick, color=red] (\xLR - 1, \hei) -- (\xRR + \Dx + 1, \hei);	
		
		\draw [dashed, thick] (\xLR+\Dx, \yR) -- (\xLR+\Dx + 0.2, \yR + \tp);
		\draw [dashed, thick] (\xLR+\Dx, 0) -- (\xLR+\Dx + 0.5, \tp);
		\draw [dashed, thick] (\xLR+\Dx, \hei) -- (\xLR+\Dx - 0.3, \hei + \tp);
		
		\end{tikzpicture}
	\end{center}
	\caption{Schematic representation of a stratified flow at the numerical level.}\label{Fig:Strat}
\end{figure}

By computing the corresponding probability pairs, one can convince oneself that these acts as flux-weights in (\ref{ensamble_semi-discrete}), corresponding to the area of the surface through which a specific flux is applied.\\
Unfortunately, the same computations would also be carried out in the case of disconnected phases at the interface, see Fig. \ref{Fig:Disp}.
This second case will be called \emph{dispersed flow}, where the phase that does not share a segment of the cell interface is called the dispersed phase.
In such a case, each probability coefficient $\mathfrak{P}_{i+\frac{1}{2}}=\mathfrak{P}^0_{i+\frac{1}{2}}$ would still be non vanishing, thus introducing in the computation a non-zero numerical flux for the disperse phase, even if the regime is discontinuous. At the physical level, for the dispersed phase, this is equivalent to saying that a sound wave propagated in the dispersed phase in cell $i$ gets propagated into the corresponding phase of cell $i+1$ even if no material connection between the phases exists. 

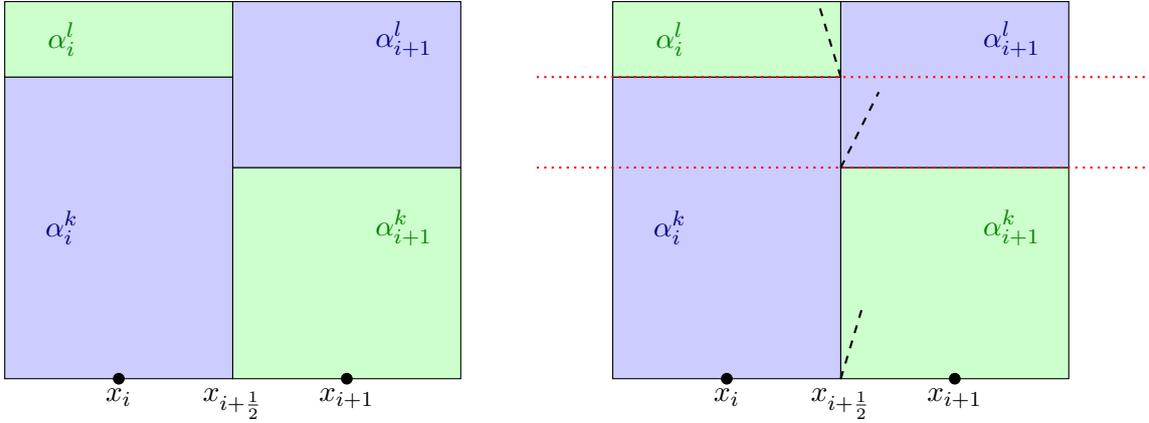
\begin{figure}[h!]
\begin{center}
		\begin{tikzpicture}
		
		\pgfmathsetmacro{\Dx}{3};
		\pgfmathsetmacro{\hei}{4};		
		
		\pgfmathsetmacro{\xL}{-5};
		\pgfmathsetmacro{\xR}{\xL + \Dx};
		
		\pgfmathsetmacro{\yL}{0};
		\pgfmathsetmacro{\yR}{0.7*\hei};
		
		\pgfmathsetmacro{\tp}{1};
		
		\tkzDefPoint(\xL + 0.5*\Dx, 0){xi};
		\tkzLabelPoint[below](xi){$x_{i}$};
		\tkzDefPoint(\xL + \Dx, 0){xih};
		\tkzLabelPoint[below](xih){$x_{i+\frac{1}{2}}$};
		\tkzDefPoint(\xL + 1.5*\Dx, 0){xip};
		\tkzLabelPoint[below](xip){$x_{i+1}$};	
				
		\draw [fill=blue!20] (\xL,\yL) rectangle (\xR,\hei);
		\draw [fill=green!20] (\xR,\yL) rectangle (\xR+\Dx,\yR);
		\draw [fill=green!20] (\xL,\hei) rectangle (\xR,\hei + \tp);
		\draw [fill=blue!20] (\xR,\yR) rectangle (\xR+\Dx,\tp + \hei);
		
		\node [color=green!50!black] at (\xL+0.25*\Dx, \hei + 0.5*\tp) {$\alpha^l_{i}$};
		\node [color=blue!50!black] at (\xL+0.25*\Dx, 0.5*\hei) {$\alpha^k_{i}$};
		\node [color=blue!50!black] at (\xR+0.75*\Dx, \hei + 0.5*\tp) {$\alpha^l_{i+1}$};
		\node [color=green!50!black] at (\xR+0.75*\Dx, 0.5*\hei) {$\alpha^k_{i+1}$};	
		
		\foreach \n in {xi,xip}{
			\filldraw[black] (\n) circle (2pt); 
		}	
		
		\pgfmathsetmacro{\sp}{2};
		\pgfmathsetmacro{\xLR}{\xR + \Dx + \sp};
		\pgfmathsetmacro{\xRR}{\xLR + \Dx};
	
		\tkzDefPoint(\xLR + 0.5*\Dx, 0){xiR};
		\tkzLabelPoint[below](xiR){$x_{i}$};
		\tkzDefPoint(\xLR + \Dx, 0){xihR};
		\tkzLabelPoint[below](xihR){$x_{i+\frac{1}{2}}$};
		\tkzDefPoint(\xLR + 1.5*\Dx, 0){xipR};
		\tkzLabelPoint[below](xipR){$x_{i+1}$};

		\draw [fill=blue!20] (\xLR,\yL) rectangle (\xRR,\hei);
		\draw [fill=green!20] (\xRR,\yL) rectangle (\xRR+\Dx,\yR);
		\draw [fill=green!20] (\xLR,\hei) rectangle (\xRR,\hei + \tp);
		\draw [fill=blue!20] (\xRR,\yR) rectangle (\xRR+\Dx,\tp + \hei);
		
		\node [color=green!50!black] at (\xLR+0.25*\Dx, \hei + 0.5*\tp) {$\alpha^l_{i}$};
		\node [color=blue!50!black] at (\xLR+0.25*\Dx, 0.5*\hei) {$\alpha^k_{i}$};
		\node [color=blue!50!black] at (\xRR+0.75*\Dx, \hei + 0.5*\tp) {$\alpha^l_{i+1}$};
		\node [color=green!50!black] at (\xRR+0.75*\Dx, 0.5*\hei) {$\alpha^k_{i+1}$};	
		
		\foreach \n in {xiR,xipR}{
			\filldraw[black] (\n) circle (2pt); 
		}	
		
		\draw [dotted, thick, color=red] (\xLR - 1, \yR) -- (\xRR + \Dx + 1, \yR);
		\draw [dotted, thick, color=red] (\xLR - 1, \hei) -- (\xRR + \Dx + 1, \hei);	
		
		\draw [dashed, thick] (\xLR+\Dx, 0) -- (\xLR+\Dx + 0.3, \tp);		
		\draw [dashed, thick] (\xLR+\Dx, \yR) -- (\xLR+\Dx + 0.5, \yR + \tp);
		\draw [dashed, thick] (\xLR+\Dx, \hei) -- (\xLR+\Dx - 0.3, \hei + \tp);
		\end{tikzpicture}
	\end{center}
\caption{Schematic representation of a numerical dispersed flow of phase $k$ into phase $l$: the DEM method would predict $\mathcal{P}_{i+\frac{1}{2}}\left[\Sigma_k,\Sigma_k\right] = \alpha_i^k > 0$ even if there is no continuous flow at cell interface.}\label{Fig:Disp}
\end{figure}

These considerations, motivated us to investigate the structure of such probability coefficients: the following proposition identify the complementary lower-upper bounds.

\begin{proposition}\label{Prop:NewProbs}
Let $\mathfrak{P}_{i+\frac{1}{2}}=\Bigg(\mathcal{P}_{i+\frac{1}{2}}\left[\Sigma_p, \Sigma_p\right], \mathcal{P}_{i+\frac{1}{2}}\left[\Sigma_p, \Sigma_q\right]\Bigg)$ be a pair of probability coefficients defined in (\ref{Prob_def}). Then for every $p\neq q \in\lbrace 1,2 \rbrace$ it holds
\begin{subequations}\label{NewProbs}
\begin{align}
&\mathcal{P}_{i+\frac{1}{2}}\left[\Sigma_p, \Sigma_q\right]\leq \min\left(\alpha_{i}^p,\alpha_{i+1}^q\right)\label{Condmin2}\\
&\mathcal{P}_{i+\frac{1}{2}}\left[\Sigma_p, \Sigma_p\right]\geq \max\left(\alpha_{i}^p-\alpha_{i+1}^q,0\right)\label{Condmax2}
\end{align}
\end{subequations}
under the assumption that (\ref{ConsistencyCondsA}) holds. Moreover, the probably pair 
\[
\mathfrak{P}^1_{i+\frac{1}{2}}:=\Bigg(\max\left(\alpha_{i}^p-\alpha_{i+1}^q,0\right), \min\left(\alpha_{i}^p,\alpha_{i+1}^q\right) \Bigg)
\]
is a consistent probability pair, provided that the saturation condition (\ref{vol}) holds.
\end{proposition}

\begin{proof}
See Appendix \ref{appendix:Proof}.
\end{proof}

\begin{remark}\label{Rem:Disp}
Following the considerations made before Prop. \ref{Prop:NewProbs}, one can see that the probability pair $\mathfrak{P}^1_{i+\frac{1}{2}}$ is associated with a dispersed flow regime, see Fig. \ref{Fig:Disp}.
\end{remark}

It light of this final remark, it becomes not surprising the following theorem.

\begin{theorem}\label{Thm:Convex}
Let $\mathfrak{P}_{i+\frac{1}{2}}=\Bigg(\mathcal{P}_{i+\frac{1}{2}}\left[\Sigma_p, \Sigma_p\right], \mathcal{P}_{i+\frac{1}{2}}\left[\Sigma_p, \Sigma_q\right]\Bigg)$ be a pair of probability coefficients defined in (\ref{Prob_def}) with $p\neq q \in\lbrace 1,2\rbrace$. Then there exist a $r\in[0,1]$ depending of $(x_{i+\frac{1}{2}},t^n)$ but not on $p$, such that
\begin{subequations}
\label{eq:r_probs}
\begin{align}
&\mathcal{P}_{i+\frac{1}{2}}\left[\Sigma_p, \Sigma_p\right] = r \max\left(\alpha_{i}^p-\alpha_{i+1}^q,0\right) + (1-r)\min\left(\alpha_{i}^p,\alpha_{i+1}^p\right)\label{Convex-pp}\\
&\mathcal{P}_{i+\frac{1}{2}}\left[\Sigma_p, \Sigma_q\right]= r\min\left(\alpha_{i}^p,\alpha_{i+1}^q\right) + (1-r)\max\left(\alpha_{i}^p-\alpha_{i+1}^p,0\right)\label{Convex-pq}
\end{align}
\end{subequations}
or, succinctly,
\begin{equation*}
\mathfrak{P}_{i+\frac{1}{2}} = r\mathfrak{P}^1_{i+\frac{1}{2}} + (1-r)\mathfrak{P}^0_{i+\frac{1}{2}}
\end{equation*}
under the assumption that (\ref{ConsistencyCondsA}) and (\ref{vol}) hold.
\end{theorem}

\begin{proof}
See Appendix \ref{appendix:Proof}.
\end{proof}
\noindent
Previous result leads us to substitute the probabilities appearing in (\ref{ensamble_semi-discrete_prob}) with the right hand side of (\ref{eq:r_probs}), leading to a one-parameter semi-discrete scheme, modeling the mesoscopic description of the underlying two-phase flow. 

\section{Continuous limit and solution strategy}

By suppressing the dependency on the time variable $t$ for notation convenience, the global, one-parameter semi-discrete DEM scheme takes the form

\begin{equation}
\label{semi-discrete_form}
\frac{d \left(\alpha_k\textbf{U}_k\right)_i}{dt} + \frac{\mathcal{E}_{i+\frac{1}{2}}[X^{(k)}F]-\mathcal{E}_{i-\frac{1}{2}}[X^{(k)}F]}{\Delta x} = \frac{\mathcal{E}_{boundary}[F^{lag}]_i}{\Delta x} + \mathcal{E}_{relax}[F^{lag}]_i
\end{equation}
where
\begin{align*}
\mathcal{E}_{i+\frac{1}{2}}\left[X^{(k)}F\right]& := \mathcal{P}_{i+\frac{1}{2}}\left[\Sigma_k,\Sigma_k\right]F\left(U^{(k)}_{i}, U^{(k)}_{i+1}\right) + \left(\beta_{i+\frac{1}{2}}^{(k,l)}\right)^+ \mathcal{P}_{i+\frac{1}{2}}\left[\Sigma_k,\Sigma_l\right] F\left(U^{(k)}_{i},U^{(l)}_{i+1}\right)\\
&\qquad  + \left(-\beta_{i+\frac{1}{2}}^{(l,k)}\right)^+ \mathcal{P}_{i+\frac{1}{2}}\left[\Sigma_l,\Sigma_k\right]F\left(U^{(l)}_{i},U^{(k)}_{i+1}\right)\\
\mathcal{E}_{boundary}\left[F^{lag}\right]_i& := \left(\beta_{i-\frac{1}{2}}^{(l,k)}\right)^+ \mathcal{P}_{i-\frac{1}{2}}\left[\Sigma_l, \Sigma_k\right] F^{lag}\left(U^{(l)}_{i-1},U^{(k)}_{i}\right) -\left(\beta_{i-\frac{1}{2}}^{(k,l)}\right)^+ \mathcal{P}_{i-\frac{1}{2}}\left[\Sigma_k, \Sigma_l\right] F^{lag}\left(U^{(k)}_{i-1},U^{(l)}_{i}\right) \\
&+ \left(-\beta_{i+\frac{1}{2}}^{(l,k)}\right)^+ \mathcal{P}_{i+\frac{1}{2}}\left[\Sigma_l, \Sigma_k\right] F^{lag}\left(U^{(l)}_{i},U^{(k)}_{i+1}\right) -\left(-\beta_{i+\frac{1}{2}}^{(k,l)}\right)^+ \mathcal{P}_{i+\frac{1}{2}}\left[\Sigma_k, \Sigma_l\right] F^{lag}\left(U^{(k)}_{i},U^{(l)}_{i+1}\right)\\
\mathcal{E}_{relax}[F^{lag}]_i& := \mathcal{E}\left[\frac{N_{int}(\omega)}{\Delta x}\right]\left( F^{lag}(U^{(l)}_i,U^{(k)}_i)-F^{lag}(U^{(k)}_i,U^{(l)}_i)\right)\\
\mathcal{P}_{i+\frac{1}{2}}\left[\Sigma_p, \Sigma_p\right] &:= r_{i+\frac{1}{2}} \max\left(\alpha_{i}^p-\alpha_{i+1}^q,0\right) + (1-r_{i+\frac{1}{2}})\min\left(\alpha_{i}^p,\alpha_{i+1}^p\right)\\
\mathcal{P}_{i+\frac{1}{2}}\left[\Sigma_p, \Sigma_q\right] &:= r_{i+\frac{1}{2}}\min\left(\alpha_{i}^p,\alpha_{i+1}^q\right) + (1-r_{i+\frac{1}{2}})\max\left(\alpha_{i}^p-\alpha_{i+1}^p,0\right)
\end{align*}

\subsection{Continuous Limit}
\noindent
Due to the substantial disagreement in the scientific community about the governing equations which regulate multiphase phenomena, many authors have tried to derive such mathematical models in different ways. One of the advantages of taking the perspective of the DEM method, is the possibility to derive it, starting from a local description.
A first example in this direction was performed in \cite{Saurel2003}, for the specific choice of $r\equiv 0$. Such a model can be summarized into the following system of PDEs: each phase $k\neq l \in\lbrace 1,2\rbrace$ evolves according to
\begin{equation}
\label{ConLim}
\begin{split}
\partial_t& \alpha_k + u_I\partial_x \alpha_k = \mu (p_{k}-p_{l})\\
\partial_t&(\alpha_k\rho_k) + \partial_x(\alpha_k\rho_k u_k)=0\\
\partial_t&(\alpha_k\rho_k u_k) + \partial_x\Big(\alpha_k(\rho_ku_k^2 + p_k)\Big) = p_I\partial_x\alpha_k - \lambda (u_{k}-u_{l}) \\
\partial_t&(\alpha_k\rho_kE_k) + \partial_x\Big(\alpha_k u_k(\rho_kE_k +p_k)\Big)=p_Iu_I\partial_x\alpha_k-\mu p_I^{'}(p_{k}-p_{l}) -\lambda u_I^{'}(u_{k}-u_{l})
\end{split}
\end{equation}
where the interfacial pressure $p_I$ and velocity $u_I$ are given by
\begin{equation}
\label{interfacials}
p_I := p_I^{'} + \mathrm{sign(\partial_x \alpha_k)} \frac{Z_kZ_l}{Z_k+Z_l}(u_l-u_k), \qquad u_I := u_I^{'} + \mathrm{sign(\partial_x \alpha_k)} \frac{1}{Z_k+Z_l}(u_l-u_k)
\end{equation}
where $Z_k := \rho_k a_k$ denotes the acoustic impedance and mean interfacial pressure $p_I^{'}$ and velocity $u_I^{'}$ read
\begin{equation}
\label{meaninterfacials1}
p_I^{'} := \frac{Z_k p_l + Z_l p_k}{Z_k + Z_l},\qquad u_I^{'} := \frac{Z_k u_k + Z_l u_l}{Z_k + Z_l}
\end{equation}
Relaxation parameters $\mu,\lambda$ are defined according to the interfacial area $A_I=\mathbb{E}\left[N_{int}(\omega)/\Delta x\right]$ via
\begin{equation}
\label{relax_pars}
\mu := \frac{A_I}{Z_1 + Z_2}\qquad \lambda :=Z_1Z_2\mu
\end{equation}

In Appendix \ref{appendix:ContLimit} we detail the procedure to derive the continuous limit of such scheme, as well as the specific assumptions. The resulting model for the description of (possibly) disperse flow of phase $k$ into $l$ reads 

\begin{equation}
\label{ConLim_OnePar}
\begin{split}
\partial_t& \alpha_k + (1-r)u_I\partial_x \alpha_k = -r \partial_x\left(\alpha_k u_I\right) +\mu (p_{k}-p_{l})\\
\partial_t&(\alpha_k\rho_k) + \partial_x(\alpha_k\rho_k u_k)=0\\
\partial_t&(\alpha_k\rho_k u_k) + \partial_x\Big(\alpha_k(\rho_ku_k^2 + p_k)\Big) = (1-r)p_I\partial_x\alpha_k + r\partial_x\left(\alpha_k p_I\right) - \lambda (u_{k}-u_{l}) \\
\partial_t&(\alpha_k\rho_kE_k) + \partial_x\Big(\alpha_k u_k(\rho_kE_k +p_k)\Big) = (1-r)p_Iu_I\partial_x\alpha_k - r\partial_x\left(p_Iu_I\alpha_k\right)\\
&\qquad\qquad\qquad\qquad\qquad\qquad\qquad\qquad\qquad\qquad
 -\mu p_I^{'}(p_{k}-p_{l}) -\lambda u_I^{'}(u_{k}-u_{l})
 \end{split}
 \end{equation}
\begin{remark}
An interesting fact concerning this limit is that it has, in the case $r=1$, a conservative character, which has already been established by other authors, see \cite{Marble, Saurel2017}, independently. 
Indeed, in the limit of small values of $\alpha_k$, one recovers the same model of \cite{Marble}.\\
In \cite{Saurel2017}, a similar model is proposed replacing the volume fraction equation making assumptions on the production rate of dispersed particles.
\end{remark}

\subsection{Solution Strategy}
\label{sec:NumStrat}

In this section we make some comments about the resulting scheme (\ref{semi-discrete_form}), its equilibrium variety, its numerical approximation and the use of relaxation procedure.
In order to simplify the notation and the following discussion, notice that (\ref{semi-discrete_form}) can be rewritten as

\begin{equation}\label{eq:7EM}
\frac{d}{d t}\left( \alpha_k\textbf{U}_k\right)_i + \frac{1}{\Delta x} G_i(\textbf{U}_i) = \lambda_i R(\textbf{U}_i)
\end{equation}
where $G(\textbf{U}_i) =\mathcal{E}_{i+\frac{1}{2}}[X^{(k)}F]-\mathcal{E}_{i-\frac{1}{2}}[X^{(k)}F]- \mathcal{E}_{boundary}[F^{lag}]_i$ is the numerical contribution coming from the application of the space-discretization operator applied to the states
\[
\textbf{U}_i = [(\alpha_k)_i,\left(\alpha_k\textbf{U}_k\right)_i,(\alpha_l)_i,\left(\alpha_l\textbf{U}_l\right)_i]^T,
\]
$\lambda_i = \mathcal{E}\left[N_i/\Delta x\right]$ is the average number of internal particles per cell and $R$ is the relaxation term arising from the presence of internal disperse particles.\\
Due to the assumption that the micro-scale is so rich that an infinite number of dispersed particles can be considered inside each cell (i.e. $\lambda_i\rightarrow \infty$), the system (\ref{eq:7EM}) is typically split into two step, namely the hyperbolic and the relaxation ones. This is also the strategy we follow in this work: the approximation of (\ref{semi-discrete_form}) is accomplished by the following operator splitting method

\begin{enumerate}
\item \underline{Hyperbolic Step}: The hyperbolic step stands for the evolution of the variables according to the left hand side of (\ref{eq:7EM}), namely
\begin{equation}\label{eq:HypStep}
\frac{d}{d t}\left( \alpha_k\textbf{U}_k\right)_i + \frac{1}{\Delta x} G_i(\textbf{U}_i) = 0.
\end{equation} 
\item \underline{Relaxation Step}: The relaxation step updates the approximation of the solution $\textbf{U}$, coming from the hyperbolic step, by computing the equilibrium state of the following ODE

\begin{equation}\label{eq:relaxODE}
\frac{d}{d t}\left( \alpha_k\textbf{U}_k\right)_i  = \lambda_i R(\textbf{U}_i)
\end{equation} 
as $\lambda_i\rightarrow\infty$. 
\end{enumerate}

Notice that (\ref{eq:HypStep})-(\ref{eq:relaxODE}) are intended also for the volume fraction $\alpha_k$ with the formal substitution $F \equiv 0$ and $U \equiv 1$.

We conclude this section by stating a convexity property of the numerical scheme (\ref{eq:HypStep}). 
In particular, we approximate the set of ODEs (\ref{eq:HypStep}) with a Forward Euler (FE) method, as it is usual in first-order numerical schemes. Hence, the update formula for (\ref{eq:HypStep}) reads

\begin{equation}
\label{eq:HypStep-FEapprox}
\left(\alpha_k\textbf{U}_k\right)^{n+1}_i = 
\left(\alpha_k\textbf{U}_k\right)^{n}_i - \frac{\Delta t}{\Delta x} G^n_i(\textbf{U}^{n}_i, r^n_{i-\frac{1}{2}}, r^n_{i+\frac{1}{2}})
\end{equation}
where we introduced explicitly the dependency of the scheme (\ref{semi-discrete_form}) with respect to the two parameters $r_{i+\frac{1}{2}}^n\approx r_{i+\frac{1}{2}}(t^n)$ and $r_{i-\frac{1}{2}}^n\approx r_{i-\frac{1}{2}}(t^n)$.

\begin{proposition}\label{prop:r}
Let $U^{k,n+1}_i(r)$ denote the numerical approximation resulting from (\ref{eq:HypStep-FEapprox}) when considering a constant value of the function $r=r(x,t)$, i.e. $U^{k,n+1}_i(r) := \left(\alpha_k \textbf{U}_k\right)^{n+1}_i(r,r)$. Assume also that each contribution of the numerical flux is positive.\\
Then, the numerical solution $(\alpha_k\textbf{U}_k)_i^{n+1}$ predicted by the scheme (\ref{eq:HypStep}) with the FE time-approximation lies between the corresponding numerical approximations generated by the choices $r\equiv 1$ and $r\equiv 0$, i.e.
\[
\min\lbrace 
U^{k,n+1}_i(0), U^{k,n+1}_i(1)
\rbrace
\leq
\left(\alpha_k\textbf{U}_k\right)^{n+1}_i
\leq 
\max\lbrace
U^{k,n+1}_i(0), U^{k,n+1}_i(1)
\rbrace
\]
\end{proposition}

\begin{proof}
Let us rewrite the scheme (\ref{semi-discrete_form}) by using the following form
\begin{equation}
\mathcal{E}_{i+\frac{1}{2}}\left[ X^{(k)} F \right] = 
\textbf{A}_{i+\frac{1}{2}}^{(k)} \mathcal{P}_{i+\frac{1}{2}}^{(k)}(r_{i+\frac{1}{2}}^n)
\quad
\mathcal{E}_{boundary}\left[F^{lag}\right]_i := \textbf{B}_{i-\frac{1}{2}}^{(k),+} \mathcal{P}_{i-\frac{1}{2}}^{(k)}(r_{i-\frac{1}{2}}^n) + \textbf{B}_{i+\frac{1}{2}}^{(k),-} \mathcal{P}_{i+\frac{1}{2}}^{(k)}(r_{i+\frac{1}{2}}^n)
\end{equation}
where the matrices $\textbf{A}^{(k)}_{i+\frac{1}{2}}\in\mathbb{R}^{3\times 3}$, $\textbf{B}^{(k),\pm}_{i+\frac{1}{2}}\in\mathbb{R}^{3\times 3}$ and the vector $\mathcal{P}^{(k)}_{i+\frac{1}{2}}\in\mathbb{R}^3$ are defined as
\begin{align*}
\textbf{A}_{i+\frac{1}{2}}^{(k)} &:=
\begin{bmatrix}
F(U^{(k)}_i,U^{(k)}_{i+1})^T & 
\left(\beta^{(k,l)}_{i+\frac{1}{2}}\right)^{+}
F(U^{(k)}_i,U^{(l)}_{i+1})^T &
\left(-\beta^{(l,k)}_{i+\frac{1}{2}}\right)^{+}
F(U^{(l)}_i,U^{(k)}_{i+1})^T
\end{bmatrix}\\
\textbf{B}_{i+\frac{1}{2}}^{(k),\pm} &:= 
\begin{bmatrix}
\mathbf{0}^T  &
-\left(\pm\beta^{(k,l)}_{i+\frac{1}{2}}\right)^{+} F^{lag}\left(U^{(k)}_{i},U^{(l)}_{i+1}\right)^T  &
\left(\pm\beta^{(l,k)}_{i+\frac{1}{2}}\right)^{+} F^{lag}\left(U^{(l)}_{i},U^{(k)}_{i+1}\right)^T
\end{bmatrix}\\
\mathcal{P}_{i+\frac{1}{2}}^{(k)} (r_{i+\frac{1}{2}}^n) &:=
\begin{bmatrix}
\mathcal{P}_{i+\frac{1}{2}}\left[\Sigma_k,\Sigma_k\right]\left(r_{i+\frac{1}{2}}^n\right) &
\mathcal{P}_{i+\frac{1}{2}}\left[\Sigma_k,\Sigma_l\right]\left(r_{i+\frac{1}{2}}^n\right) &
\mathcal{P}_{i+\frac{1}{2}}\left[\Sigma_l,\Sigma_k\right]\left(r_{i+\frac{1}{2}}^n\right)
\end{bmatrix}^T
\end{align*}
and $\mathcal{P}_{i+\frac{1}{2}}\left(r_{i+\frac{1}{2}}^n\right)\in\mathbb{R}$ is defined in (\ref{semi-discrete_form}).
Hence, (\ref{eq:HypStep-FEapprox}) can be reformulated into
\begin{align*}
(\alpha_k \textbf{U}_k)_i^{n+1}& (r_{i-\frac{1}{2}}^n, r^n_{i+\frac{1}{2}}) - U^{k,n}_i =\\
&
-\frac{\Delta t}{\Delta x}\left[ \textbf{A}_{i+\frac{1}{2}}^{(k)} \mathcal{P}_{i+\frac{1}{2}}^{(k)}\left(r_{i+\frac{1}{2}}^n\right) - \textbf{A}_{i-\frac{1}{2}}^{(k)} \mathcal{P}_{i-\frac{1}{2}}^{(k)}\left(r_{i-\frac{1}{2}}^n\right) + \textbf{B}_{i-\frac{1}{2}}^{(k),+} \mathcal{P}_{i-\frac{1}{2}}^{(k)}\left(r_{i-\frac{1}{2}}^n\right) + \textbf{B}_{i+\frac{1}{2}}^{(k),-} \mathcal{P}_{i+\frac{1}{2}}^{(k)}\left(r_{i+\frac{1}{2}}^n\right) \right]\\
&= 
\frac{\Delta t}{\Delta x}
\begin{bmatrix}
\textbf{A}_{i-\frac{1}{2}}^{(k)} - \textbf{B}_{i-\frac{1}{2}}^{(k),+} &
-\textbf{A}_{i+\frac{1}{2}}^{(k)} - \textbf{B}_{i+\frac{1}{2}}^{(k),-}
\end{bmatrix}
\begin{bmatrix}
\mathcal{P}_{i-\frac{1}{2}}^{(k)}\left(r_{i-\frac{1}{2}}^n\right)^T\\
\mathcal{P}_{i+\frac{1}{2}}^{(k)}\left(r_{i+\frac{1}{2}}^n\right)^T
\end{bmatrix}
\end{align*}
showing that $(\alpha_k \textbf{U}_k)_i^{n+1}(r_{i-\frac{1}{2}}^n, r^n_{i+\frac{1}{2}})$ is an affine transformation of the vector 
$\mathcal{P}_i^{(k)}:= [ 
\mathcal{P}_{i-\frac{1}{2}}^{(k)}\left(r_{i-\frac{1}{2}}^n\right), 
\mathcal{P}_{i+\frac{1}{2}}^{(k)}\left(r_{i+\frac{1}{2}}^n\right) 
]^T$.\\
By Theorem \ref{Thm:Convex}, each $\mathcal{P}_{i-j+\frac{1}{2}}^{(p)} (r^n_{i-j+\frac{1}{2}})$ is an affine transformation of $r^n_{i-j+\frac{1}{2}}$, so that
\[
\mathcal{P}_{i}^{(k)}
=
\begin{bmatrix}
\mathcal{P}_{i-\frac{1}{2}}^{(k)}\left(r_{i-\frac{1}{2}}^n\right) \\
\mathcal{P}_{i+\frac{1}{2}}^{(k)}\left(r_{i+\frac{1}{2}}^n\right)
\end{bmatrix}
\in
\left[
\mathcal{P}_{i}^{(k)}(0)
\mathcal{P}_{i}^{(k)}(1)
\right]^T
\]
where the latter relation is understood for each entry of the vector $\mathcal{P}_i^{(k)}$.
The thesis follows by positivity of the fluxes contributions.
\end{proof}

\begin{remark}
Notice that the above proposition guarantees a bound for the numerical approximation over the hyperbolic step. This is in principle not true for the two-stages scheme (\ref{eq:7EMeps}). 
Nevertheless, in the following numerical tests we do observe such behavior even though we were not able to prove the conclusion of Proposition \ref{prop:r} when including the relaxation step.
This may be due to some monotonicity property of the relaxation step, whose study is out of the scope of this paper. 
\end{remark}

\subsection{A comment about the relaxation step}

The relaxation step has attracted a lot of attention, due to its paramount importance for an accurate multi-scale description. 
Due to the discrete nature of the right-hand side in ({\ref{eq:relaxODE}), the system of ODEs one has to solve depend on the specific choice of the RS under use.
For example, when considering the exact RS (for the single-phase case) one would need to solve two RP (for each computational cell), typically via some root-finding procedure. This latter can become quite cumbersomeness, and a way to circumvent it \cite{Lallemand, Saurel2001,Saurel2007} is to simplify the system of ODEs (\ref{eq:relaxODE}) by substituting the RS with a fixed, simple approximate RS. Typically the acoustic solver \cite{ToroRS,Murrone} constitutes a reasonable and sufficiently simple choice.
After such a simplification step, one just derives the corresponding continuous limit of the right hand side (in terms of the variables $\textbf{U}_i$), and the system is solved by computing the equilibrium state as $\lambda_i\rightarrow\infty$.\\ 
Under the choice of the acoustic solver, one can show that the set of ODEs forces the mixture constituents to move with a single velocity and single pressure, as it was observed/theorized in many works, see \cite{Bdzil} and references therein. 
Unfortunately, there are several simplification steps in this procedure, which do not guarantee that any other reasonable solver leads to the same mechanical effects.
This is sometimes reformulated saying that \emph{the equilibrium variety (i.e. the set of states $\textbf{U}^\infty_i$ at the end of the relaxation step) depends on the choice of the RS}. A first investigation in this direction was performed in \cite{AbgrallPerrier} where the authors showed that for several solvers of common use this is not the case: for such solvers, the equilibrium variety turns out to be defined by the conditions 
\begin{equation}
\label{eq:EvarCond}
p^{(k),\infty}_i = p^{(l),\infty}_i =: p_i^\infty \qquad \qquad u^{(k),\infty}_i=u^{(l),\infty}_i =: S^\infty_i.
\end{equation}
where the index $\infty$ denotes the states at the end of the relaxation step.
Hence, one is tempted to conclude that the relaxation variety is invariant under the choice of (reasonable) solvers.\\
Here we consider the two following assumptions:

\begin{enumerate}
\item \underline{\textbf{Assumption on the Equilibrium Variety}}: We assume that the equilibrium variety defined by solving (\ref{eq:relaxODE}) and letting $\lambda_i\rightarrow\infty$, can be alternatively computed as the reduced set of variables which make $R$ vanish, that is, we assume that there exist a Maxwellian $M:$ $\textbf{u}\mapsto M(\textbf{u}) = \textbf{U}^\infty$ such that $R(\textbf{U}^\infty) \equiv 0$.
\item \underline{\textbf{Assumption on the Riemann Solver}}: We assume that the following flux-vector splitting condition holds
\begin{equation}\label{eq:RSassumption}
F^*\left(U_p,U_q\right) = u^*_{pq} U^*_{pq} + p^*_{pq} D^*_{pq}
\end{equation}
where $D^*_{pq} = [0,1,u^*_{pq}]^T$.
\end{enumerate}

\begin{proposition}
Under the assumptions $1$ and $2$, the states $\textbf{U}^{k,\infty}_i$ resulting from resolving the relaxation step (\ref{eq:relaxODE}) are defined by relations (\ref{eq:EvarCond}).
\end{proposition}

\begin{proof}
See Appendix \ref{appendix:variety}.
\end{proof}

\begin{remark}
Notice that assumption (\ref{eq:RSassumption}) is satisfied by many  popular Riemann solvers, including the exact, HLLC, and acoustic solvers.
\end{remark}

\begin{remark}
Notice that (\ref{eq:EvarCond}) \emph{does not} imply that any solver fulfilling the aforementioned assumptions will produce the same approximations for $S^\infty_i$ or $p_i^\infty$. Specifically, different solvers will produce different value for the interface velocities, in general. 
Thus, the form of the relaxation term is characterizing the equilibrium variety, but it yields no information on \emph{how to compute such values}.  
\end{remark}

\section{Numerical Experiments}
\label{sec:numex}

In this section we test the numerical algorithm to show the influence of the newly derived set of probabilities. Numerical fluxes have been computed using the HLLC flux for the Euler equations \cite{ToroRS,ToroHLLC} and Lagrangian fluxes have been computed according to
\begin{equation}\label{lag_HLLC}
F^{lag} = F_\mathrm{HLLC} - S^*_\mathrm{HLLC}U^*_\mathrm{HLLC}
\end{equation}
where $F_\mathrm{HLLC}$, $S^*_\mathrm{HLLC}$,$U^*_\mathrm{HLLC}$ denote the numerical flux, the speed of the contact discontinuity and the intermediate (star) value provided by the HLLC solver, see \cite{ToroRS} for details. Notice that the relation (\ref{lag_HLLC}) is crucial: indeed, one could be tempted to use the acoustic solver provided in \cite{Murrone, Saurel2003,ToroRS} to approximately compute the Lagrangian flux. However, this choice has been found to produce erroneous pressure oscillations near discontinuities, especially in absence of relaxation. Furthermore, we point out that such Riemann Solver for the Lagrangian Flux could also be interpreted to be non-positive conservative in the sense of \cite{Einfeldt1991}. For all our simulations we used a CFL constraint of $\mathrm{CFL} = 0.9$. Materials are governed by the stiffened gas equation of state
\begin{equation}
p_k = (\gamma_k-1)\rho_k e_k - \gamma_k \pi_k.
\end{equation}
The parameters of gas are $\gamma_1 = 1.4\,\mathrm{Pa}$, $\pi_1 = 0$, while for the liquid phase are $\gamma_2 = 4.4$ and $\pi_l = 6\times 10^8\,\mathrm{Pa}$. Each experiment is computed on the domain $D=[-1,1]$, unless differently stated.

\subsection{Uniform Volume Fraction}

The first numerical experiment consists of a shock-tube problem, with a uniform volume fraction. The initial mixtures consists of a strong pressure difference.
The initial condition in terms of the primitive variables $\textbf{V} = [\alpha,\rho,u,p]$ reads:
\[
\textbf{V}_0(x) = 
\begin{cases}
[\textbf{V}_L^{(1)}, \textbf{V}_L^{(2)}] & \textit{ if }\, x<0,\\
[\textbf{V}_R^{(1)}, \textbf{V}_R^{(2)}] & \textit{ if }\, x>0.
\end{cases},
\qquad
\textbf{V}_L^{(k)} = \begin{bmatrix}
0.5\\
\rho_k\\
0\\
p_L
\end{bmatrix},
\quad
\textbf{V}_R^{(k)} = \begin{bmatrix}
0.5\\
\rho_k\\
0\\
p_R
\end{bmatrix}
\quad
k\in\lbrace 1,2\rbrace
\]
where $\rho_1 = 50$, $\rho_2 = 1000$, $p_L = 10^9$, $p_L = 10^5$.

\subsubsection{The relaxation-free case}

We initially assume that $\lambda_i = 0$ for any $i$: solutions associated to a stratified flow evolve independently whereas in the dispersed regime, interactions between the fluids do occur. For the sake of comparison, we report the solution of such problem with $r=0$ in Fig.\ref{Fig:T1r0} (originally proposed in \cite{Abgrall&Saurel}) and the one associated to $r=1$ in Fig.\ref{Fig:T1r1}. The latter case is presented using several meshes to show convergence, whereas the case $r=0$ is compared to the single-phase exact solutions, to show independence of the two numerical simulations.

\begin{figure}[!htbp]
\centering
\includegraphics[scale = 0.47]{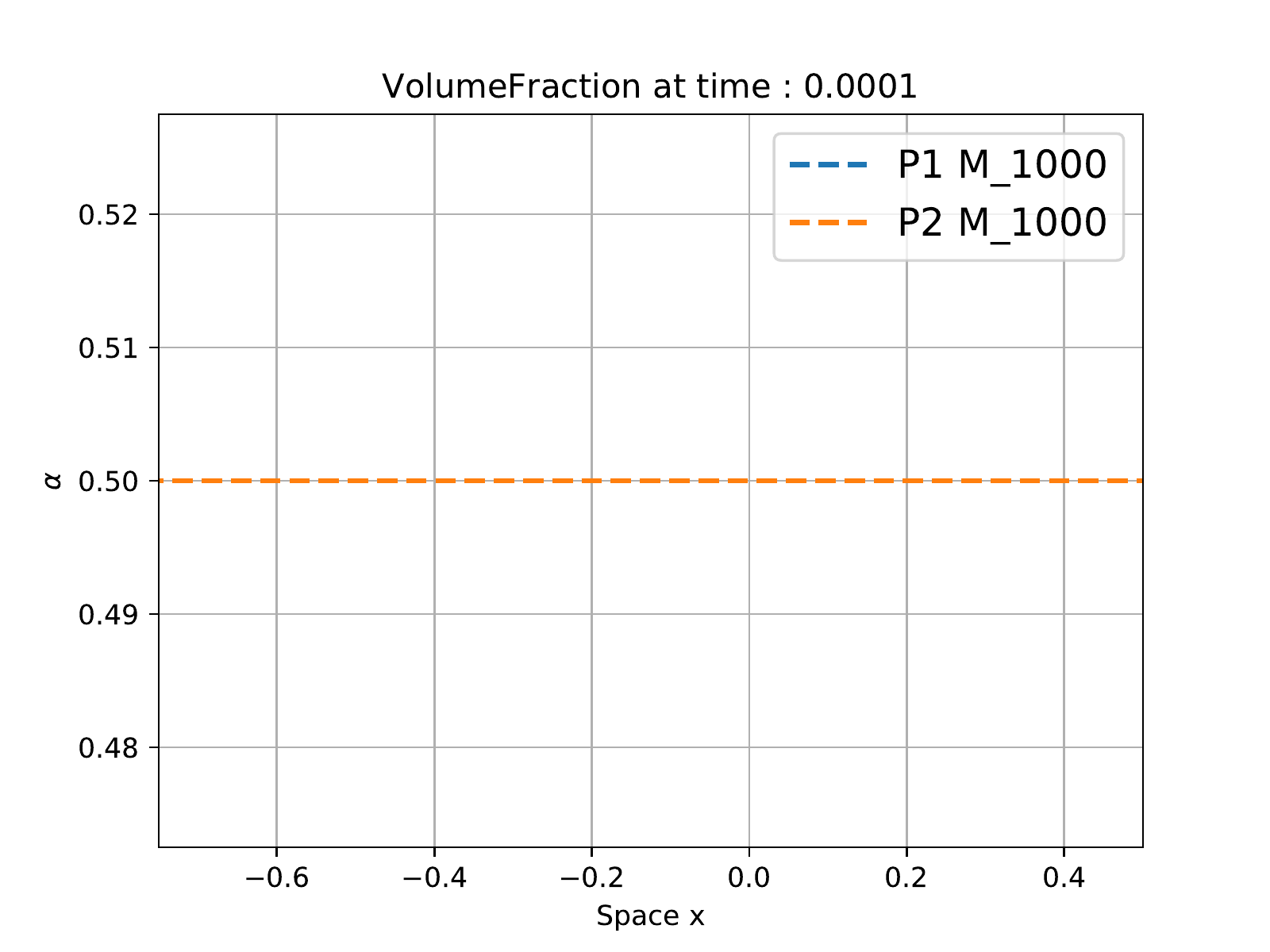}\,
\includegraphics[scale = 0.47]{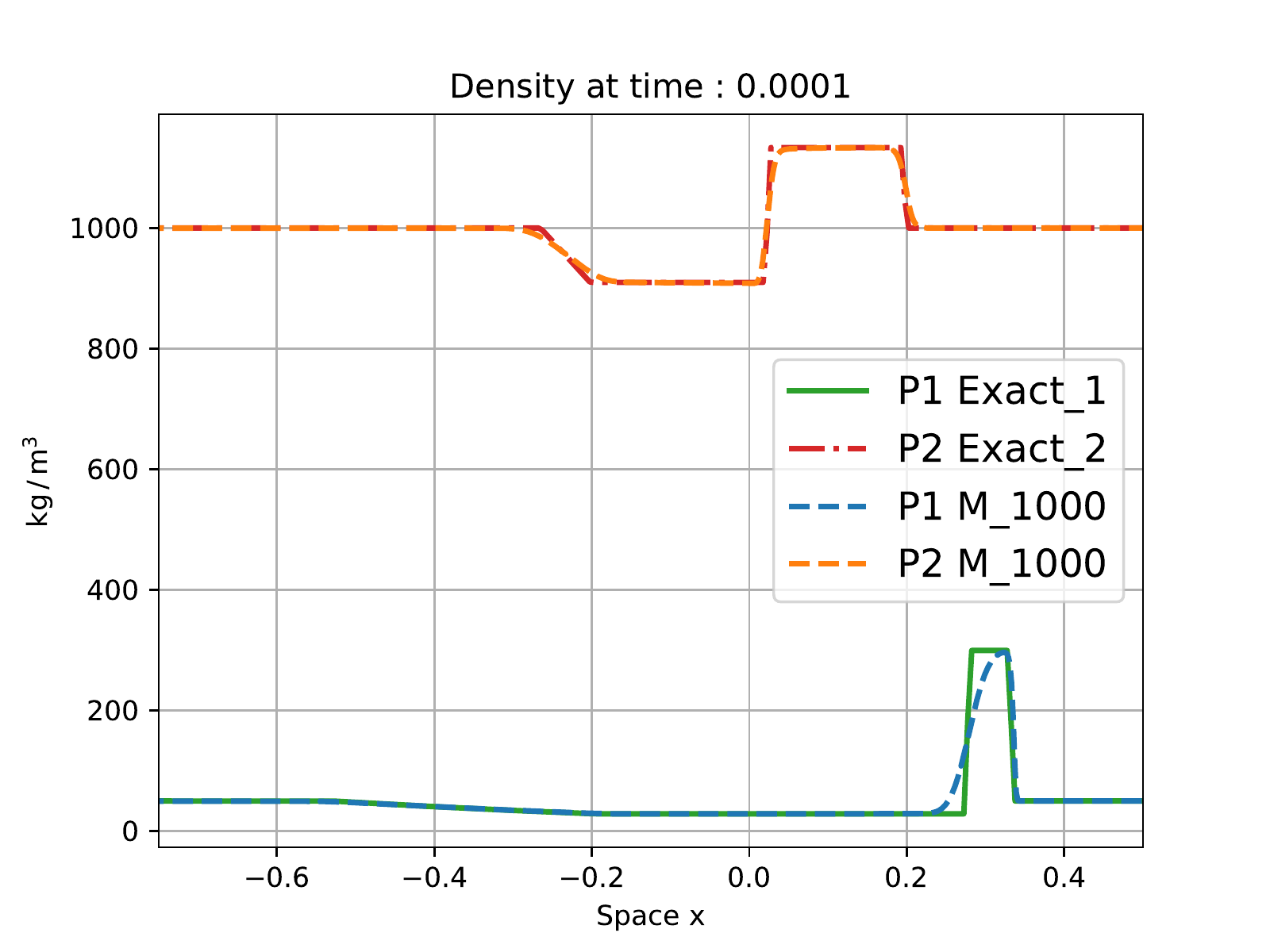}\\
\includegraphics[scale = 0.47]{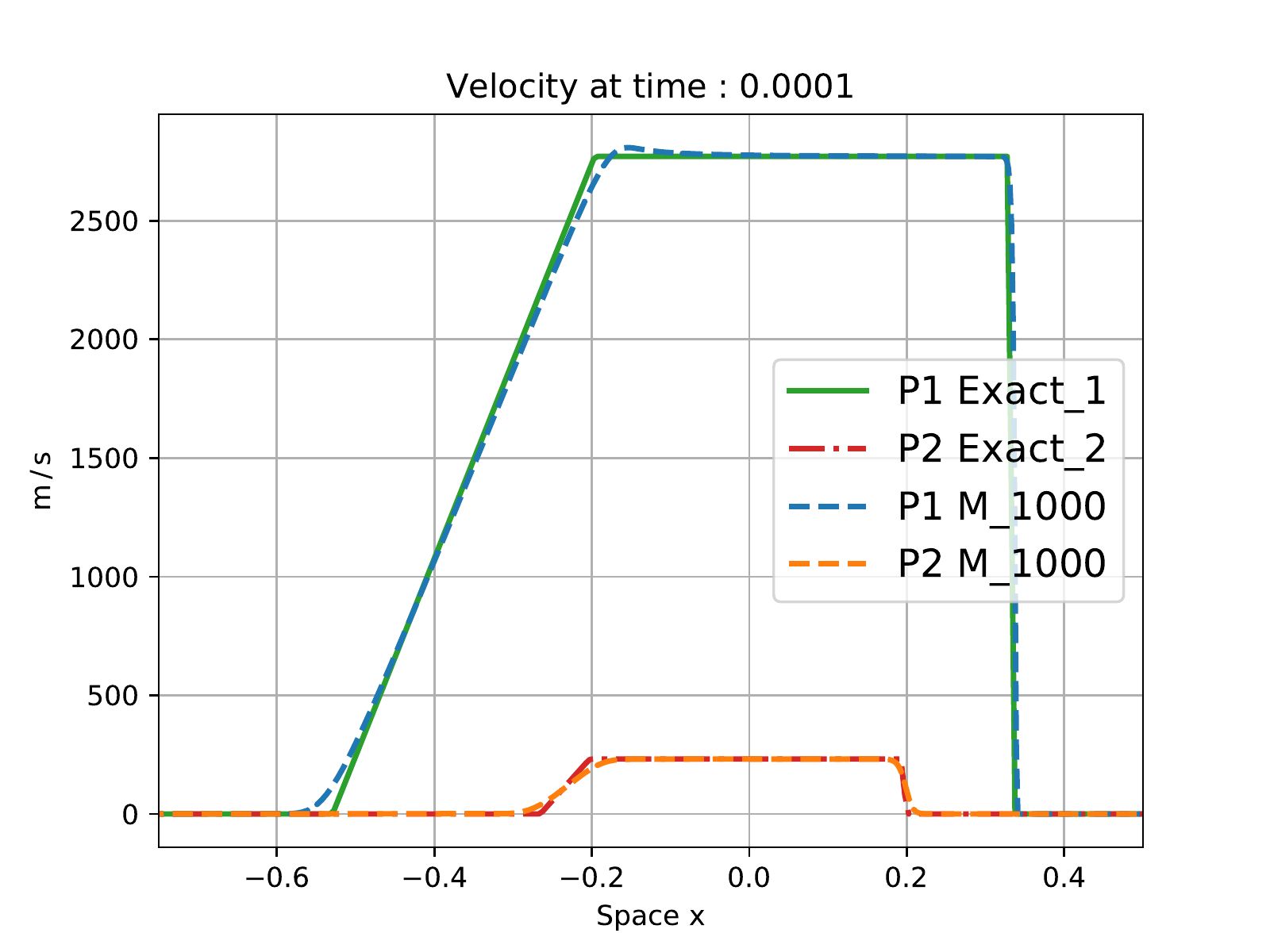}\,
\includegraphics[scale = 0.47]{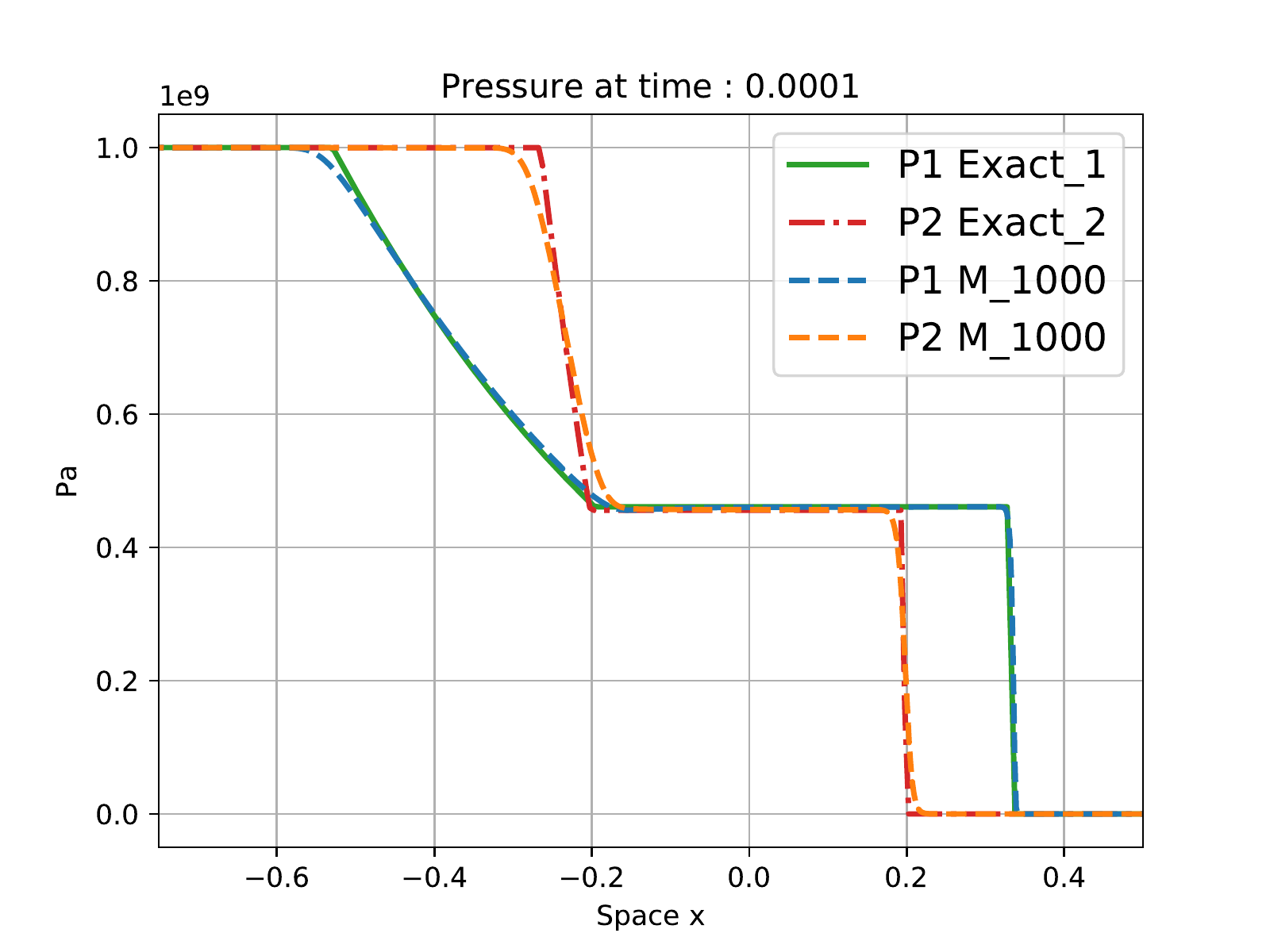}
\caption{Uniform volume fraction test at time $t = 100\, \mu\mathrm{s}$ of a stratified flow ($r=0$). Numerical solutions (dashed lines) of gas ($P1$) and liquid ($P2$) phases have been computed with a uniform mesh of $M=1000$, and are reported against their exact solutions (solid and dash-dotted line).}\label{Fig:T1r0}
\end{figure}

\begin{figure}[!htbp]
\centering
\includegraphics[scale = 0.47]{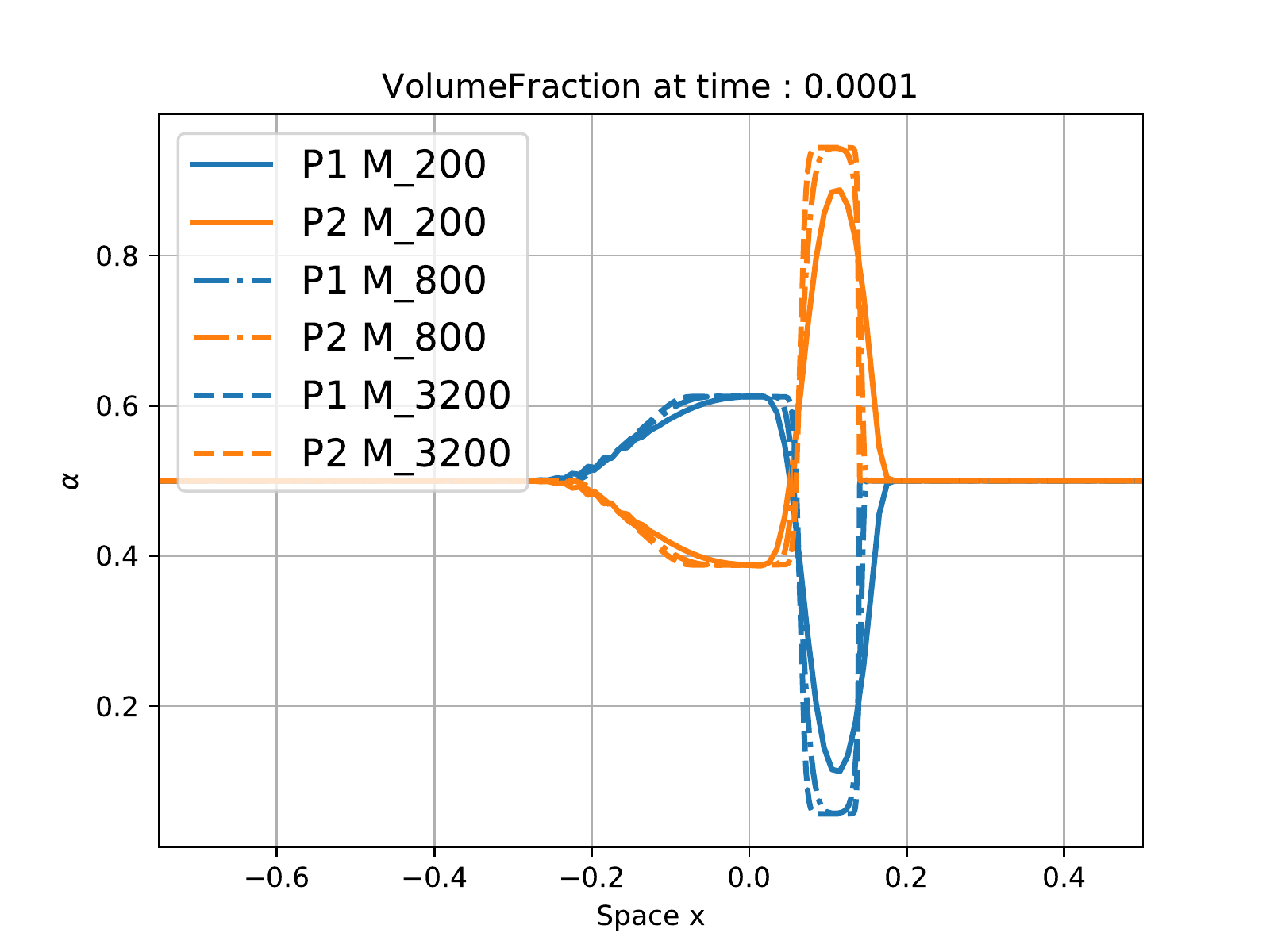}\,
\includegraphics[scale = 0.47]{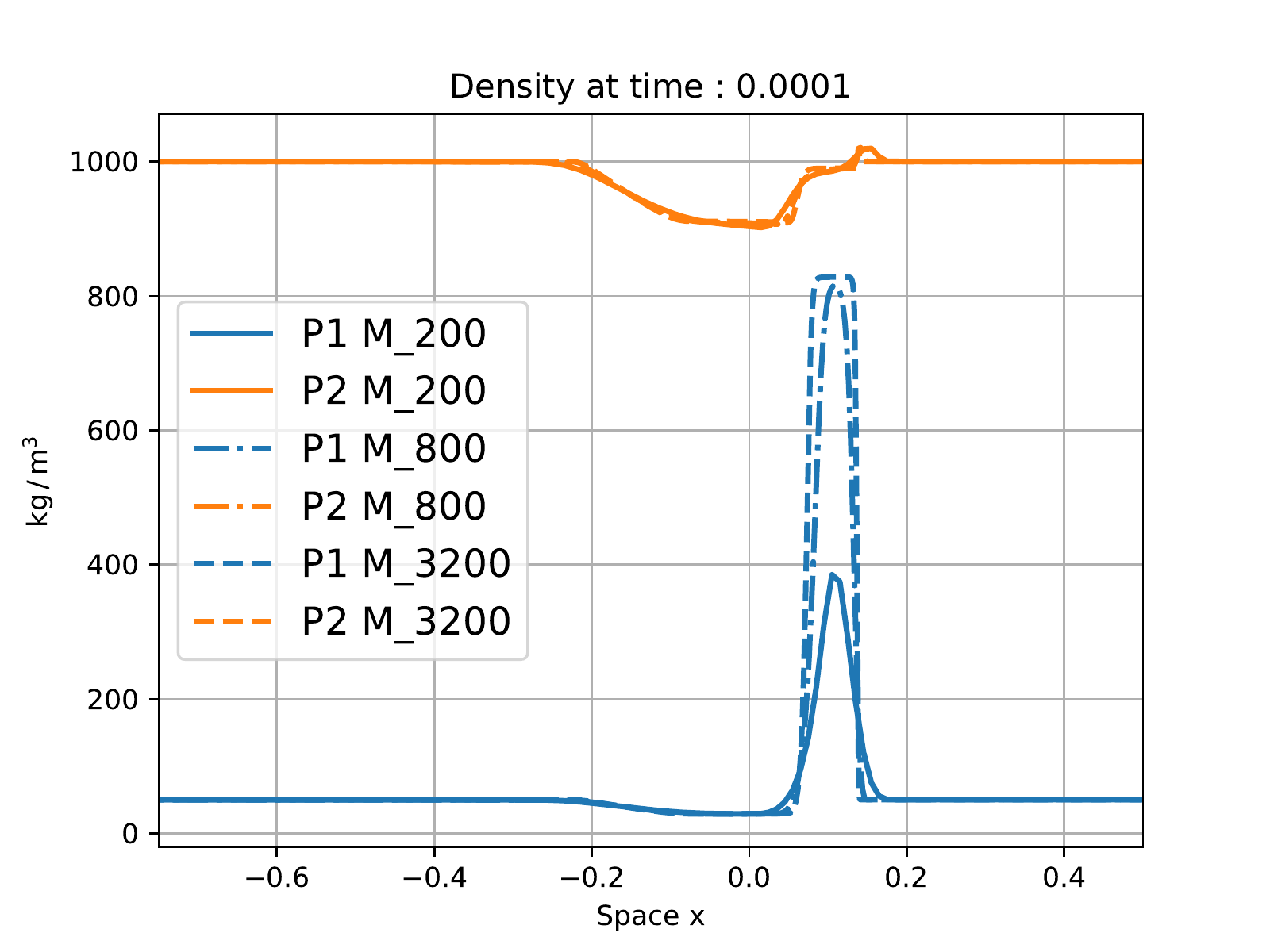}\\
\includegraphics[scale = 0.47]{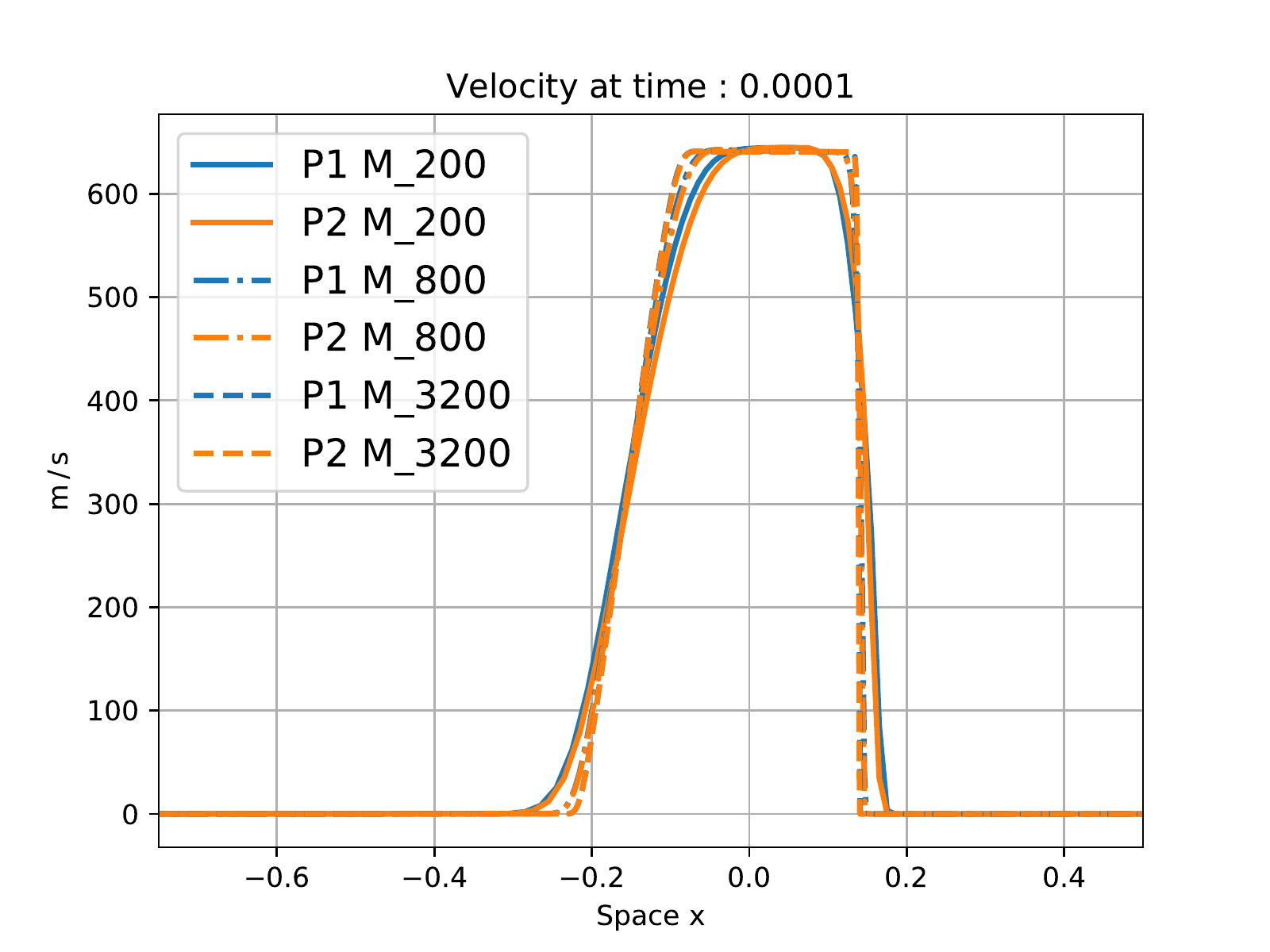}\,
\includegraphics[scale = 0.47]{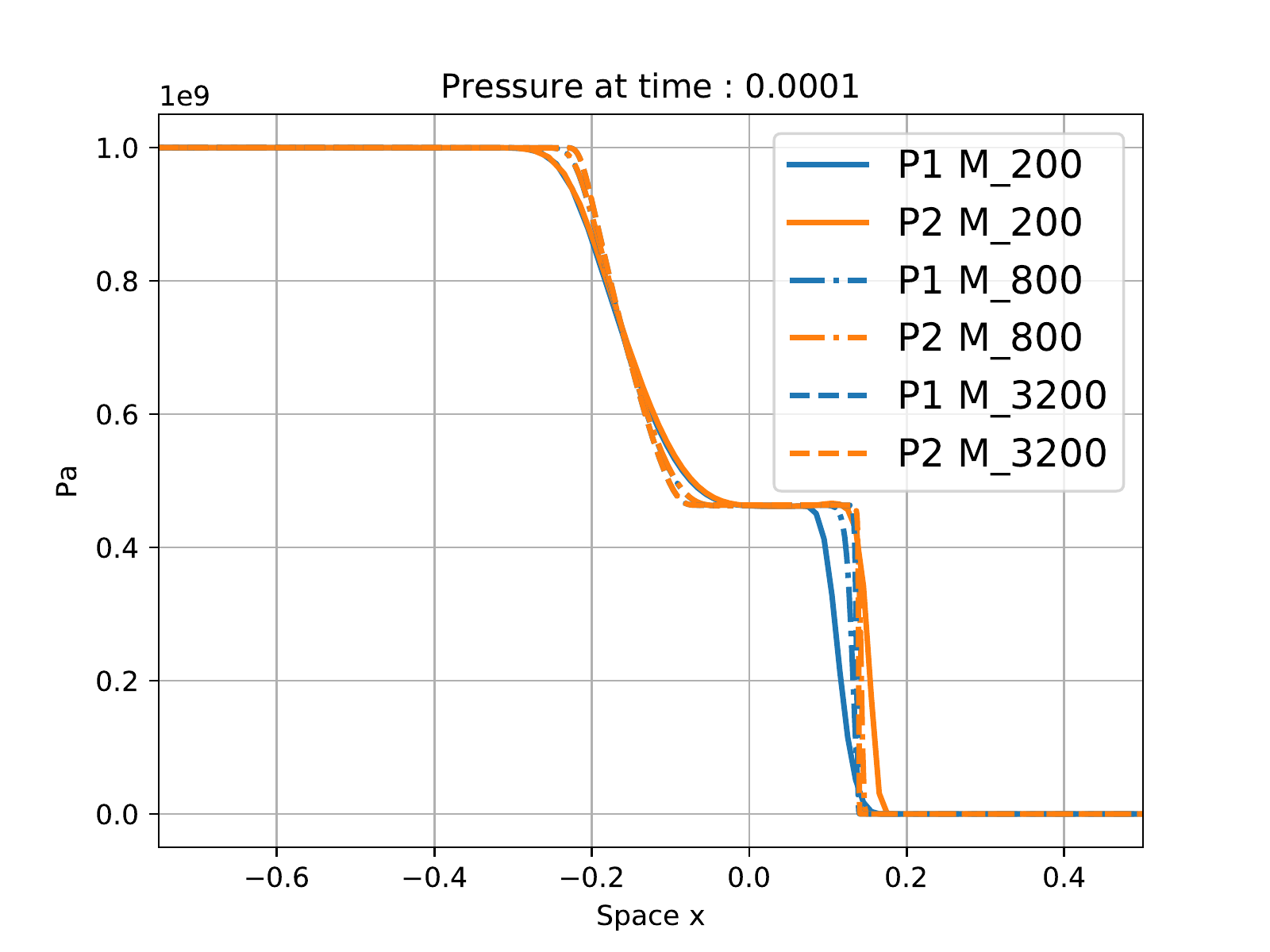}
\caption{Uniform volume fraction test at time $t = 100\, \mu\mathrm{s}$ of a disperse flow ($r=1$). Numerical solutions for gas phase ($P1$) and the liquid phase ($P2$) computed with number of cells $M$ are reported.}\label{Fig:T1r1}
\end{figure}

The stratified flow regime simulates two non-interacting fluids one on top of the other, while the disperse one models a dilute flow of air inside water.
As expected, this latter situation leads to interaction of phases, even though no relaxation is imposed. This is due to the discontinuity of volume fraction at cell interface that enters in the numerical flux through the probability coefficients. Notice the perfect coupling of phases in absence of relaxation for the case $r=1$. This clearly highlights the importance of Lagrangian fluxes to maintain it.
Furthermore, results for the case $r=1$ show near coalescence of velocity and pressure curves: the two phases seem to converge to equilibrium. However, inspection of shock profiles shows slight differences between the fluids, see Fig. \ref{Fig:T1r1Zoom}.

\begin{figure}[!htbp]
\centering
\includegraphics[scale=0.47]{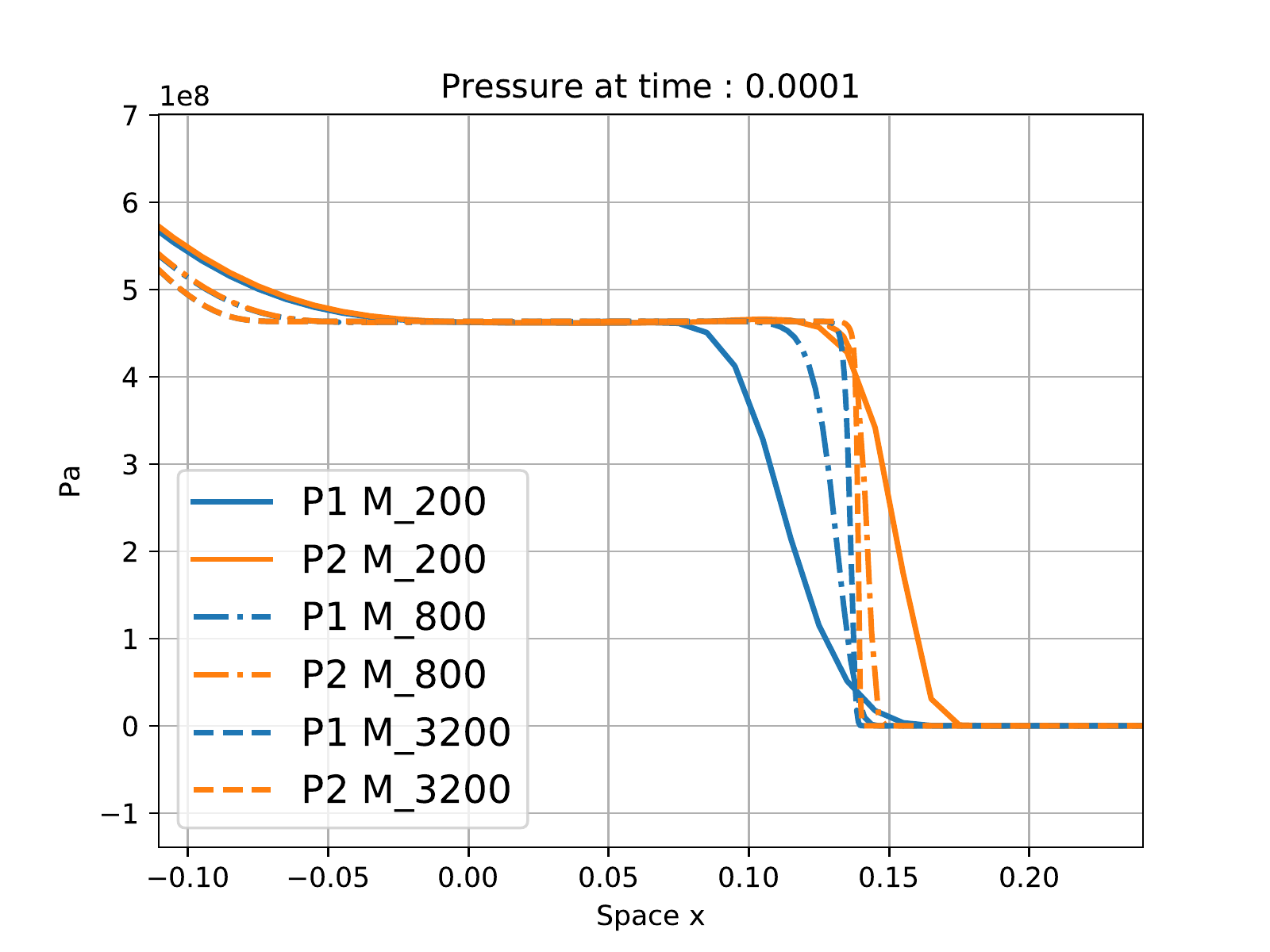}\,
\includegraphics[scale=0.47]{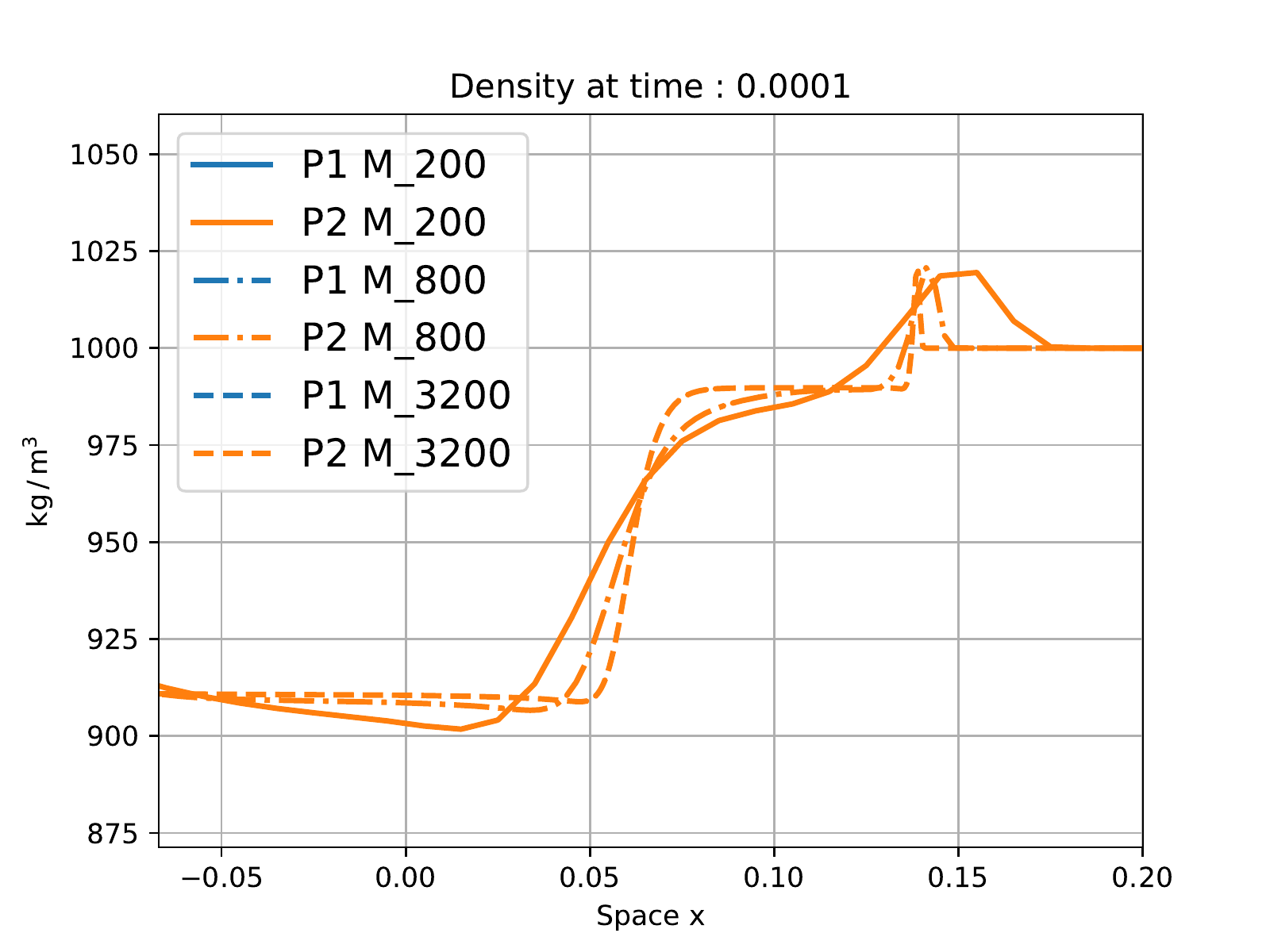}
\caption{Magnified details of pressure shock profiles (left) and density overshoots (right), from Fig. \ref{Fig:T1r1}}\label{Fig:T1r1Zoom}
\end{figure}

Notice that such an example suggests that, when choosing $r\neq 0$, relaxation is not the only mechanical interaction between the two fluids. Finally, an overshoot in the top right corner of gas density phase appears for $r=1$, whose amplitude reduces by mesh refinement, see Fig. \ref{Fig:T1r1Zoom}, suggesting convergence in $L^1$-norm.

\begin{figure}[!htbp]
\centering
\includegraphics[scale = 0.47]{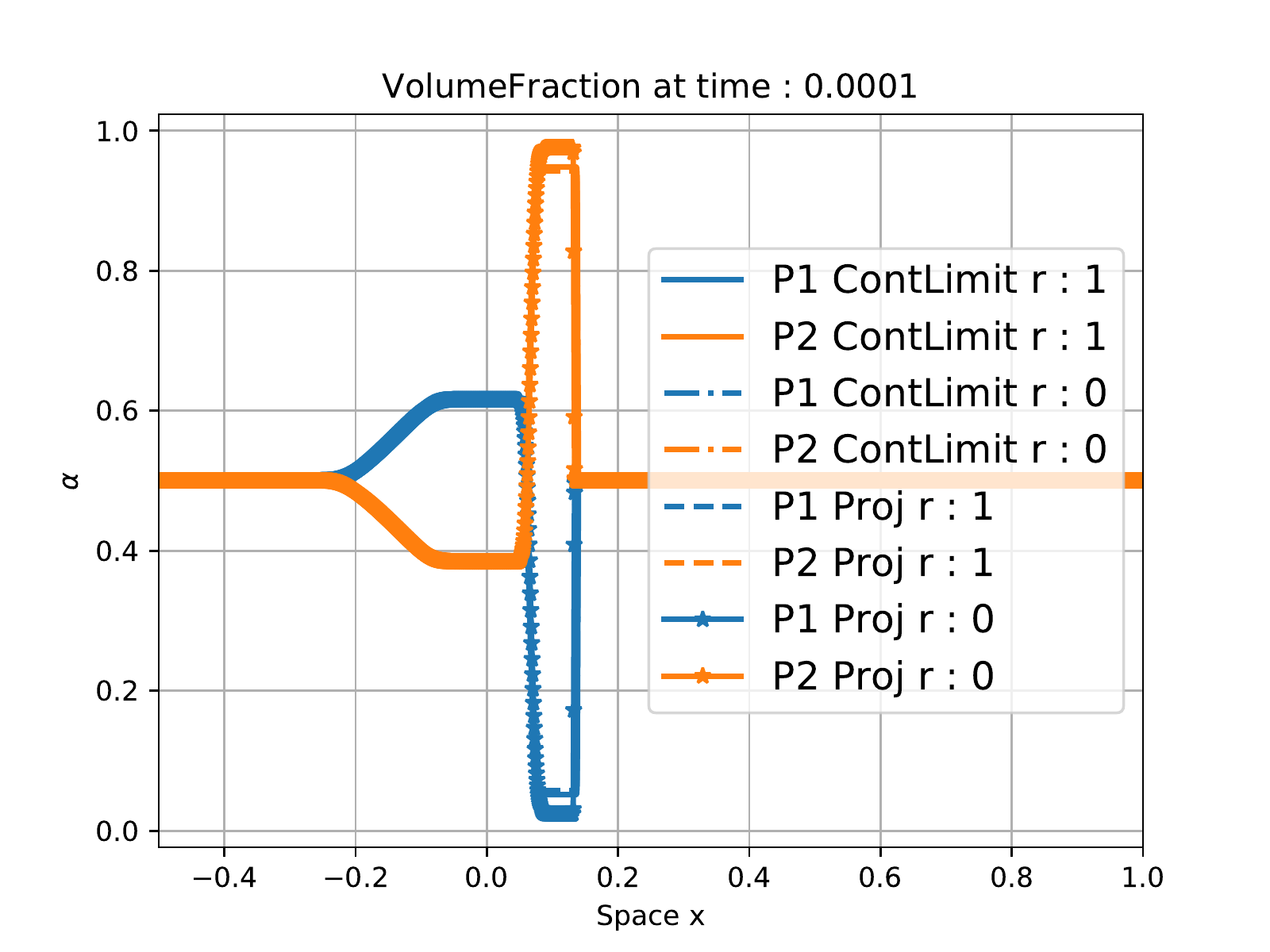}\,
\includegraphics[scale = 0.47]{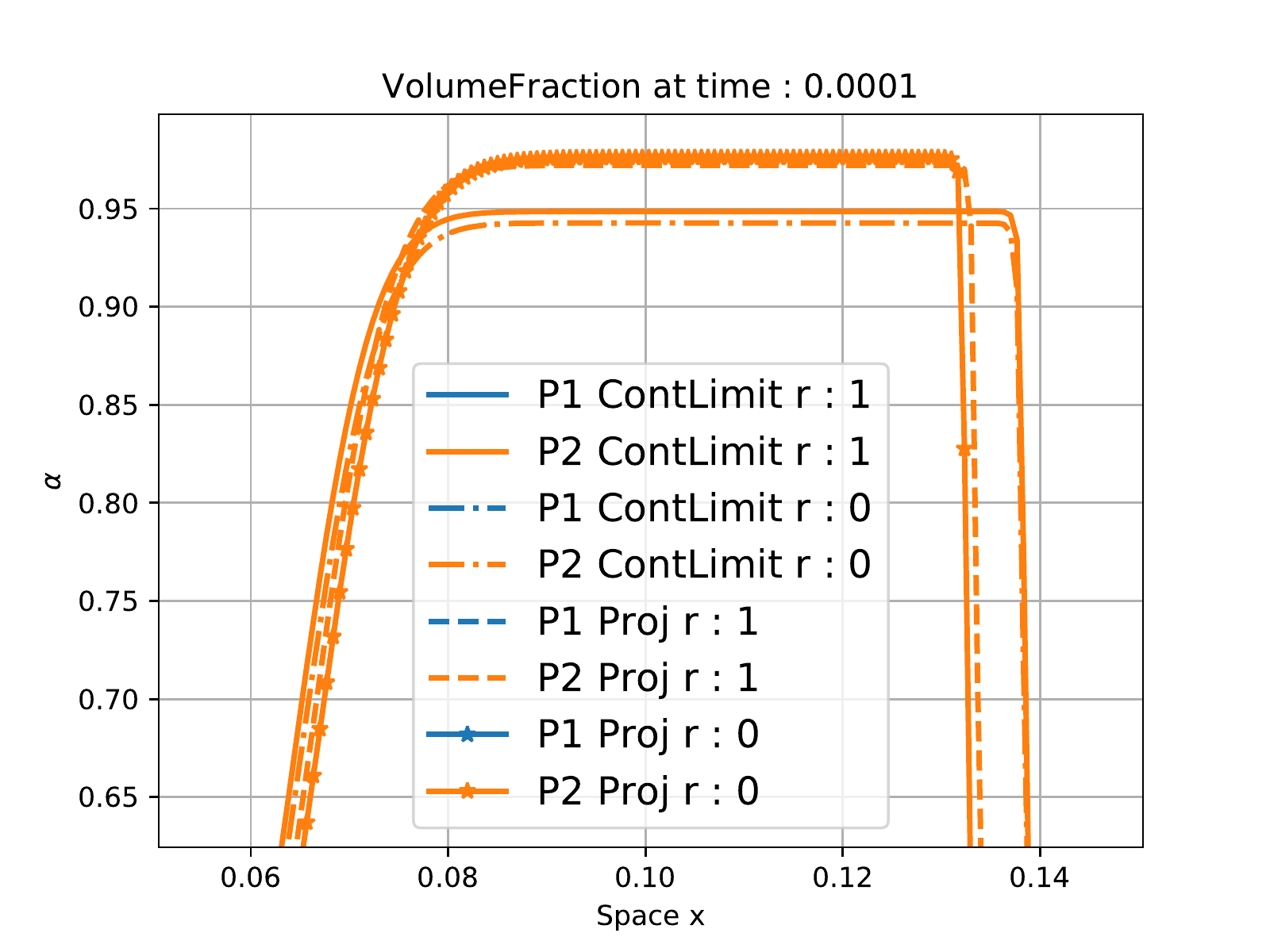}
\caption{Uniform volume fraction test at time $t = 100\, \mu\mathrm{s}$ using an infinite drag coefficient. Left column : comparison between the limit-based relaxation strategy and the projection method; right column: magnified detail of phase $2$ post-sock state. 
Numerical solutions for gas phase ($P1$) and the liquid phase ($P2$) have been computed with $M=3000$ cells for both the stratified ($r=0$) and disperse ($r=1$) regimes.}\label{Fig:T2_a}
\end{figure}

\subsubsection{Adding relaxation}

We perform the same test, but adding the relaxation procedures described in Appendix \ref{appendix:RelaxationStrategies}. In this setting an infinite interfacial area is present inside each cell. Results for both regimes (i.e. the stratified and the disperse case) computed with different relaxation procedures are reported in Fig. \ref{Fig:T2_a} and Fig. \ref{Fig:T2}. In each of these figures, results for each relaxation strategy and each $r\in \lbrace 0,1\rbrace$ are presented for comparison.\\
Fixing a relaxation strategy, analogous results are obtained for both flow regimes (each choice of $r$), even though discrepancies between the two patterns can be recognized near rarefaction and shocks. 
Particularly evident is the impact of the choice of $r$ for the post-shock state of density (see first raw of Fig. \ref{Fig:T2}), even if comparable discrepancies can be recognized even for the rest of quantities of interest.\\
Furthermore, discrepancies can also be seen comparing results for different relaxation procedures. For example, differences in shock location predictions are present between relaxation procedures, see Fig. \ref{Fig:T2_a}. This highlights the fact that the numerical solution of such a test problem is highly dependent on both the relaxation strategy and the probability coefficients at the volume interfaces, raising the question of uniqueness: how can we single-out a physically relevant solution among the infinitely many generated by different realizations of the relaxation strategy and the parameter $r$ ?

\begin{figure}[!htbp]
\centering
\includegraphics[scale = 0.47]{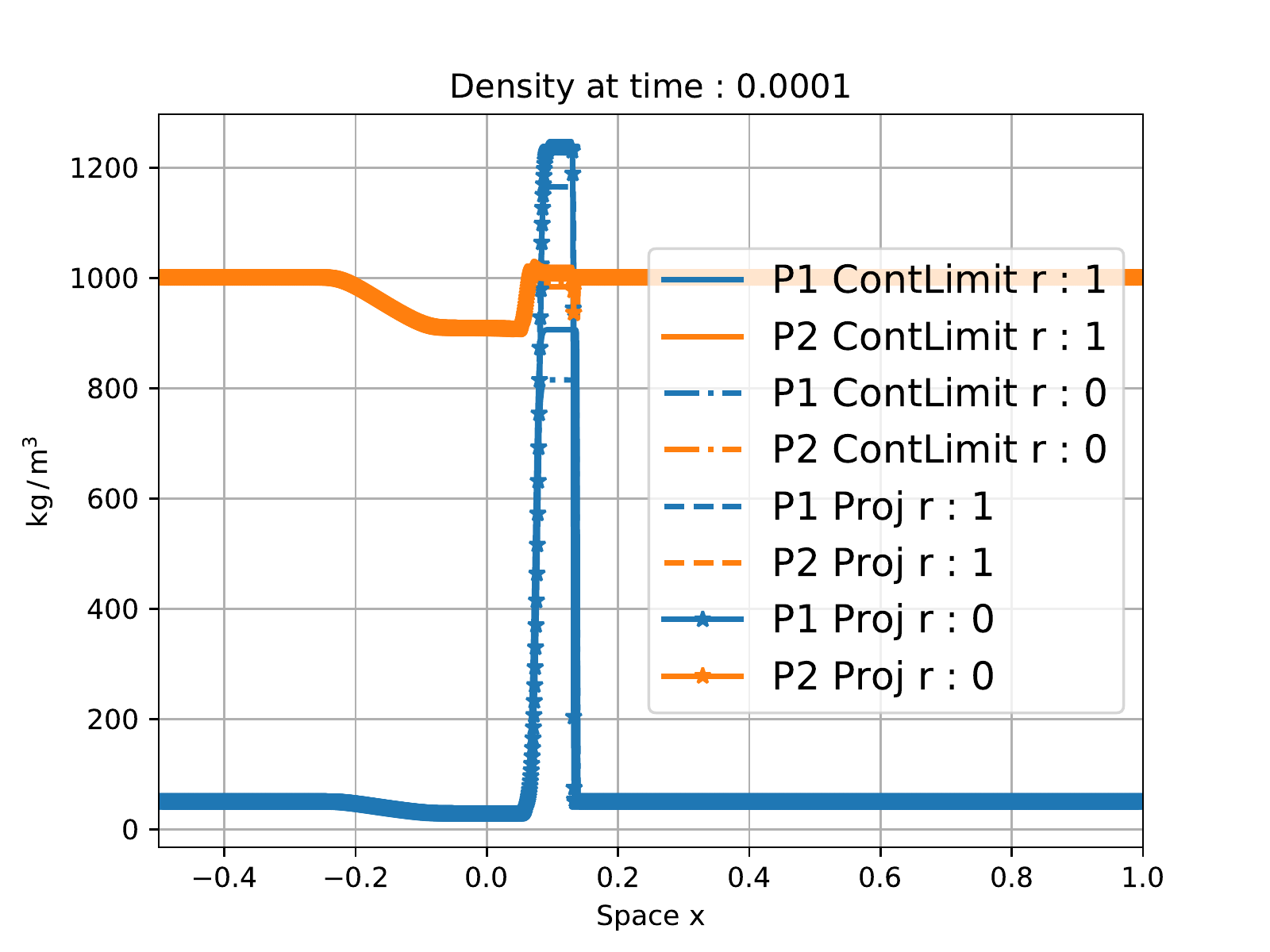}\,
\includegraphics[scale = 0.47]{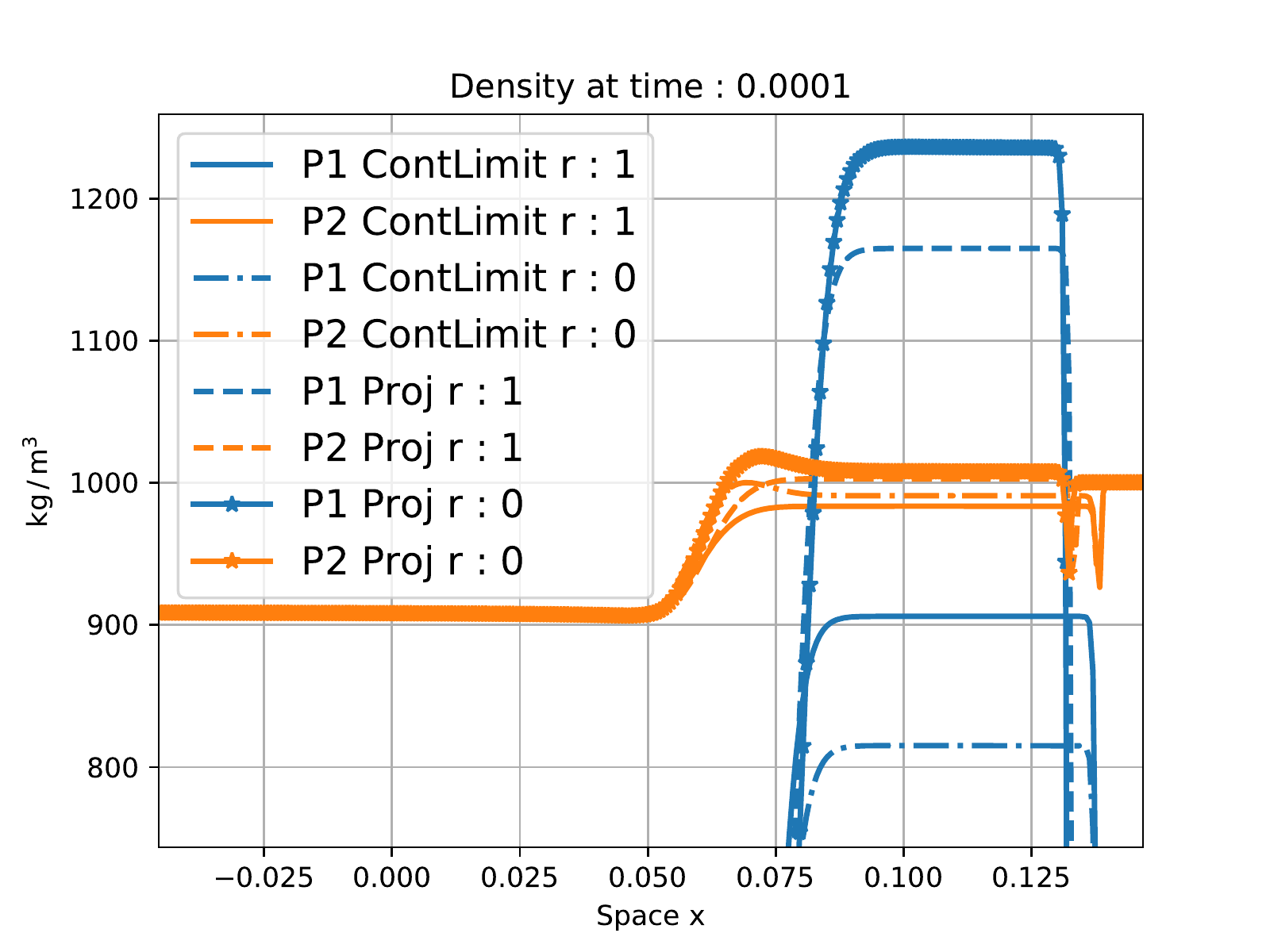}\\
\includegraphics[scale = 0.47]{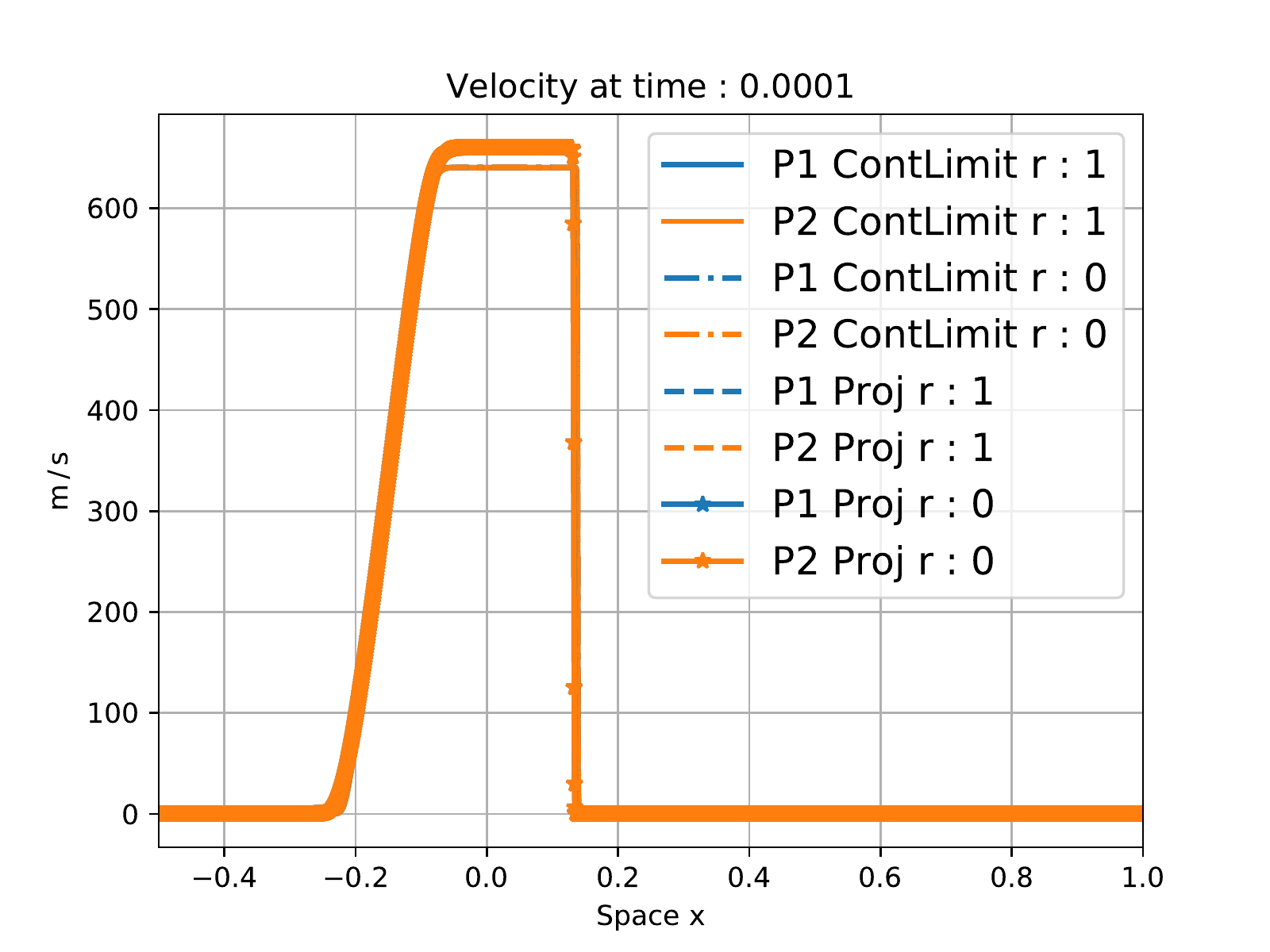}\,
\includegraphics[scale = 0.47]{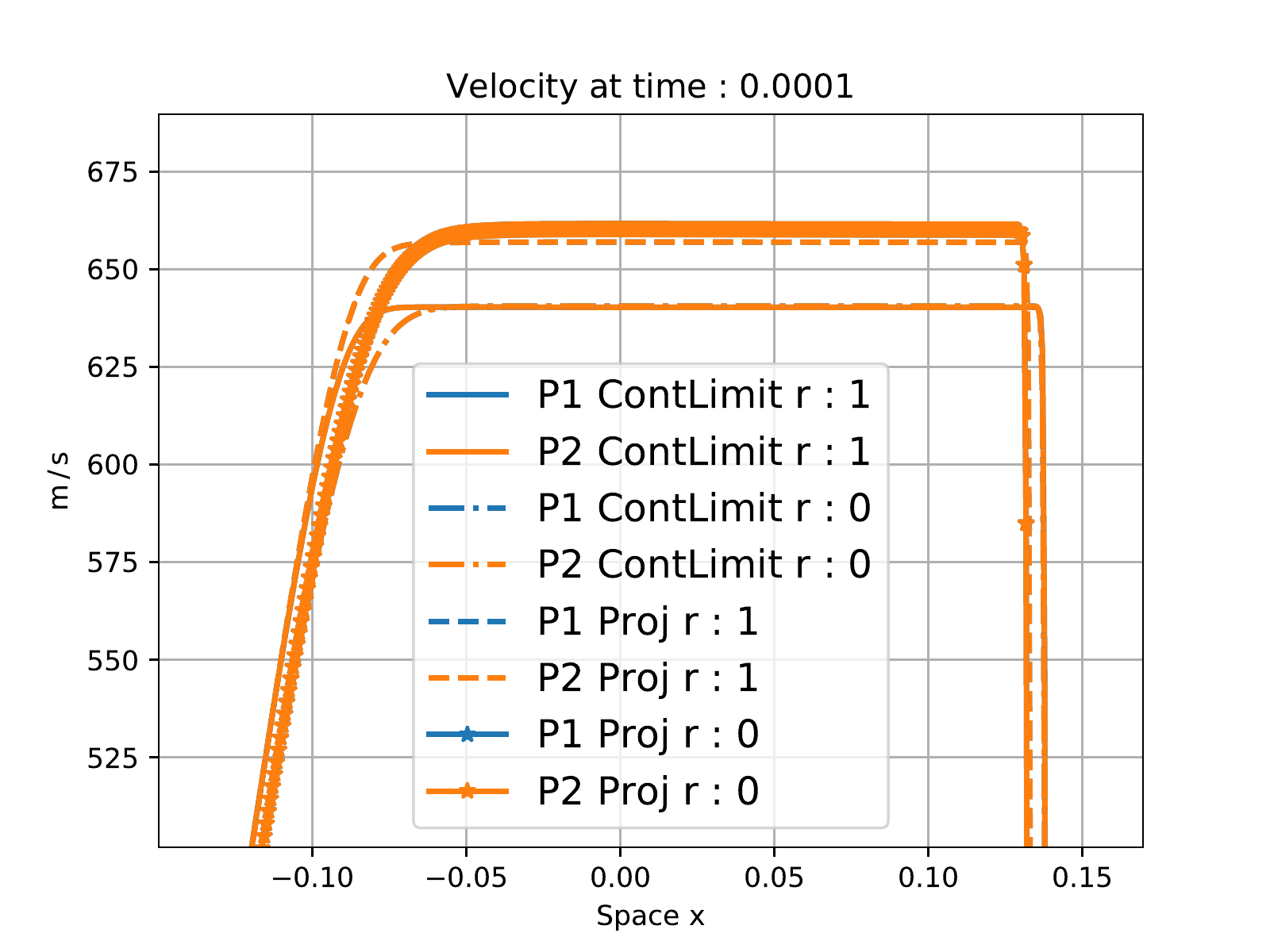}\\
\includegraphics[scale = 0.47]{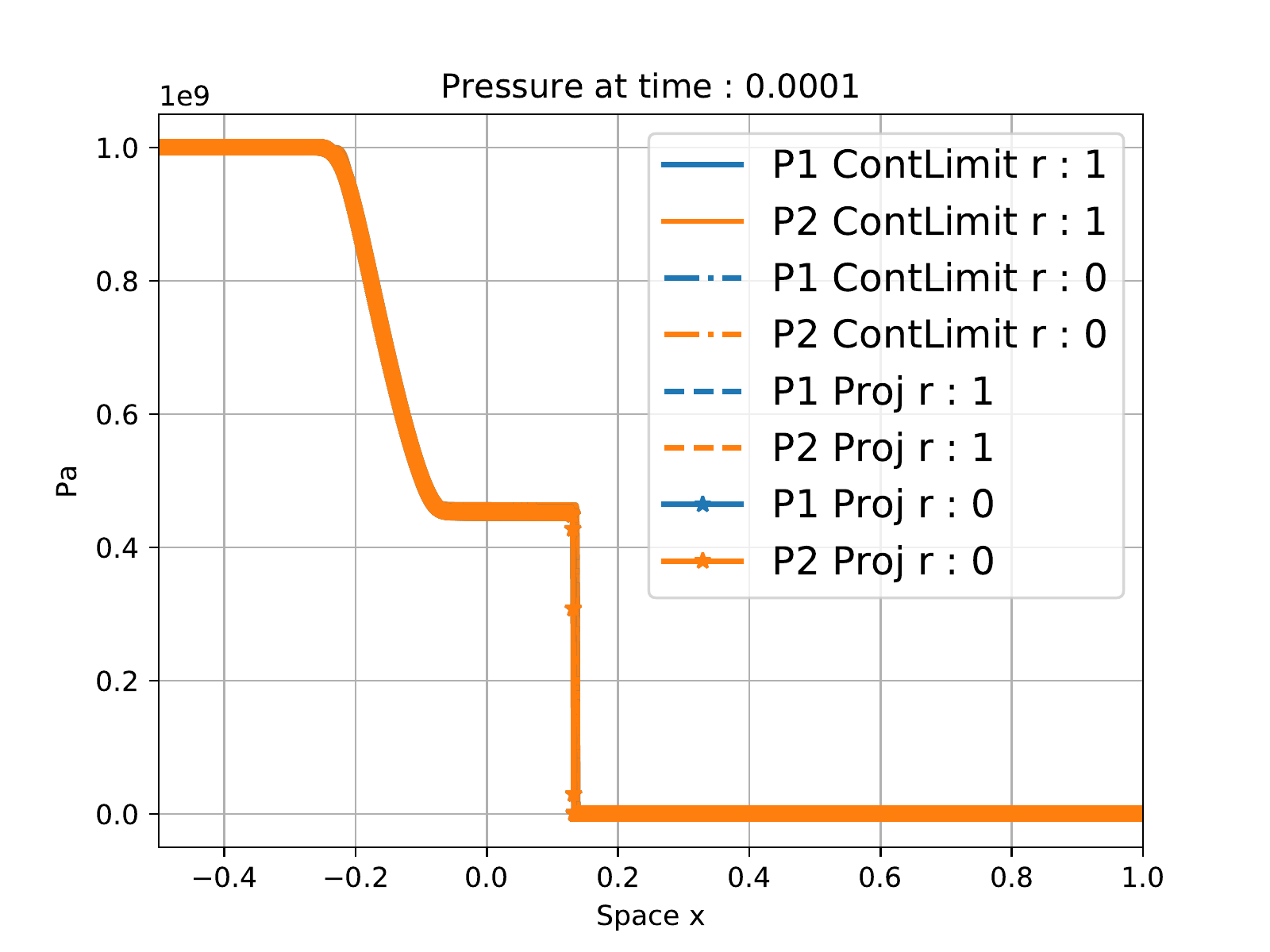}\,
\includegraphics[scale = 0.47]{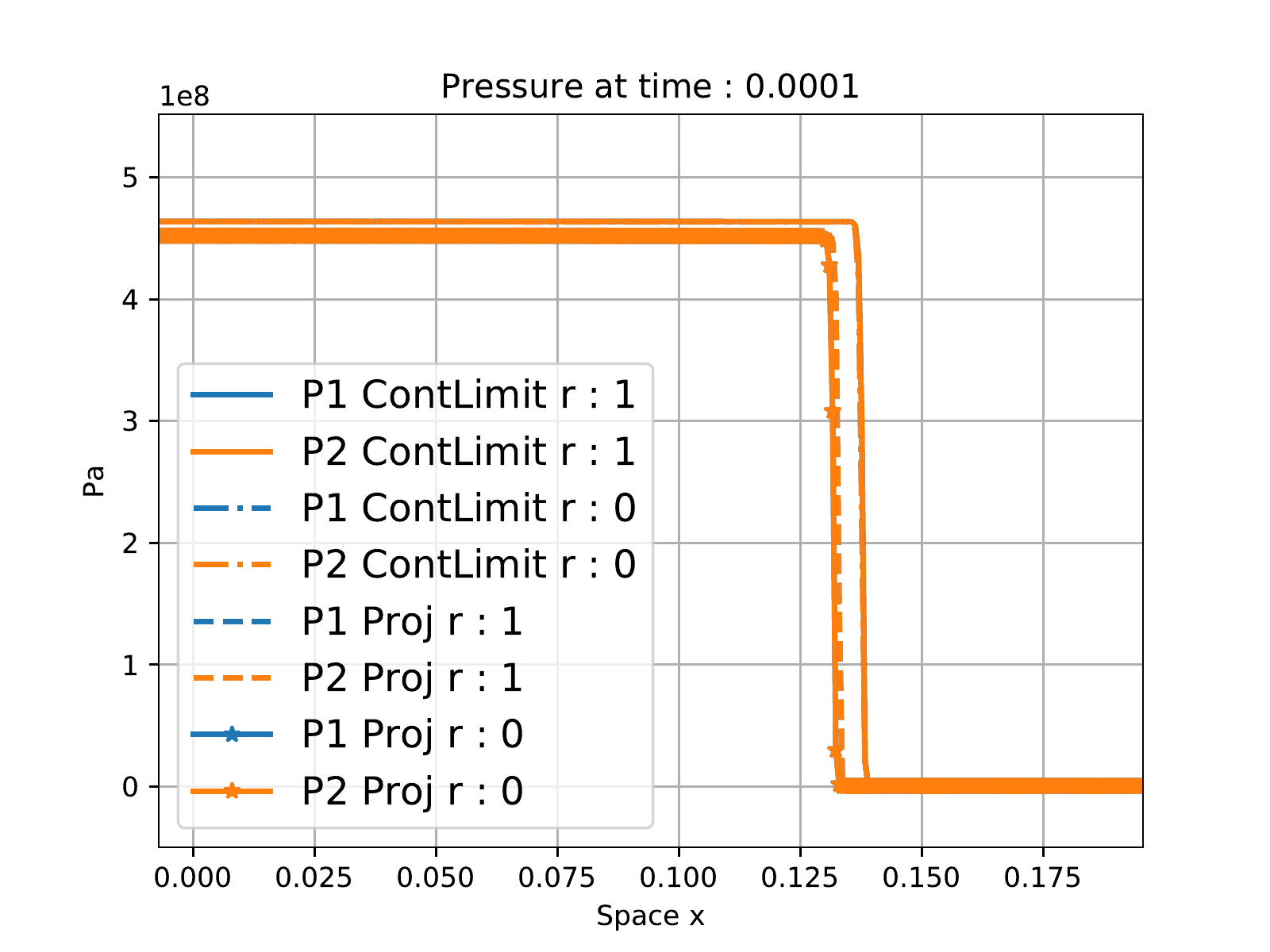}
\caption{Uniform volume fraction test at time $t = 100\, \mu\mathrm{s}$ using an infinite drag coefficient. Left column : comparison between the limit-based relaxation strategy and the projection method; right column: magnified detail of phase $2$ post-sock state. 
Numerical solutions for gas phase ($P1$) and the liquid phase ($P2$) have been computed with $M=3000$ cells for both the stratified ($r=0$) and disperse ($r=1$) regimes.}\label{Fig:T2}
\end{figure}

\subsection{Pure Phases}

A prototypical benchmark problem for the simulation of two-phase flow is the ability of a scheme of resolving sharp interfaces or reproducing pure phases.
Unfortunately the present scheme does not enjoy such property, due to numerical viscosity. Indeed, when attempting to simulate sharp interfaces separating different constituents, the numerical scheme will not maintain the volume fraction in the set $\mathcal{X} = \lbrace 0,1\rbrace$, due to numerical diffusion. This corresponds to smearing out the interface over several computational cells, thus creating a mixing zone around the exact interface location. A numerical artifact used to circumvent the numerical failure arising in such situation is to assume a negligible amount of dispersed phase, as to stabilize the algorithmic procedure.\\
For the sake of comparison we therefore assume such a strategy to investigate the impact of parameter $r$ when simulating pure phases.
We consider the following initial condition in terms of the primitive variables $\textbf{V} = [\alpha,\rho,u,p]$,
\[
\textbf{V}_0(x) = 
\begin{cases}
[\textbf{V}_L^{(1)}, \textbf{V}_L^{(2)}] & \textit{ if }\, x<0,\\
[\textbf{V}_R^{(1)}, \textbf{V}_R^{(2)}] & \textit{ if }\, x>0.
\end{cases}
\]
where
\[
\textbf{V}_L^{(1)} = \begin{bmatrix}
10^{-6}\\
50\\
0\\
2\cdot 10^8
\end{bmatrix},
\quad
\textbf{V}_L^{(2)} = \begin{bmatrix}
1-10^{-6}\\
1000\\
0\\
2\cdot 10^8
\end{bmatrix}
\quad
\textbf{V}_R^{(1)} = \begin{bmatrix}
1-10^{-6}\\
50\\
0\\
10^5
\end{bmatrix},
\quad
\textbf{V}_R^{(2)} = \begin{bmatrix}
10^{-6}\\
1000\\
0\\
10^5
\end{bmatrix}
\]
\begin{figure}[!htbp]
\centering
\includegraphics[scale = 0.47]{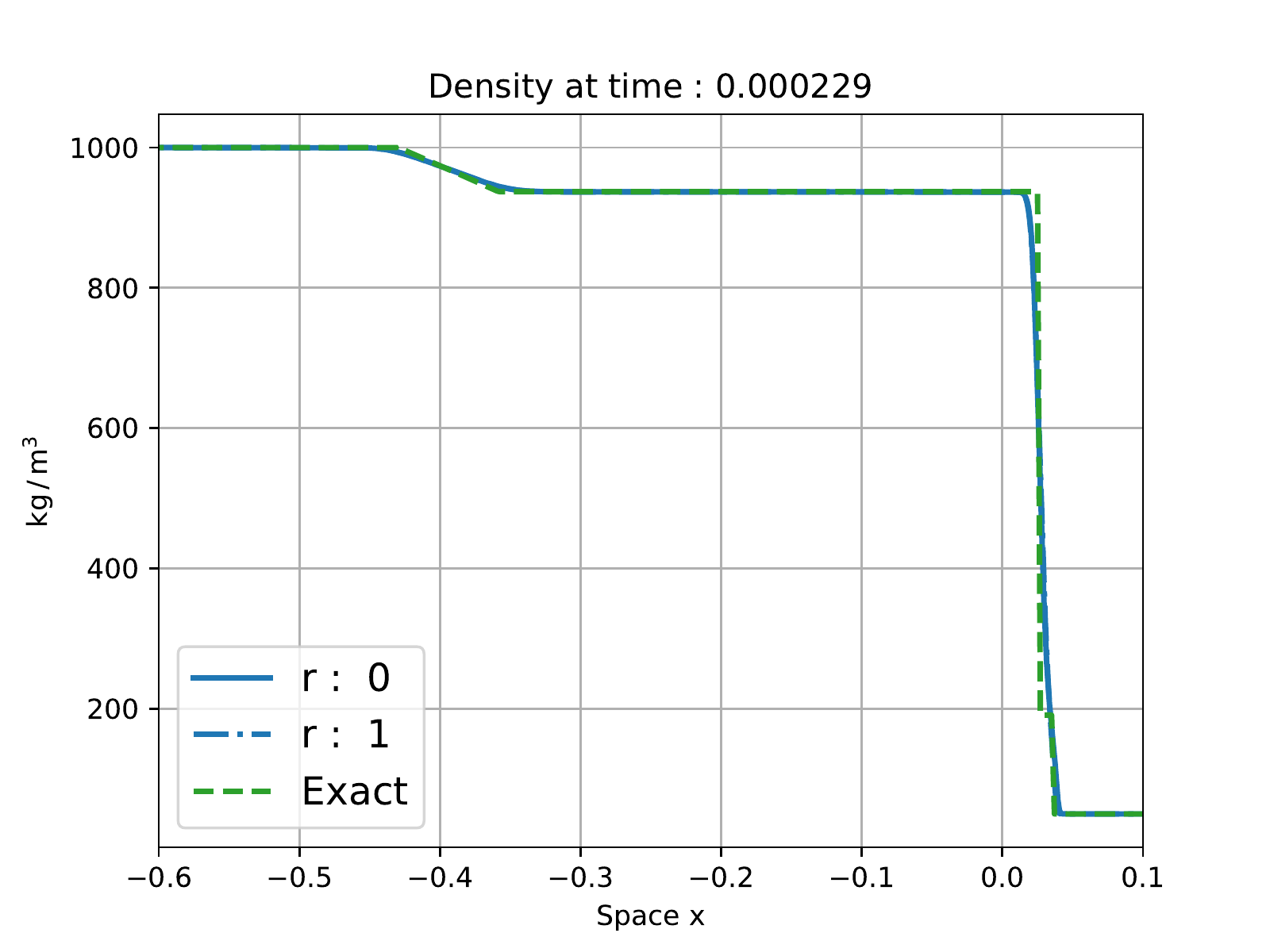}\\
\includegraphics[scale = 0.47]{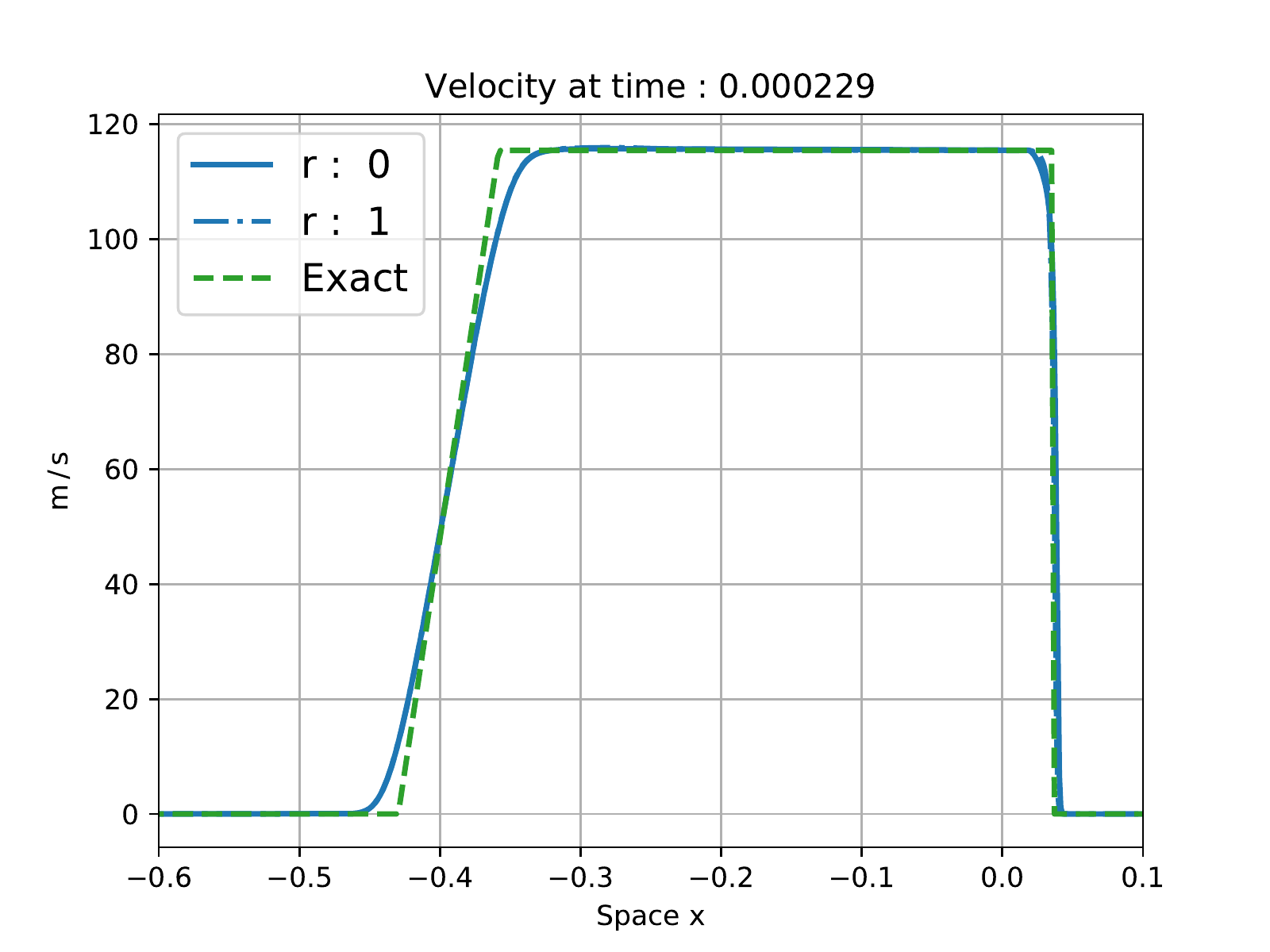}\,
\includegraphics[scale = 0.47]{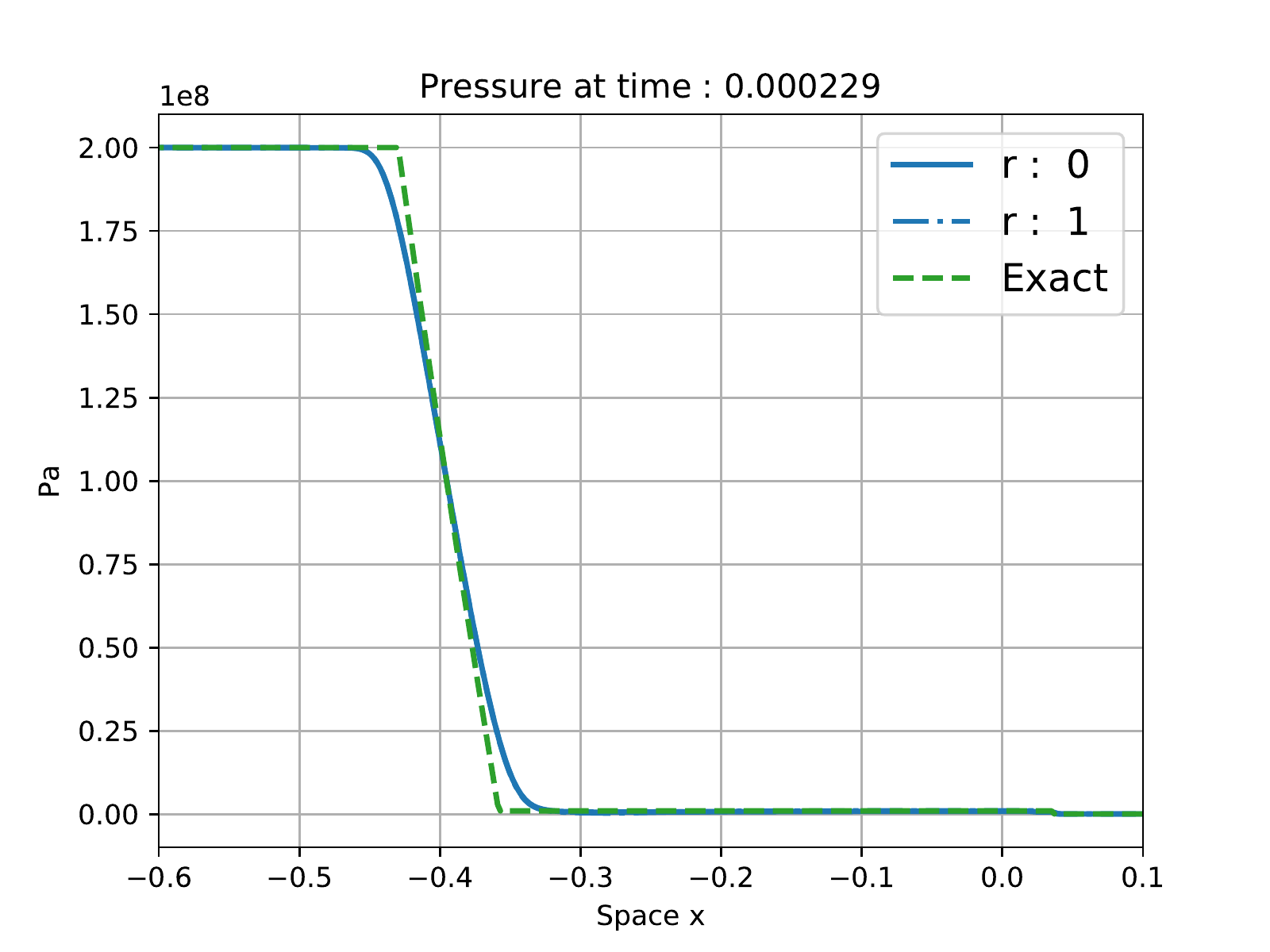}
\caption{Mixture quantities for the nearly pure phases test at time $t = 229\, \mu\mathrm{s}$, using an infinite drag coefficient. Left column : continuous limit-based relaxation strategy; right column: projection-based relaxation strategy. Numerical solutions for gas phase ($P1$) and the liquid phase ($P2$) have been computed with $M=1000$ cells for both the stratified ($r=0$) and disperse ($r=1$) regimes.}\label{Fig:T3}
\end{figure}
We aim at showing the capacity of the scheme to handle nearly pure mixtures, then the following mixture quantities of interest are computed for each regime
\begin{equation}
\rho := \alpha_1\rho_1 + \alpha_2\rho_2\\
\qquad
u := \frac{\alpha_1\rho_1u_1 + \alpha_2\rho_2u_2}{\rho}\\
\qquad
p := \alpha_1p_1 + \alpha_2p_2\\
\end{equation}
As the sum of the equations for each phase at each point location $x= x_i$ results in a formally equivalent system to the single phase Euler equations for the mixture, one can compute the corresponding exact solution according to well-known solvers \cite{Ivings,ToroRS}.
Results for both the disperse and stratified flow mixtures are compared with exact solution between pure phases in Fig. \ref{Fig:T3}. The continuous-limit relaxation strategy is used for this test case.
Again relaxed models produce analogous results, see Fig. \ref{Fig:T3}.
Notice that this should be not surprising: indeed, when considering the probability coefficients one these do not depend on the choice of $r$, in the regions of single-phase flow: let us assume that a material interface is located at $x=x_{i+\frac{1}{2}}$ for some $i$; then it holds that for any $j\neq i$ and a $k\neq l \in \lbrace 1,2\rbrace$ 
\begin{align*}
\alpha^k_j = \alpha^k_{j+1} = 1 - \varepsilon
\qquad
\textit{and}
\qquad
\alpha^l_j = \alpha^l_{j+1} = \varepsilon
\end{align*}
where $\varepsilon$ represents the virtual amount of phase $l$ used at the numerical level, and is assumed to be comparatively small, i.e. $0<\varepsilon \ll \frac{1}{2}$.
Then, one gets
\begin{align*}
\mathcal{P}_{j+\frac{1}{2}}\left[\Sigma_k,\Sigma_k\right] &= r\max(1-\varepsilon-\varepsilon,0) + (1-r)\min(1-\varepsilon,1-\varepsilon) = 1 - (1+r)\varepsilon\\
\mathcal{P}_{j+\frac{1}{2}}\left[\Sigma_k,\Sigma_l\right] &= r\min(1-\varepsilon,\varepsilon) + (1-r)\max(1-\varepsilon-1+\varepsilon,0) = r\varepsilon\\
\mathcal{P}_{j+\frac{1}{2}}\left[\Sigma_l,\Sigma_l\right] &= r\max(\varepsilon-1+\varepsilon,0) + (1-r)\min(\varepsilon,\varepsilon) = (1-r)\varepsilon\\
\mathcal{P}_{j+\frac{1}{2}}\left[\Sigma_l,\Sigma_k\right] &= r\min(\varepsilon,1-\varepsilon) + (1-r)\max(\varepsilon-\varepsilon,0) = r\varepsilon
\end{align*}
The significance of this is that, when considering the ideal case of pure phases ($\varepsilon = 0$), one obtains that the only non-zero probability coefficient is the one associated to the probability of having the same phase on both sides of an interface for the phase that has the higher-volume fraction. This corresponds to making all the twophase-fluxes contributions vanish, and the classical, single-phase Godunov scheme is recovered. Hence, each phase behaves independently of the complementary one.\\
Conversely, around the material interface located at $x = x_{i+\frac{1}{2}}$, it holds
\begin{align*}
\alpha^k_i = 1- \varepsilon\textit{ and } \alpha^k_{i+1} = \varepsilon
\qquad
\textit{and}
\qquad
\alpha^l_i = \varepsilon \textit{ and } \alpha^l_{i+1} = 1 - \varepsilon
\end{align*}
so that
\begin{align*}
\mathcal{P}_{i+\frac{1}{2}}\left[\Sigma_k,\Sigma_k\right] &= r\max(1-\varepsilon-1+\varepsilon,0) + (1-r)\min(1-\varepsilon,\varepsilon) = (1-r)\varepsilon\\
\mathcal{P}_{i+\frac{1}{2}}\left[\Sigma_k,\Sigma_l\right] &= r\min(1-\varepsilon,1-\varepsilon) + (1-r)\max(1-\varepsilon-\varepsilon,0) = 1 - (2-r)\varepsilon\\
\mathcal{P}_{i+\frac{1}{2}}\left[\Sigma_l,\Sigma_l\right] &= r\max(\varepsilon-\varepsilon,0) + (1-r)\min(\varepsilon,1-\varepsilon) = (1-r)\varepsilon\\
\mathcal{P}_{i+\frac{1}{2}}\left[\Sigma_l,\Sigma_k\right] &= r\min(\varepsilon,\varepsilon) + (1-r)\max(\varepsilon-1+\varepsilon,0) = r\varepsilon
\end{align*}
which again make vanish all the contributions not associating to finding phase $k$ and phase $l$, respectively, on each side of the interface, as one would expect.\\
Notice that such behavior is immediately broken if even negligible (but not-zero) amount of complementary phase is considered in each volume (i.e. $\varepsilon > 0$). This introduces the contribution of other fluxes terms which slightly affect the solution profile, depending on $r$.
Indeed, slight discrepancies can be seen around shocks, even if the overall performance of both models ($r=0$ and $r=1$) results acceptable and virtually equal.
The significance of the above analysis is that, the discrepancies between the two models generated by different choices of parameter $r$, are dependent on the amount of virtual phase we allocate in each pure chamber.
This in turn, also highlights the importance of moderately small disperse particles/sub-scale phenomena in determining shock profiles and corresponding jump relations.

\begin{figure}[!htbp]
\centering
\includegraphics[scale = 0.4]{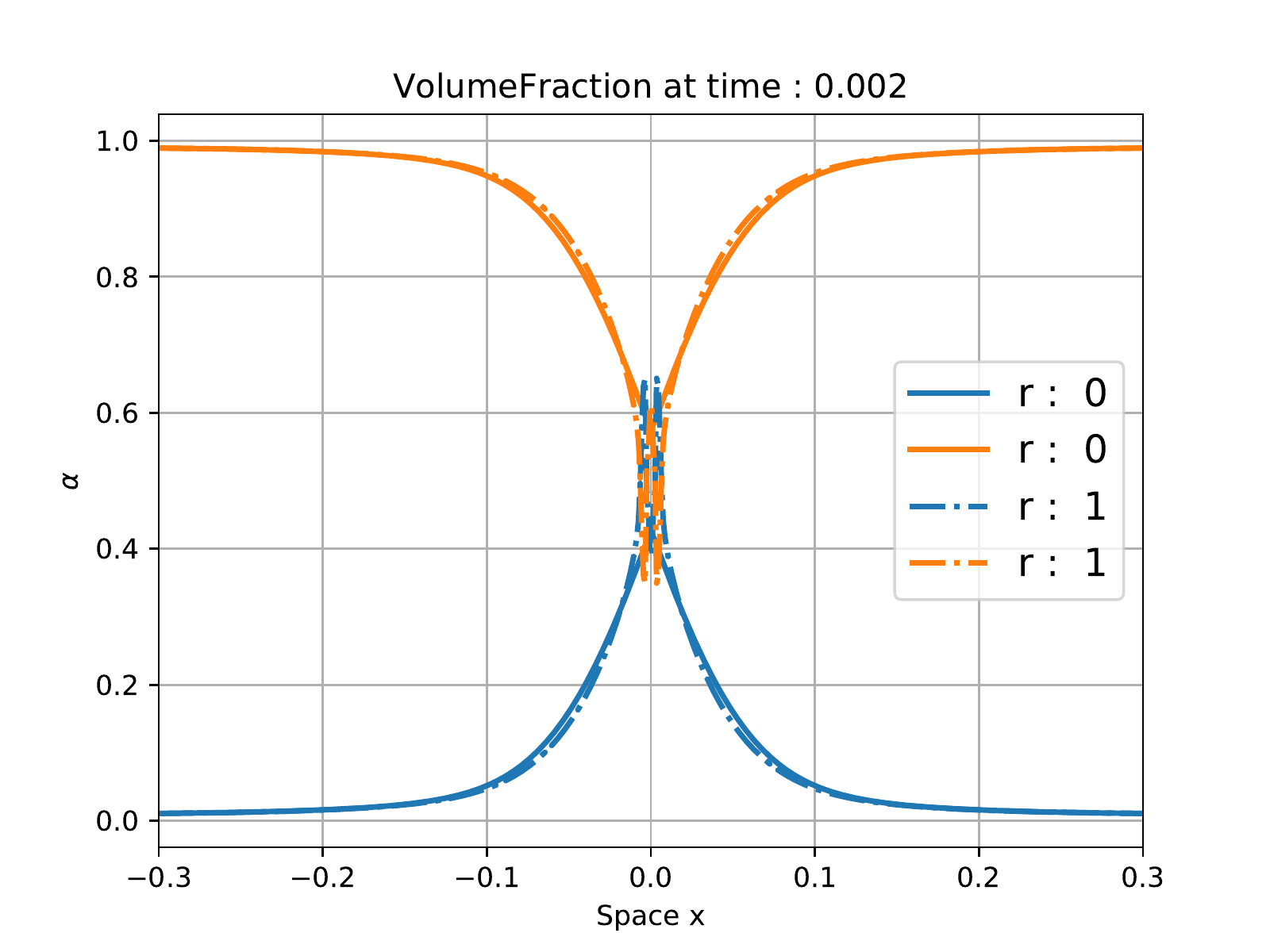}\,
\includegraphics[scale = 0.4]{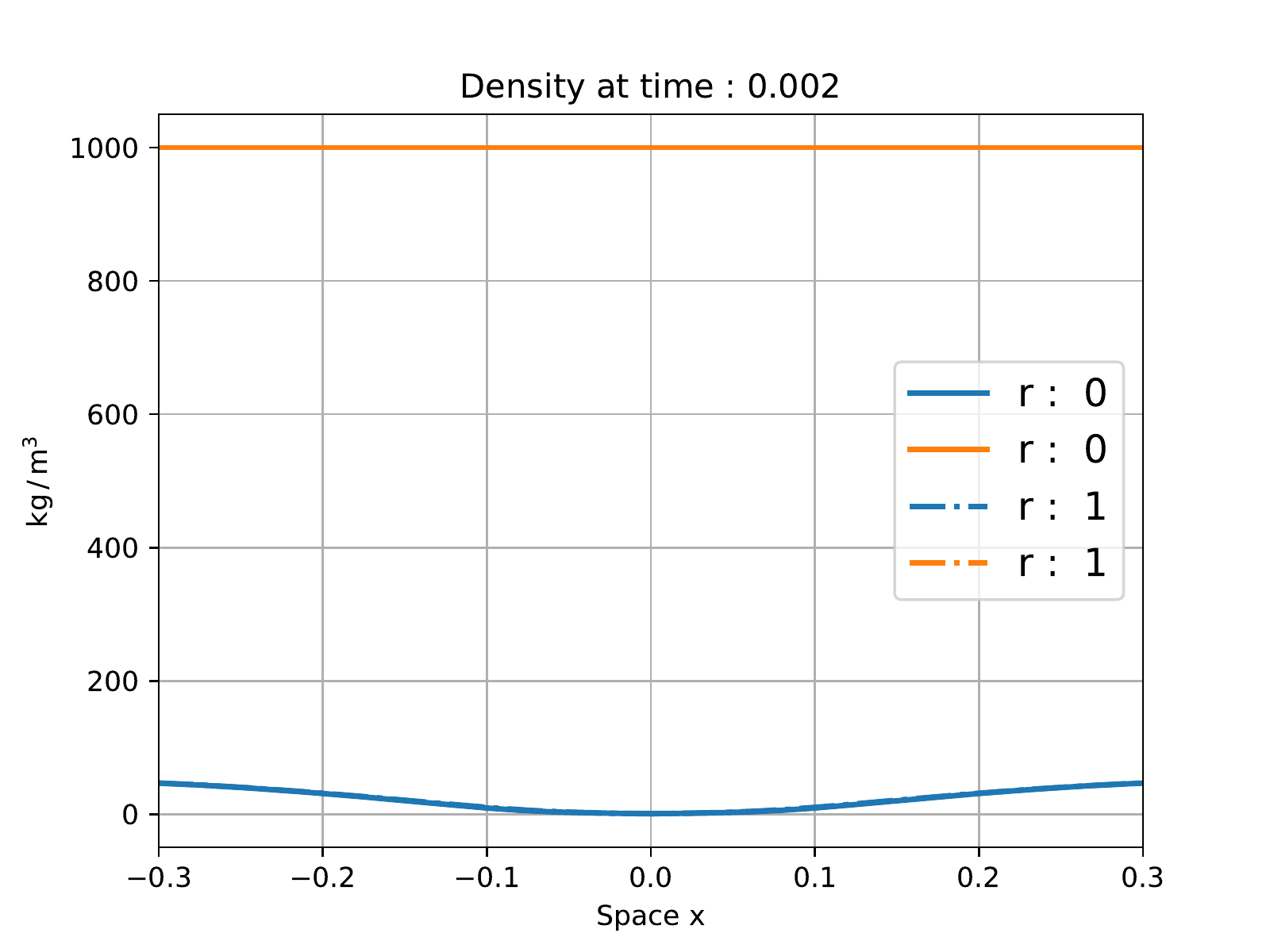}\\
\includegraphics[scale = 0.4]{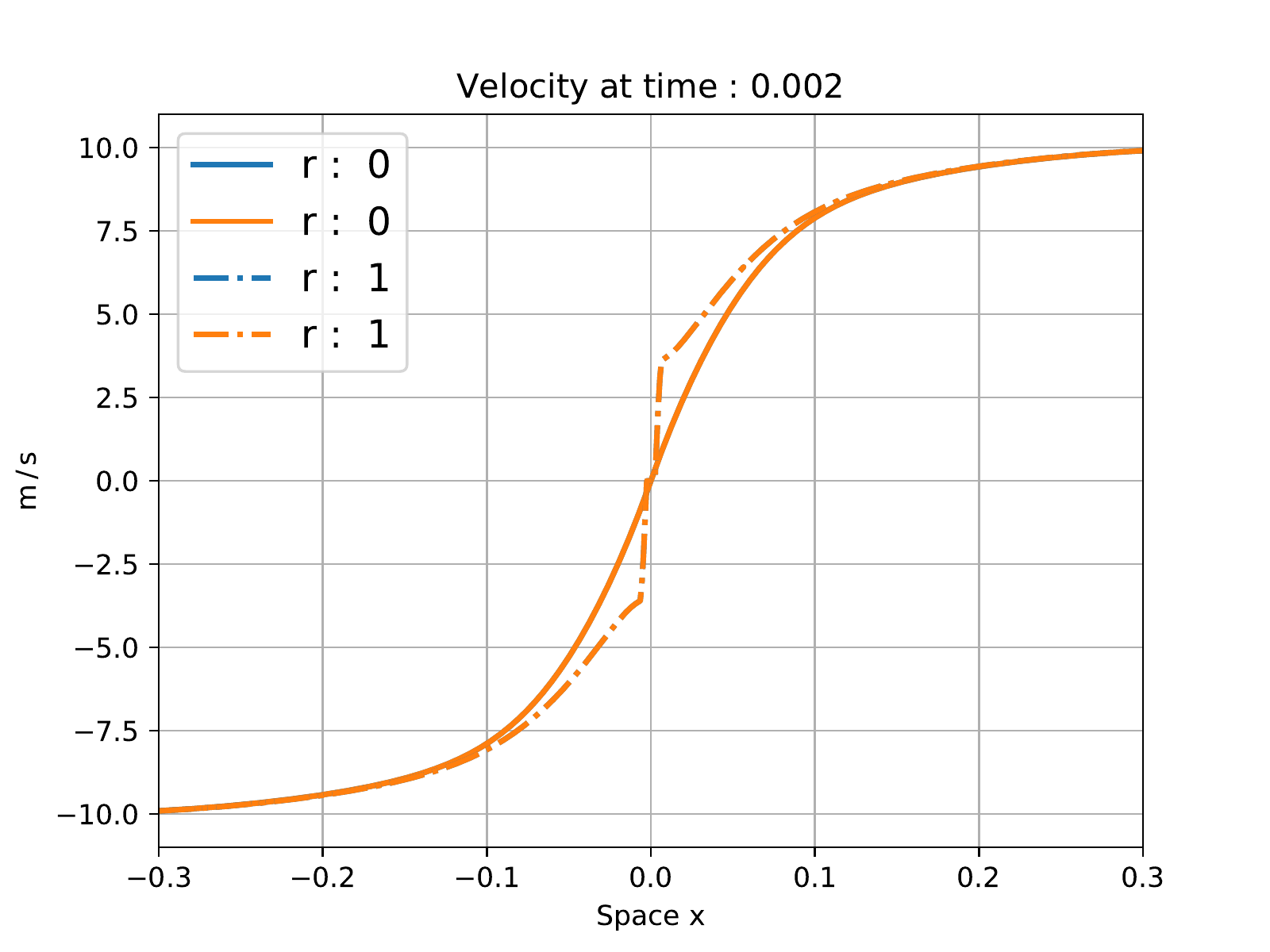}\,
\includegraphics[scale = 0.4]{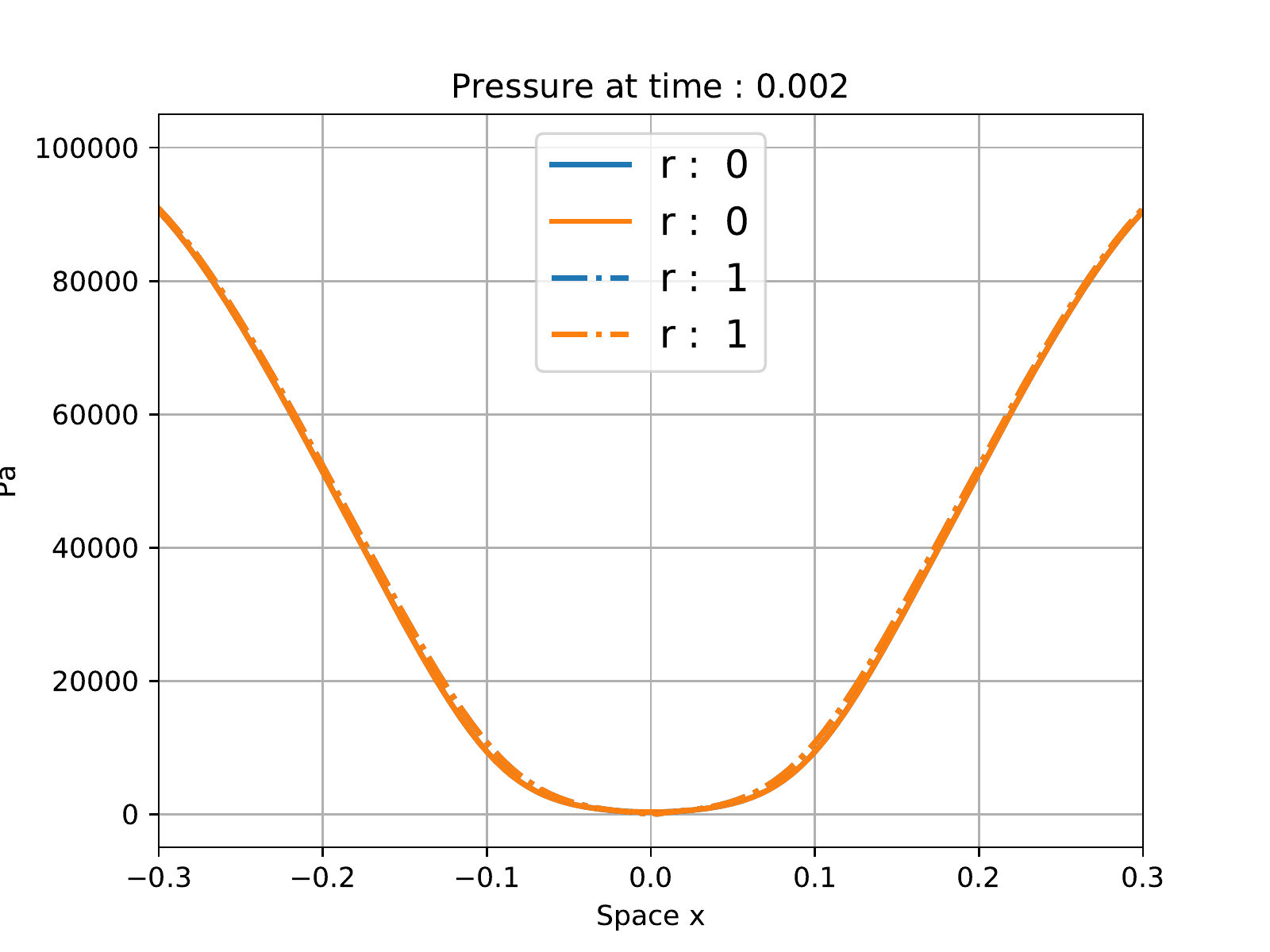}\\
\caption{Cavitation Test problem at time $t = 2\,\mathrm{ms}$ using an infinite drag force. Mixture quantities have been computed using a mesh of $M=2000$ cells for both the stratified ($r=0$) and the disperse ($r = 1$) regimes.}\label{Fig:T4}
\end{figure}

\subsection{Dynamical creation of interfaces}

In this test we examine the capability of the one parameter model to dynamically create interfaces.
We consider the following initial condition in terms of the primitive variables $\textbf{V} = [\alpha,\rho,u,p]$,
\[
\textbf{V}_0(x) = 
\begin{cases}
[\textbf{V}_L^{(1)}, \textbf{V}_L^{(2)}] & \textit{ if }\, x<0,\\
[\textbf{V}_R^{(1)}, \textbf{V}_R^{(2)}] & \textit{ if }\, x>0.
\end{cases}
\]
where
\[
\textbf{V}_L^{(1)} = \begin{bmatrix}
10^{-2}\\
50\\
-10\\
10^5
\end{bmatrix},
\quad
\textbf{V}_L^{(2)} = \begin{bmatrix}
1-10^{-2}\\
1000\\
-10\\
10^5
\end{bmatrix}
\quad
\textbf{V}_R^{(1)} = \begin{bmatrix}
10^{-2}\\
50\\
10\\
10^5
\end{bmatrix},
\quad
\textbf{V}_R^{(2)} = \begin{bmatrix}
1-10^{-2}\\
1000\\
10\\
10^5
\end{bmatrix}
\]
Results for the cavitation test case are reported in Fig. \ref{Fig:T4}, with magnified details shown in Fig. \ref{Fig:T4_m}. Here we present results only for the first relaxation procedure of Appendix \ref{appendix:RelaxationStrategies}.
Both the stratified ($r=0$) and the disperse ($r=1$) regimes are able to dynamically create interfaces, meaning that gas pockets are generated at the discontinuity position. Discrepancies in the velocity field can be appreciated around the discontinuity, like in the oscillations around the peaks of volume fractions.

\begin{figure}[!htbp]
\centering
\includegraphics[scale = 0.47]{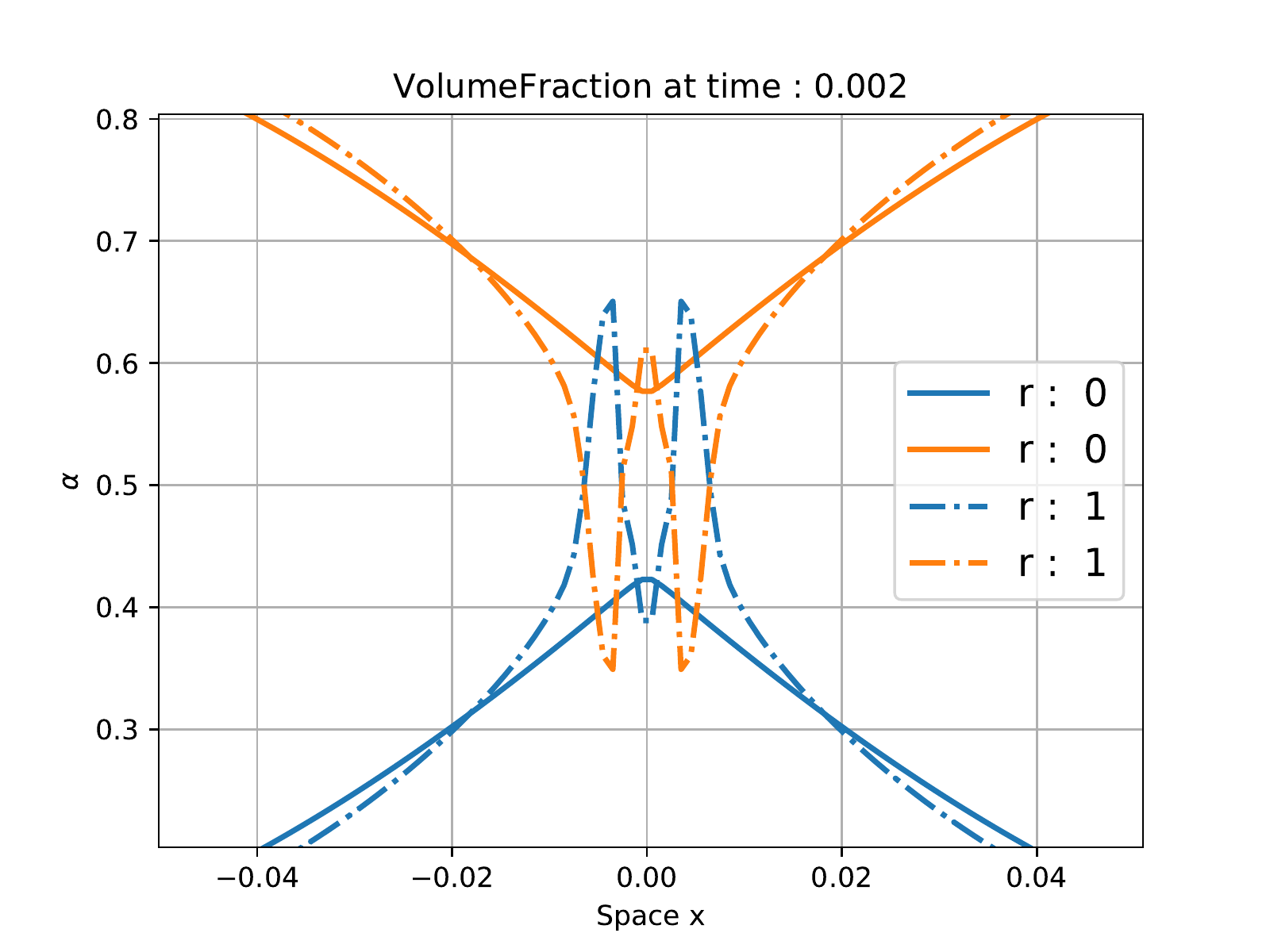}
\caption{Cavitation Test problem at time $t = 2\,\mathrm{ms}$ using an infinite drag force. Magnified details of results reported in Fig. \ref{Fig:T5}}\label{Fig:T4_m}
\end{figure}

It is worth highlighting that this test presents a moderate speed on both sides of the diaphragm. Increasing the expansion velocity (up to $u = 100\,\mathrm{m/s}$, for example) would results in computational failure. Indeed, in the present formulation no mass transfer is considered, so that the creation of gas pockets is only due to the relaxation step.

\subsection{Randomly chosen, spatially dependent regimes}

In this test we want to investigate the difference of predictions with respect to the imposition of randomly chosen $r$, and a piece-wise constant $r$. We perform test $1$, considering uniform volume fraction with first relaxation procedure and compare the results obtained with $r=0$ and $r=1$ with two constant, randomly chosen $r$ and the following piece-wise constant function
\begin{equation}
r = r(x) =
\begin{cases}
0.13 &\textit{if}\qquad -1\leq x < -0.52\\
0.47 &\textit{if}\qquad -0.52\leq x < 0.395\\
1 & \textit{if}\qquad 0.395\leq x < 0.761\\
0.69 & \textit{if}\qquad 0.761\leq x \leq 1\\
\end{cases}
\end{equation}
Results are shown in Fig. \ref{Fig:T5}: we report only details of the quantity of interest to help appreciate differences.

\begin{figure}[!htbp]
\centering
\includegraphics[scale = 0.4]{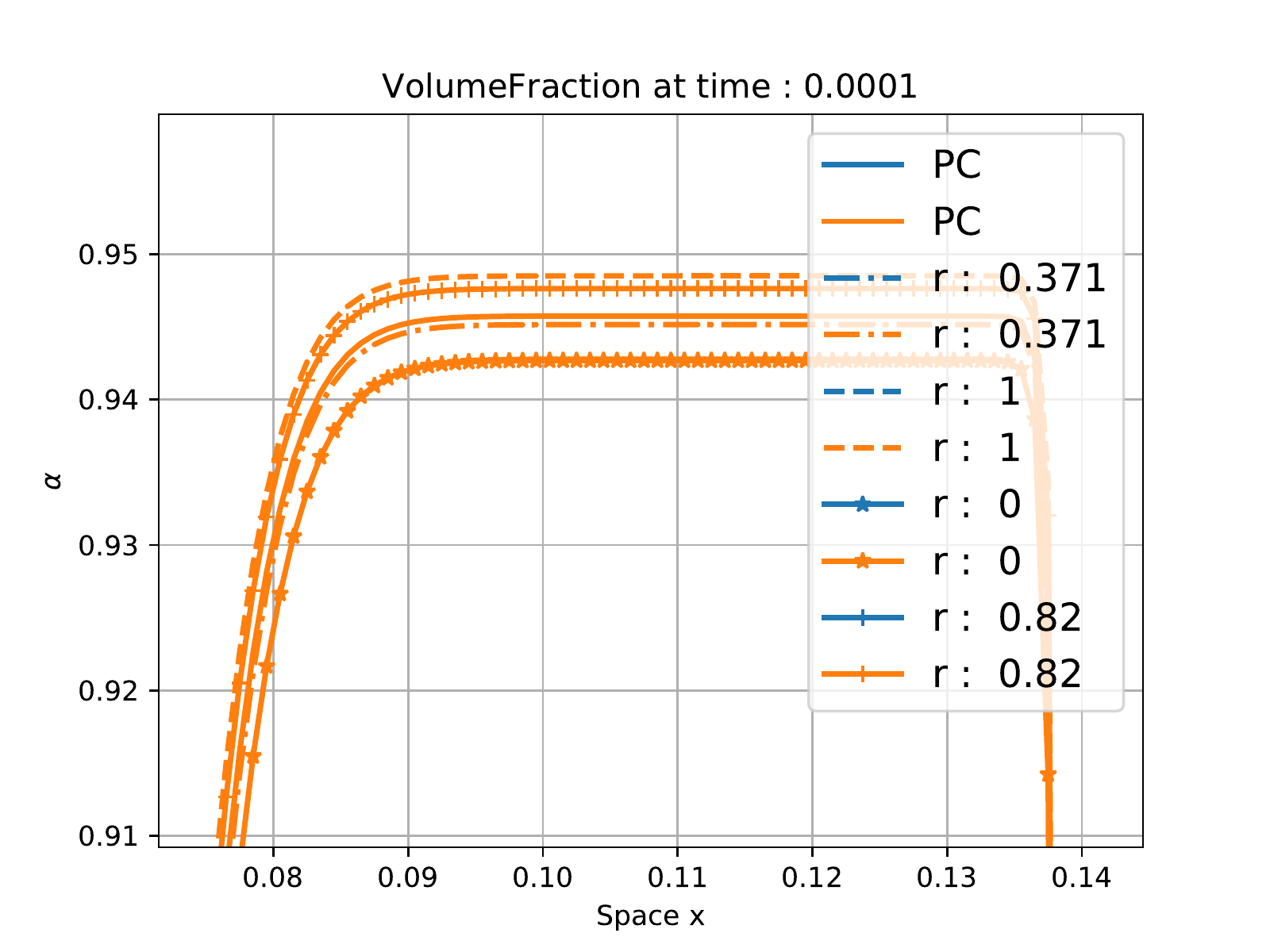}\,
\includegraphics[scale = 0.47]{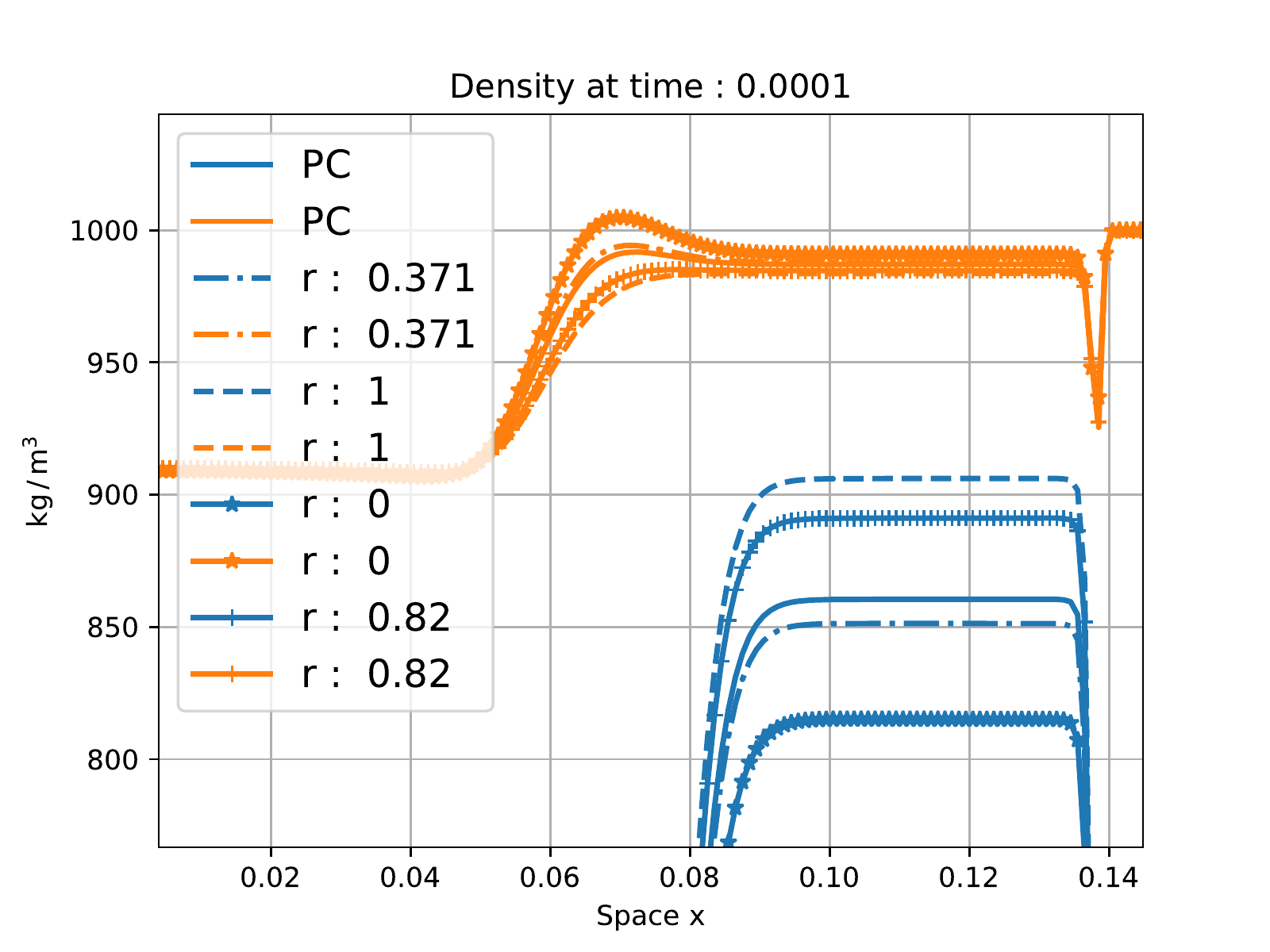}\\
\includegraphics[scale = 0.47]{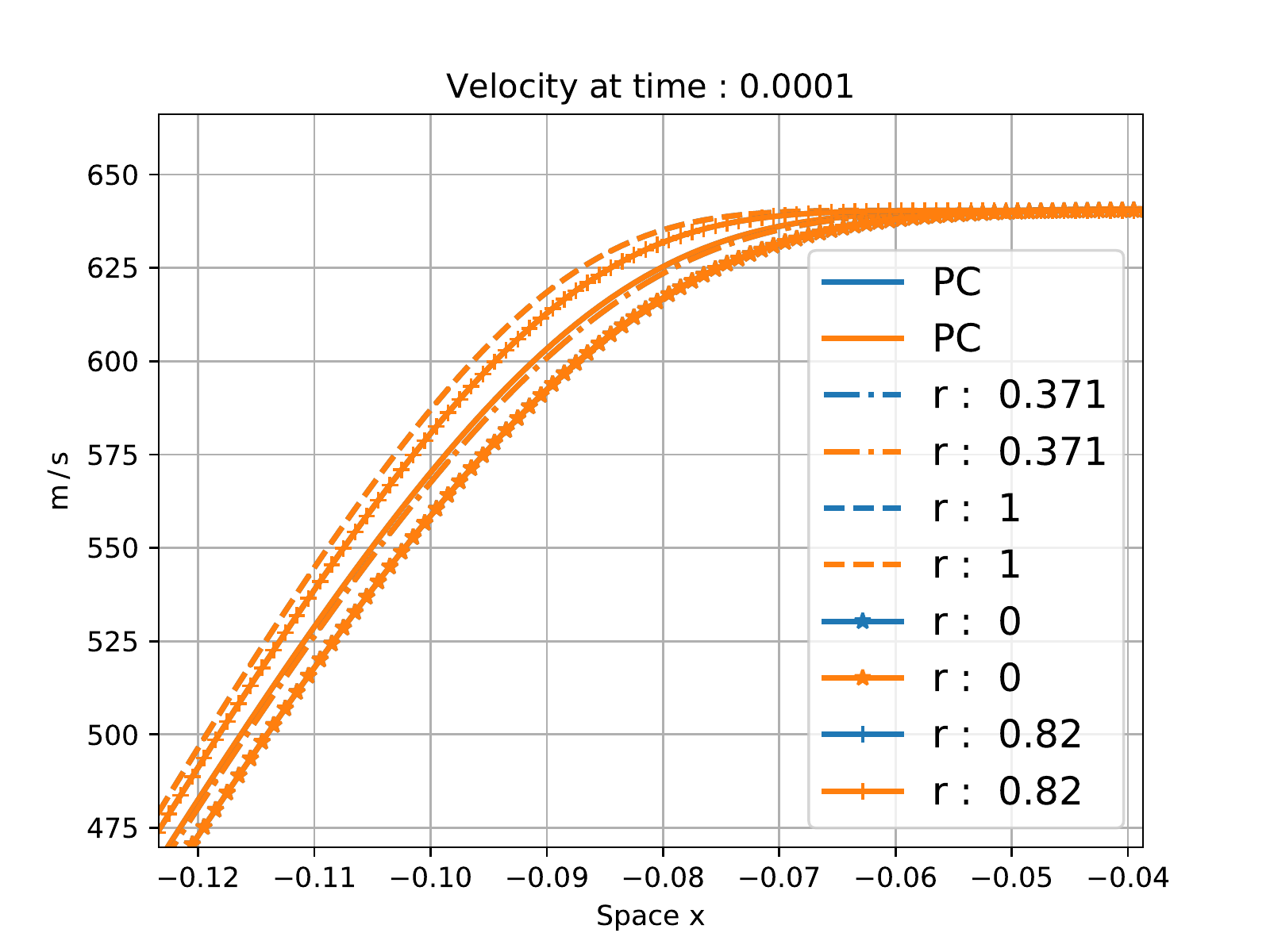}\,
\includegraphics[scale = 0.47]{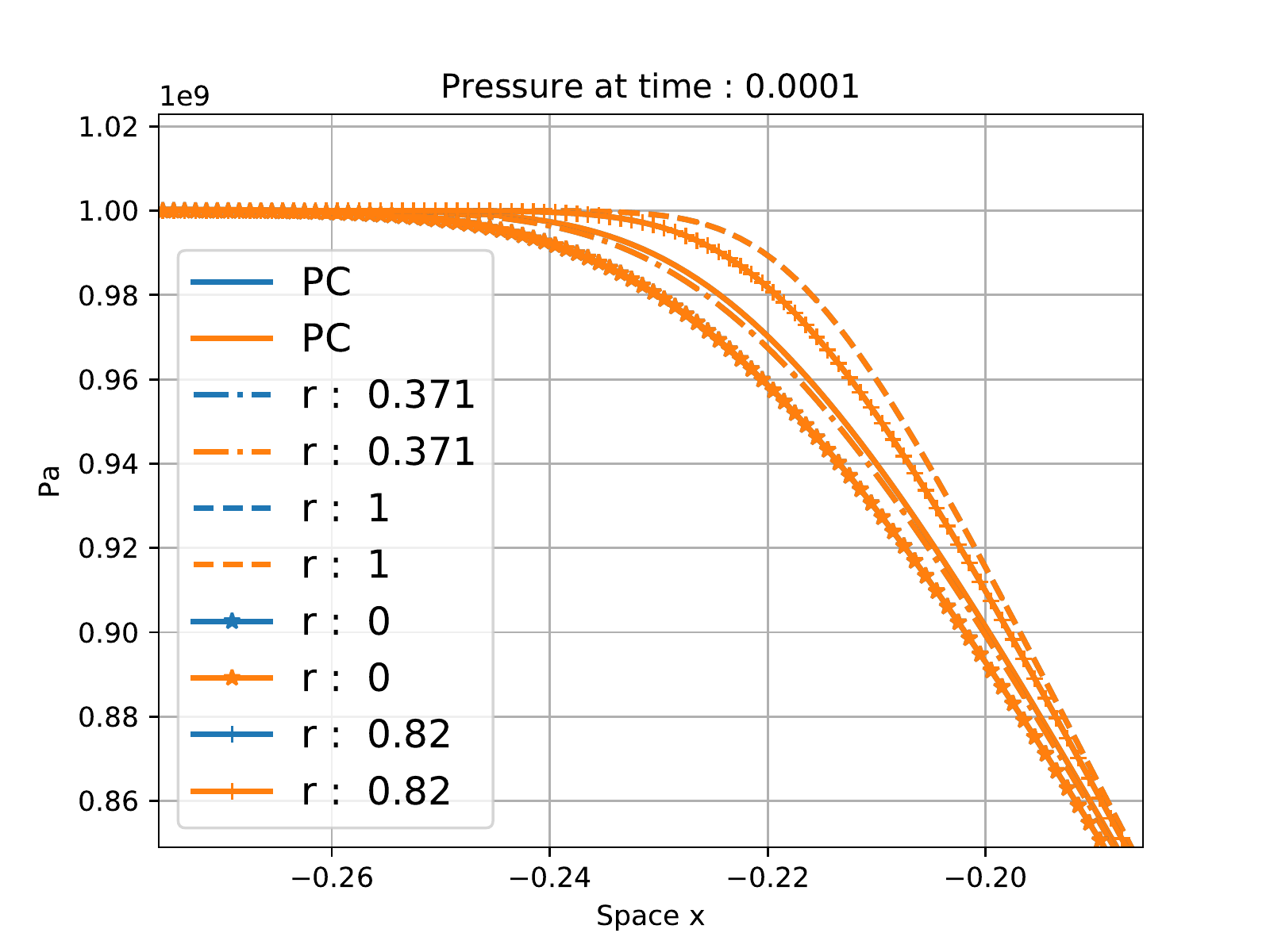}\\
\caption{Uniform volume fraction test at time $t = 100\, \mu\mathrm{s}$ of different flow regimes: fixed topology $r$ and piece-wise constant $r\, :\,\mathrm{PC}$. Numerical solutions for gas phase ($P1$) and the liquid phase ($P2$) computed with $M=2000$ number of cells are reported: volume fractions rarefactions (top left), densities shocks and rarefactions (top right), velocity shocks (bottom left) and pressure rarefactions (bottom right).}\label{Fig:T5}
\end{figure}

This test yields numerical evidence to understand the impact of the choice of the parameter $r$ on the corresponding numerical approximations. Firstly, one can recognize that solutions display smooth transition as $r$ is increased, when constant throughout space. Notice that the same conclusion carries directly to the piece-wise constant function: depending on the domain of interest, the solution generated by the piecewise constant $r$ lies between the ones computed with constant values, underlying the local dependency of the corresponding solutions.
 
\subsection{Dense-to-dilute transition}

In this test we aim at investigating the dependency of solutions with respect to variations in the parameter $r$. Indeed, so far, only spatially constant cases of the parameter $r$ have been considered; here we extend such results to space-time varying functions.\\ 
Firstly, for the sake of comparison, we fix the same initial condition of Test $1$ (i.e. the uniform volume fraction test case), imposing initial stratified flow ($r=0$). For each subsequent time level $t^n$, each of the interfacial regime is modeled updating $r_{i+\frac{1}{2}}^n$ by perturbations of a randomly chosen slight amount. More precisely, for each time level $t^n$ with $n=1, \ldots N$ and at each interface $x=x_{i+\frac{1}{2}}$ do
\begin{enumerate}
\item Produce a uniformly distributed pseudo random number $q(\omega_{i+\frac{1}{2}}^{n})$ between $-1$ and $1$, i.e. $q\sim \mathrm{Unif}[-1,1]$;
\item Perturb the previous flow regime associated to $r^{n-1}_{i+\frac{1}{2}}$ according to
\begin{equation}\label{eps}
r_{i+\frac{1}{2}}^{n,+} = r_{i+\frac{1}{2}}^{n-1} + \epsilon \cdot u(\omega_{i+\frac{1}{2}}^{n})
\end{equation}
for a sufficiently small $\epsilon$.
\item Narrow $r_{i+\frac{1}{2}}^{n,+}$ to the domain of interest:
\begin{equation}
r_{i+\frac{1}{2}}^{n} = f\left(r_{i+\frac{1}{2}}^{n,+}\right) = \begin{cases}
0 &\textit{if}\quad r_{i+\frac{1}{2}}^{n,+}<0 \\
r_{i+\frac{1}{2}}^{n,+} & \textit{if}\quad r_{i+\frac{1}{2}}^{n,+}\in [0,1]\\
1 & \textit{if}\quad r_{i+\frac{1}{2}}^{n,+}>1
\end{cases}
\end{equation}
\end{enumerate}

\begin{figure}[!htbp]
\centering
\includegraphics[scale=0.47]{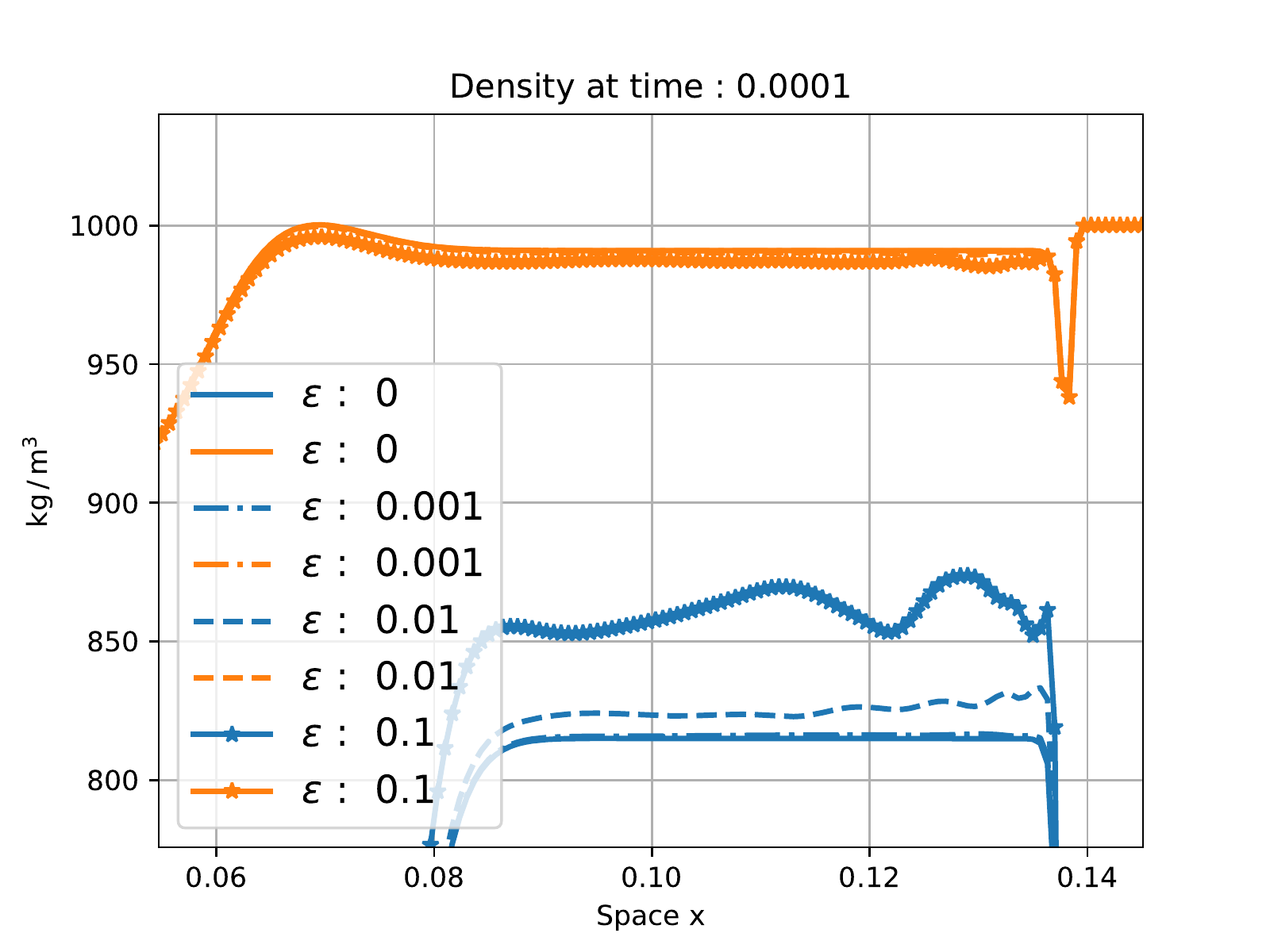}\,
\includegraphics[scale=0.47]{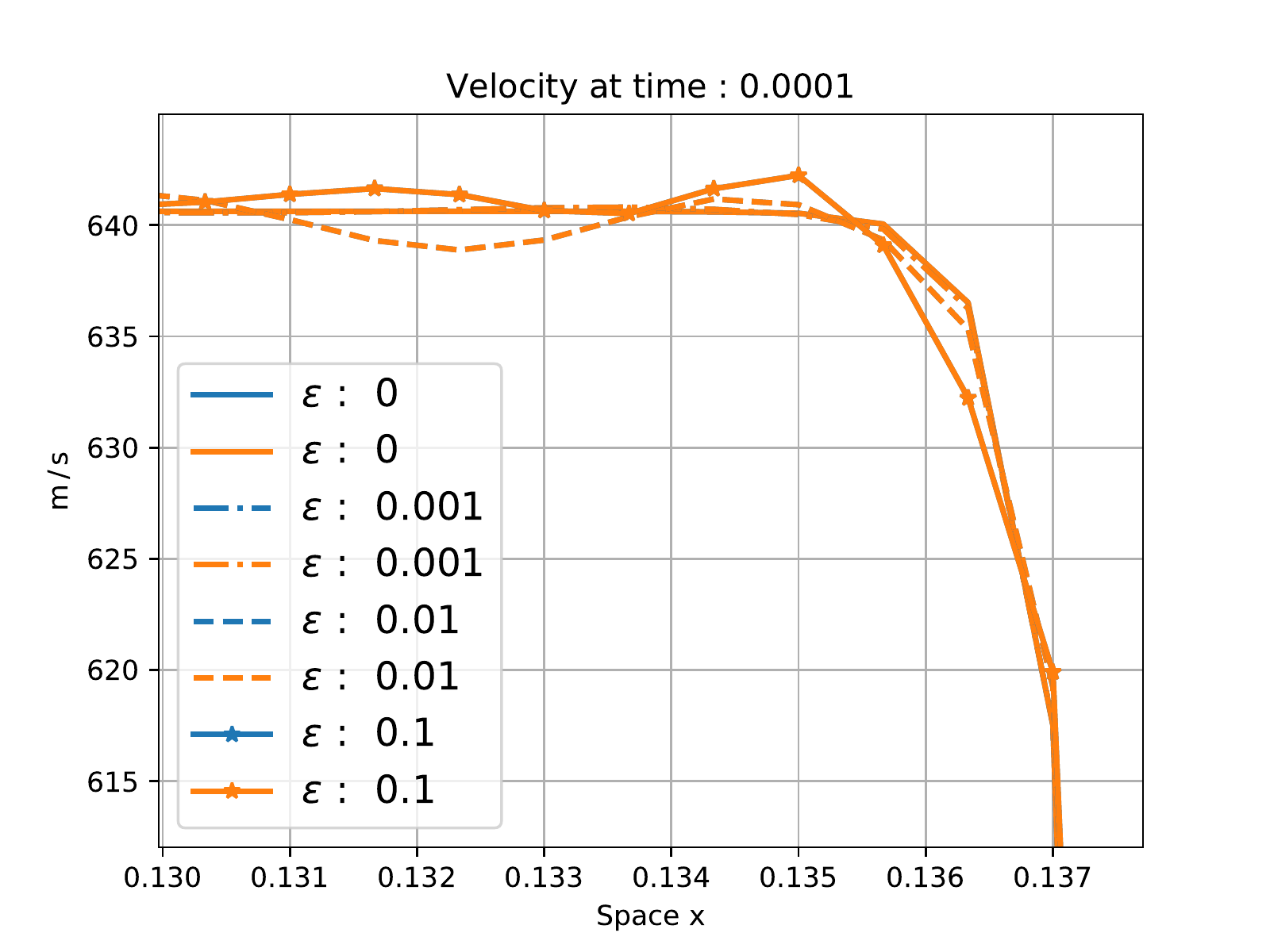}\\
\includegraphics[scale=0.47]{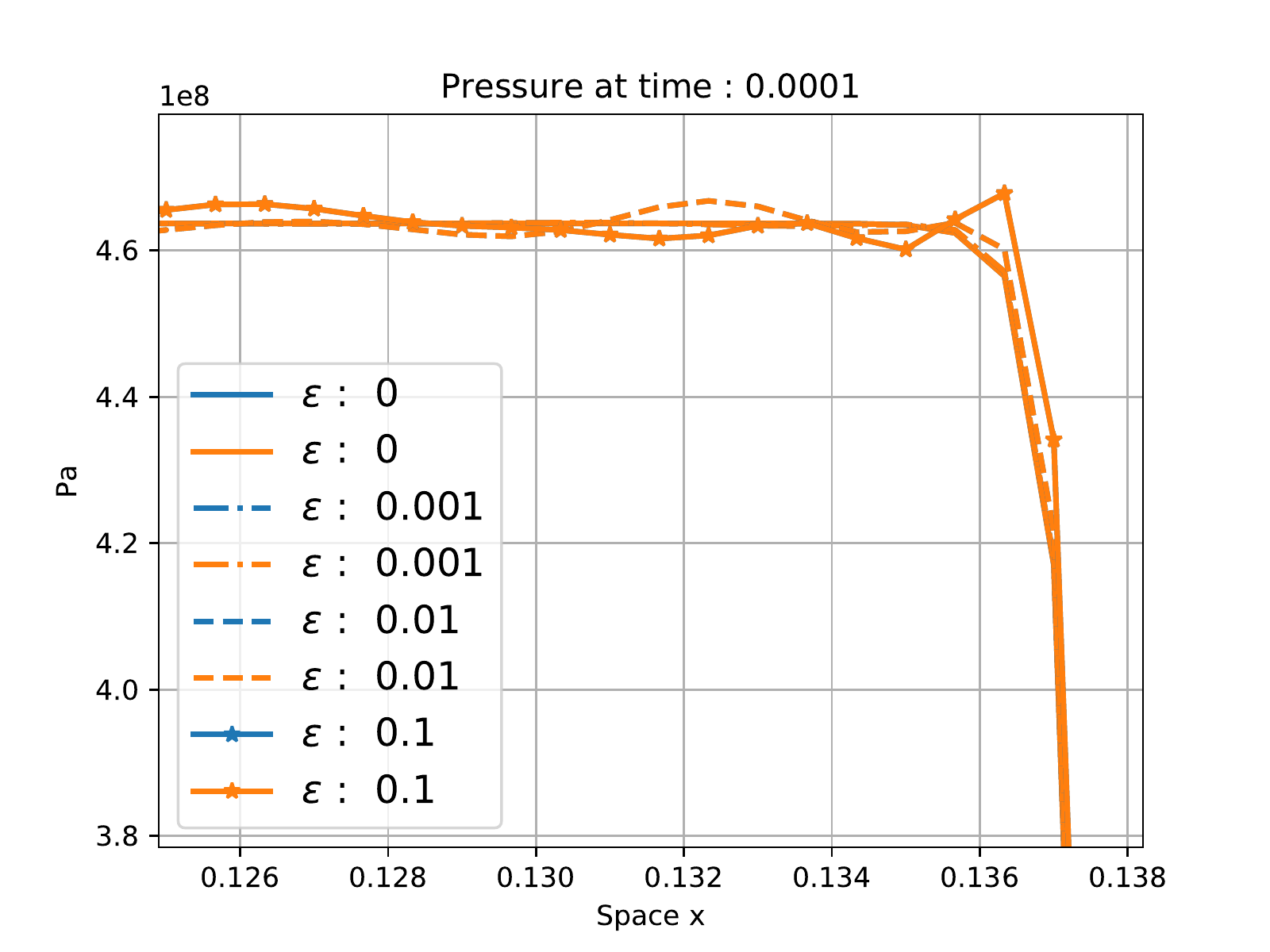}
\caption{Dense-to-dilute transition test at time $t = 100\, \mu\mathrm{s}$. Numerical solutions have been computed using a uniform mesh of $M=3000$ cells, using different values of $\epsilon$ in (\ref{eps}). Magnified details of results are displayed distinguishing the gas ($\mathrm{P1}$) and the liquid ($\mathrm{P1}$) phase, as well as the different values of $\epsilon$.}\label{Fig:T6}
\end{figure}

Once all the newly generated $r_{i+\frac{1}{2}}^n$ are computed, one updates the solution by utilizing the strategy designed in Section \ref{sec:NumStrat}. The parameter $\epsilon$  encodes the rate of dense-to-dilute transition: negligible values of the parameter $\epsilon$ will produce virtually same results of the stratified flow case, whereas excessively large values will produce discontinuous flow transition. We present in Fig. \ref{Fig:T6} only magnified regions of the approximate solutions: 
results are analogous to the ones reported in Fig. \ref{Fig:T3}, with oscillations at post-shock states.
Notice how moderately small values of $\varepsilon$ do yield virtually same results as constant $r\equiv 0$. As suggested by Proposition \ref{prop:r}, a convergence towards the stratified flow ($r \equiv 0$) can be appreciated as $\epsilon$ decreases, hence demonstrating the smooth dependency at any time of computed solutions with respect to the parameter $r$.\\
However, the most interesting result of this test is the oscillatory effect appearing for sufficiently large values of $\epsilon$. In order to further investigate such oscillatory effect appearing near discontinuities, we run the same test, increasing $\epsilon$ and comparing results with a uniformly randomly chosen $r_{i+\frac{1}{2}}^n$. Corresponding results are shown in Fig. \ref{Fig:T7}. Oscillatory effects appear near discontinuities which propagate to adjacent regions. As expected, increasing the value of $\varepsilon$ will produce with higher probabilities results oscillating around $r\equiv 1$. Corresponding numerical solutions show indeed variations around the mentioned value $r\equiv 1$.

\begin{figure}[!htbp]
\centering
\includegraphics[scale=0.45]{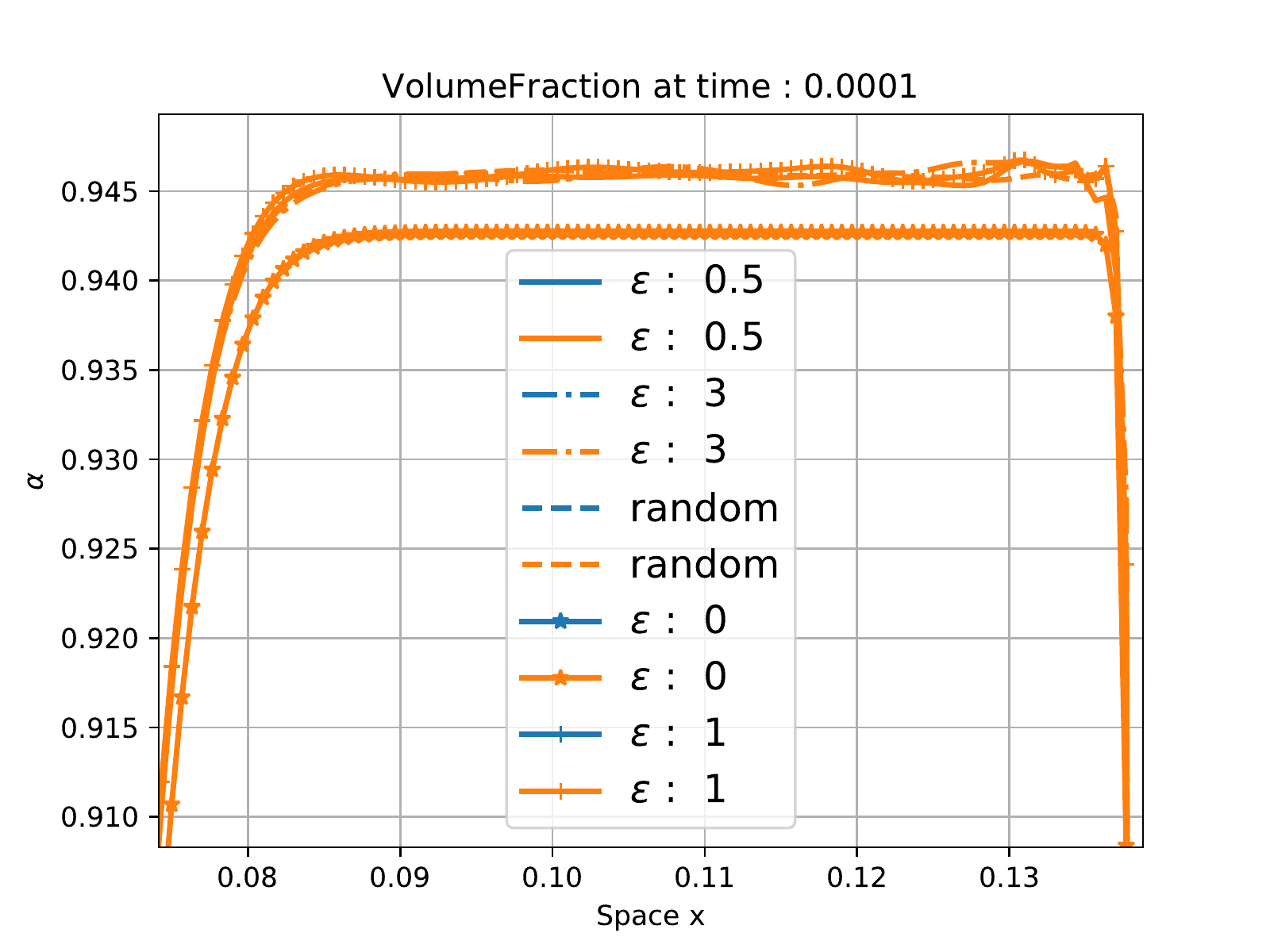}\,
\includegraphics[scale=0.45]{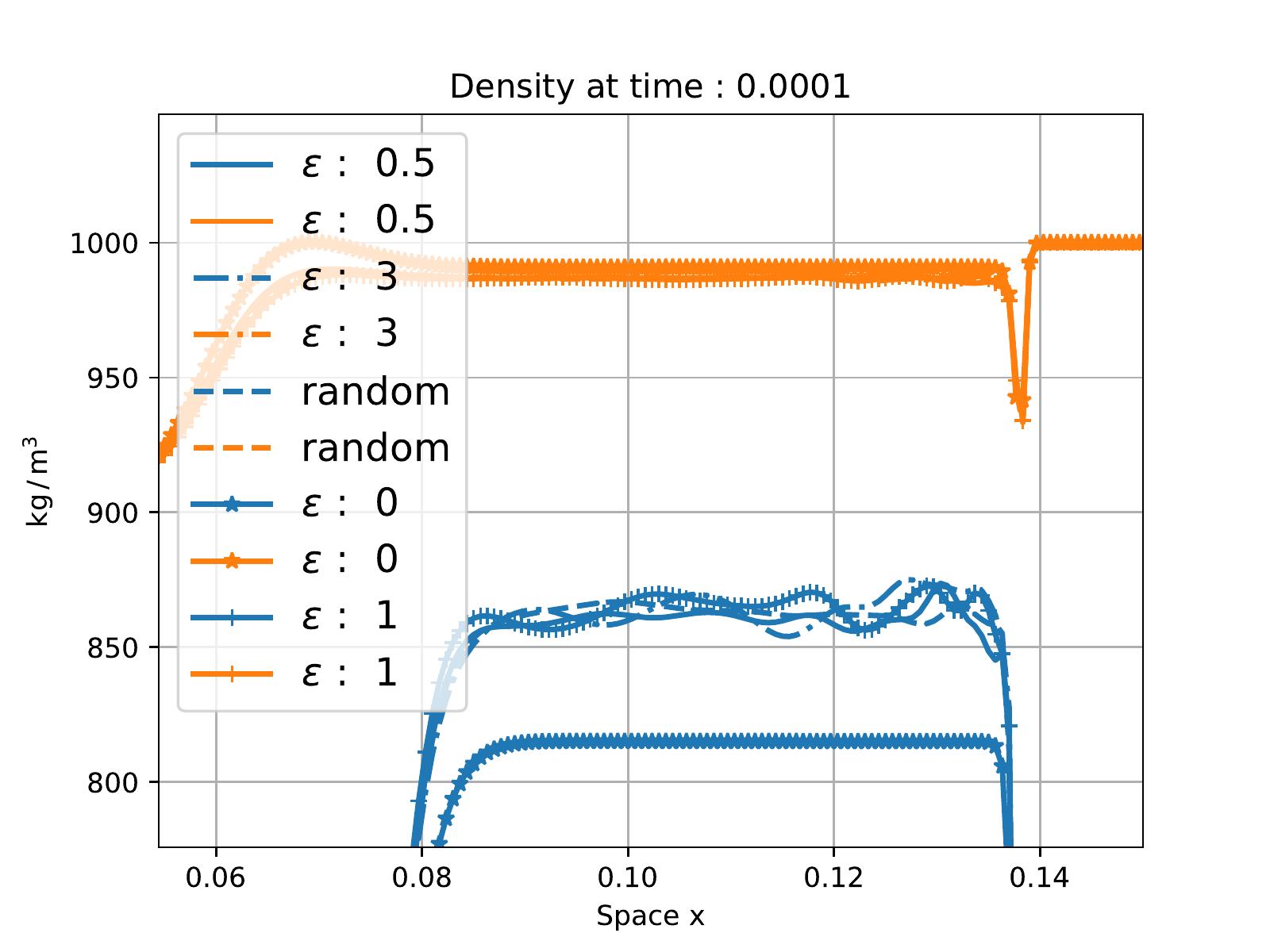}\\
\includegraphics[scale=0.45]{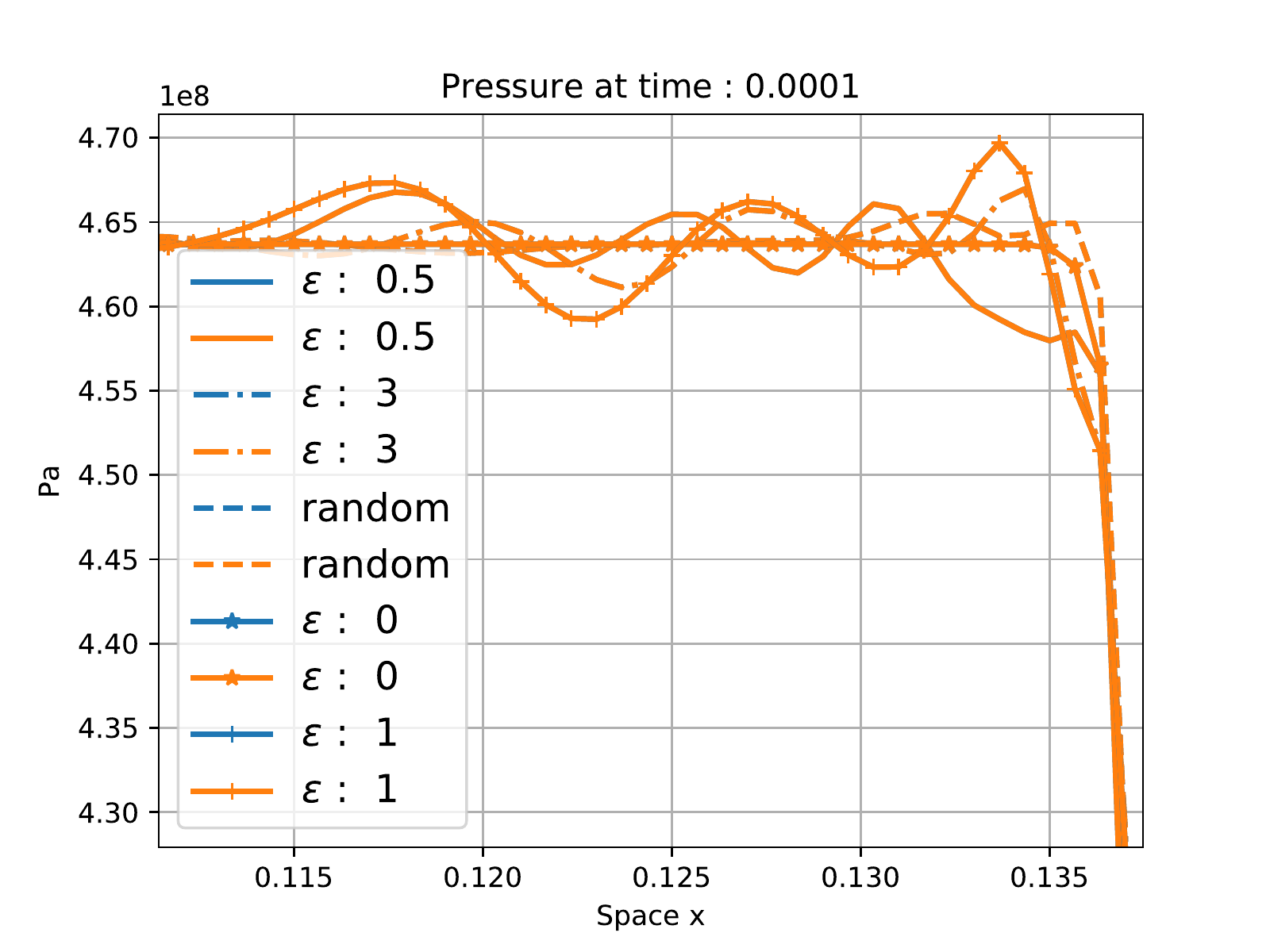}\,
\includegraphics[scale=0.45]{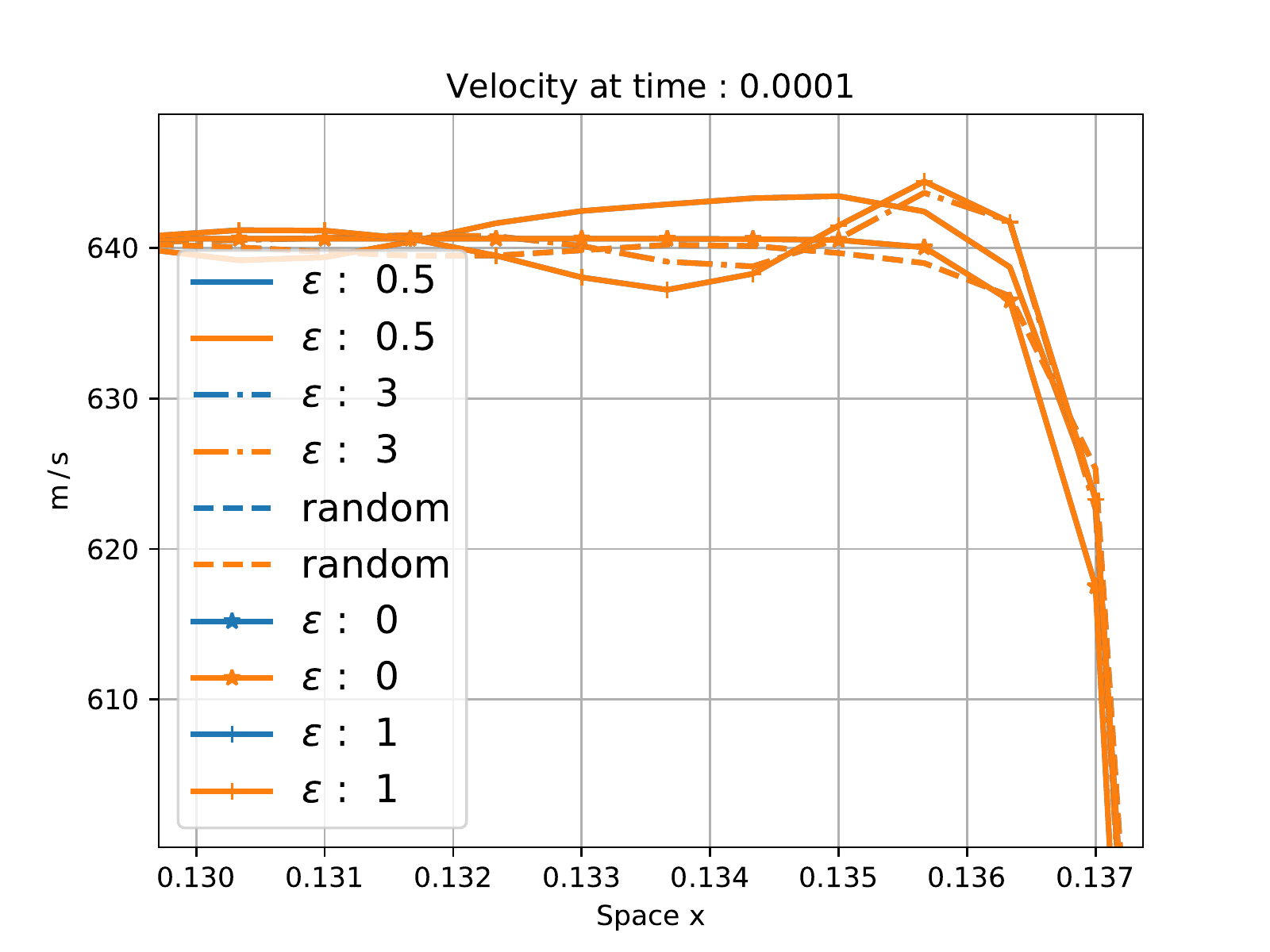}
\caption{Dense-to-dilute transition test at time $t = 100\, \mu\mathrm{s}$. Numerical solutions have been computed using a uniform mesh of $M=3000$ cells, using different values of $\epsilon$ in (\ref{eps}). Magnified details of results are displayed distinguishing the gas ($\mathrm{P1}$) and the liquid ($\mathrm{P1}$) phase, as well as the different values of $\epsilon$.}\label{Fig:T7}
\end{figure}

\section{Discussion}
The modeling of multi-phase flow is very challenging, given the range of scales as well as type of distinct flow regimes that one encounters in this context. We revisit the discrete equation method (DEM) for two-phase flow in the absence of heat conduction and mass transfer. As DEM is based on an ensemble averaging of flow realizations, the mean flow has the potential to describe different two-phase flow regimes. Our starting point was the derivation of Abgrall and Saurel \cite{Abgrall&Saurel} where the authors proposed a DEM for two-phase flow. Our main contributions in this paper was to carefully analyze the resulting probability coefficients and to prove local convexity for them. This rigorously establishes that this version of DEM can indeed model different flow regimes ranging from the disperse to stratified (or separated) flow. Moreover, we reformulated the the resulting mesoscopic model in terms of an one-parameter family of PDEs that interpolates between different regime. The limit cases of this parameter correspond to disperse and stratified flow, respectively. Furthermore, two sets of relaxation procedures were also proposed to enforce relaxation to equilibrium. 

We presented extensive numerical experiments to describe the capabilities as well as limitations of the proposed DEM. First, we demonstrated that different values of the probability coefficients yield different predictions on the mechanical interaction between phases. Indeed, it is the probability coefficients, rather than relaxation terms, that lead to this behavior, rather than the details of the relaxation terms which serve to enforce thermodynamic constraints. Indeed, the interaction of phases demonstrated through numerical tests \emph{without relaxation procedures}, suggest mechanical exchange even if no sub-particles are present inside each volume, in contrast to the interpretation of \cite{Abgrall&Saurel}. This point of view clearly brings out the complimentary roles played by the probability coefficients and relaxation terms in DEM. 

The proposed formulation also brings out possible limitations of the DEM approach. In particular, we show that an infinite number of possible models can be constructed, resulting in a ill-posed procedure to construct multiphase simulations. Although several works have investigated the mechanical/thermodynamical consistency of the continuous limit associate to stratified flow, proving it to lead to physically meaningful models. However, even under such an ansatz, the DEM method requires the relaxation operator to be added manually assuming either an infinite drag force or an estimate of the interfacial area in each cell. These latter may become problematic to obtain, if possible, without making any assumption on the flow regime. Indeed, among the desiderata for multiphase flow simulations, the avoidance of user-specification of the flow regime is paramount. 

We observe that the DEM scheme represents a finer level of description as compared to the continuum theory approach. In turn, such strategy achieves extensive modelling capabilities, ranging from stratified to disperse flows. However, as demonstrated in this paper, such a mesoscopic approach is not yielding a fully-determined system of constitutive equations as neither equilibrium states nor probability coefficients are uniquely defined. Indeed, this is due to the determination of mean flow variables whereas information about the underlying microstructure is lost. 
Such inherent under-determination does not rendering the model invalid, but it rather requires closure conditions to be supplied. 
This is equivalent to saying that the microstructure details lost in the passage to the ensemble averages has to be recovered from somewhere, which in the case of the DEM, is embodied in the probability coefficients and in the relaxation terms. In other words, one needs to adapt the free parameters according to the flow topology, but contrarily to other approaches, it is easy to see where new inputs must be supplied.

Moreover, from a mathematical point of view, the underlying probability coefficients in our DEM scheme can be interpreted as two-point correlation measures. The resulting conclusion is that the DEM approach lacks information about correlation measures, at each space-time location. Therefore, it is out belief that the point of view of measures should be preferred over classical weak forms. Indeed, many recent works dealing with numerical approximations of turbulent flow have shown success of weaker notions than the usual distributional sense. 
Notice that such an approach is in principle also capable to deal with non-conservative products, typically featured by most well-known two-phase models.

Our results also show that the form of the relaxation variety does not depend on the underlying solver for the hyperbolic step, thus suggesting that a characteristic feature of such phenomena is the determination of the speed at which they reach equilibrium.

Such an insight confirms that the essence of multiphase fluids lies in their microstructure, which \emph{has} to be considered in order to characterize mean flow variables. For this reasons, forthcoming papers aim at including such information in the modeling of multiphase flow.

\newpage
\appendix

\section{Proof of local convexity}
\label{appendix:Proof}

We start the proof of our main result Theorem \ref{Thm:Convex} by reporting the following trivial fact: for any $a,b\in\mathbb{R}$ it holds 
\begin{equation}\label{fact}
\max(a-b,0) + \min(a,b) = a
\end{equation}

Furthermore, we will adopt the notation $p\neq q \in \lbrace 1,2\rbrace$.

\begin{proof}(\textit{of Proposition \ref{Prop:NewProbs}})\\
Notice that
\begin{equation*}
\Bigg\lbrace X^{(p)}(x_{i+\frac{1}{2}}^-) = 1, X^{(p)}(x_{i+\frac{1}{2}}^+) = 0 \Bigg\rbrace = \Bigg\lbrace X^{(p)}(x_{i+\frac{1}{2}}^-) = 1\Bigg\rbrace \cap \Bigg \lbrace X^{(p)}(x_{i+\frac{1}{2}}^+) = 0 \Bigg \rbrace
\end{equation*}
Hence,
\begin{align*}
\mathcal{P}_{i+\frac{1}{2}}\left[\Sigma_p,\Sigma_q\right] &= \mathcal{P}_{i+\frac{1}{2}}\left[X^{(p)}(x_{i+\frac{1}{2}}^-) = 1, X^{(p)}(x_{i+\frac{1}{2}}^+) = 0\right]\\
&\leq \mathcal{P}_{i+\frac{1}{2}}\left[X^{(p)}(x_{i+\frac{1}{2}}^-) = 1\right] = \mathcal{P}_{i+\frac{1}{2}}\left[\Sigma_{p},\Sigma_{p}\right] + \mathcal{P}_{i+\frac{1}{2}}\left[\Sigma_{p},\Sigma_{q}\right] \,\myeq{(\ref{Condi})}\, \alpha_{i}^p 
\end{align*}
Similarly,
\begin{align*}
\mathcal{P}_{i+\frac{1}{2}}\left[\Sigma_p,\Sigma_q\right] \leq \mathcal{P}_{i+\frac{1}{2}}\left[X^{(p)}(x_{i+\frac{1}{2}}^+) = 0\right] = \mathcal{P}_{i+\frac{1}{2}}\left[\Sigma_{q},\Sigma_{q}\right] + \mathcal{P}_{i+\frac{1}{2}}\left[\Sigma_{p},\Sigma_{q}\right] \,\myeq{(\ref{Condi+1})}\, \alpha_{i+1}^q 
\end{align*}
Thus, condition (\ref{Condmin2}) follows. In turn, this latter into (\ref{Condi}) yields (\ref{Condmax2}):
\begin{equation*}
\mathcal{P}_{i+\frac{1}{2}}\left[\Sigma_p,\Sigma_p\right] \geq \alpha_i^p - \min(\alpha^p_i,\alpha^q_{i+1})\, \myeq{(\ref{fact})}\, \max\left(\alpha^p_{i}-\alpha^q_{i+1},0\right)
\end{equation*}
Finally, let us prove that, under the saturation condition (\ref{vol}), the probability pair 
$$\mathfrak{P}^1_{i+\frac{1}{2}}=\Bigg (\mathcal{P}^1_{i+\frac{1}{2}}\left[\Sigma_p,\Sigma_p\right],\mathcal{P}^1_{i+\frac{1}{2}}\left[\Sigma_p,\Sigma_q\right]\Bigg):=\Bigg( \max\left(\alpha^p_{i}-\alpha^q_{i+1},0\right), \min\left(\alpha^p_i,\alpha^q_{i+1}\right) \Bigg)$$
verify (\ref{ConsistencyConds}). Indeed,
\begin{itemize}
\item \underline{Condition (\ref{Condi})} : 
\begin{align*}
&\mathcal{P}^1_{i+\frac{1}{2}}\left[\Sigma_p,\Sigma_p\right] + \mathcal{P}^1_{i+\frac{1}{2}}\left[\Sigma_p,\Sigma_q\right] = \max\left(\alpha^p_{i}-\alpha^q_{i+1},0\right) + \min\left(\alpha^p_i,\alpha^q_{i+1}\right) \myeq{(\ref{fact})} \alpha^p_{i}
\end{align*}
\item \underline{Condition (\ref{Condi+1})} : 
\begin{align*}
&\mathcal{P}^1_{i+\frac{1}{2}}\left[\Sigma_p,\Sigma_p\right] + \mathcal{P}^1_{i+\frac{1}{2}}\left[\Sigma_q,\Sigma_p\right] = \max\left(\alpha^p_{i}-\alpha^q_{i+1},0\right) + \min\left(\alpha^p_{i+1},\alpha^q_i\right)\\
& \myeq{(\ref{fact})} \max\left(\alpha^p_{i}-\alpha^q_{i+1},0\right) + \alpha^p_{i+1} - \max\left(\alpha^p_{i+1}-\alpha^q_{i},0\right)\\
& \myeq{(\ref{vol})} \max\left(\alpha^p_{i}- 1 + \alpha^p_{i+1},0\right) + \alpha^p_{i+1} - \max\left(\alpha^p_{i+1}- 1 + \alpha^p_{i},0\right) = \alpha^p_{i+1}
\end{align*}
\item \underline{Condition (\ref{Condmin})} :
\begin{equation*}
\begin{split}
\mathcal{P}^1_{i+\frac{1}{2}}\left[\Sigma_p,\Sigma_p\right] &= \max(\alpha^p_{i}-\alpha^q_{i+1},0) \,\myeq{(\ref{vol})} \, \max(\alpha^p_{i}-1+\alpha^p_{i+1},0)\\
&=\max \Bigg( \min(\alpha_i^p,\alpha^p_{i+1}) \underbrace{ -1 + \max(\alpha_i^p,\alpha^p_{i+1}) }, 0\Bigg) \leq \max \Bigg( \min(\alpha_i^p,\alpha^p_{i+1}), 0\Bigg) \\& \myeq{(\ref{vol})} \min(\alpha_i^p,\alpha^p_{i+1})
\myleq{(\ref{vol})} 0
\end{split}
\end{equation*}
\item \underline{Condition (\ref{Condmax})} : 
\begin{equation*}
\begin{split}
\mathcal{P}^1_{i+\frac{1}{2}}\left[\Sigma_p,\Sigma_q\right] &= \min(\alpha_i^p,\alpha^q_{i+1}) \myeq{(\ref{vol})} \max \Bigg( \min(\alpha_i^p,\alpha^q_{i+1}), 0\Bigg)\\
& \mygeq{(\ref{vol})} \max \Bigg( \min(\alpha_i^p,\alpha^q_{i+1}) -1 + \max(\alpha_i^p,\alpha^q_{i+1}) , 0\Bigg)\\
&= \max(\alpha_i^p -1 + \alpha^q_{i+1},0) = \max(\alpha^p_{i}-\alpha^p_{i+1},0)
\end{split}
\end{equation*}
\end{itemize}
\end{proof}

\begin{remark}\label{Rem:Rel}
Notice that the combination of the consistency conditions (\ref{ConsistencyConds}) and (\ref{NewProbs}) leads to the following relations
\begin{subequations}
\begin{align}
\max\left(\alpha^p_{i}-\alpha^q_{i+1},0\right) &\leq \mathcal{P}_{i+\frac{1}{2}}\left[\Sigma_p, \Sigma_p\right] \leq \min\left(\alpha_i^p,\alpha_{i+1}^p\right)\label{Rel-pp}\\
\max\left(\alpha_i^p-\alpha^p_{i+1},0\right)&\leq \mathcal{P}_{i+\frac{1}{2}}\left[\Sigma_p, \Sigma_q\right] \leq \min\left(\alpha^p_i,\alpha^q_{i+1}\right)\label{Rel-pq}
\end{align}
\end{subequations}
\end{remark}

\begin{proof}(\textit{of Theorem \ref{Thm:Convex}})\\
We split the proof into several steps.
\begin{enumerate}
\item \underline{Existence of $r\in [0,1]$, verifying (\ref{Convex-pp})-(\ref{Convex-pq})}: From (\ref{Rel-pq}) we get
\begin{equation}\label{almost_r}
0\leq \mathcal{P}_{i+\frac{1}{2}}\left[\Sigma_p,\Sigma_q\right] - \max\left(\alpha_i^p-\alpha^p_{i+1},0\right) \leq \min\left(\alpha^p_i,\alpha^q_{i+1}\right) - \max\left(\alpha_i^p-\alpha^p_{i+1},0\right)
\end{equation}
Notice that if $\min\left(\alpha^p_i,\alpha^q_{i+1}\right) = \max\left(\alpha_i^p-\alpha^p_{i+1},0\right)$, then $\mathcal{P}_{i+\frac{1}{2}}\left[\Sigma_p,\Sigma_p\right] = \max\left(\alpha_i^p-\alpha^p_{i+1},0\right)$, which in turn implies that $\mathfrak{P}_{i+\frac{1}{2}}=\mathfrak{P}^0_{i+\frac{1}{2}}$. Hence, the proposition holds true taking $r = 0$. We therefore assume that $\min\left(\alpha^p_i,\alpha^q_{i+1}\right) \neq \max\left(\alpha_i^p-\alpha^p_{i+1},0\right)$: by (\ref{almost_r}), 
\begin{equation}
\label{r}
r:= \frac{ \mathcal{P}_{i+\frac{1}{2}}\left[\Sigma_p,\Sigma_q\right] - \max\left(\alpha_i^p-\alpha^p_{i+1},0\right)}{\min\left(\alpha^p_i,\alpha^q_{i+1}\right) - \max\left(\alpha_i^p-\alpha^p_{i+1},0\right) }\in [0,1]
\end{equation}
It is straightforward then to see that (\ref{Convex-pq}) holds true. Moreover,
\begin{equation}\label{den}
\begin{split}
\min\left(\alpha^p_i,\alpha^q_{i+1}\right) - \max\left(\alpha_i^p-\alpha^p_{i+1},0\right) \,&\myeq{(\ref{fact})}\, \alpha_{i}^p - \max\left(\alpha^p_{i}-\alpha^q_{i+1},0\right)-\Bigg(\alpha^p - \min\left(\alpha^p_i,\alpha^p_{i+1}\right)\Bigg)\\
&= - \Bigg(\max\left(\alpha^p_{i}-\alpha^q_{i+1},0\right) - \min\left(\alpha^p_i,\alpha^p_{i+1}\right) \Bigg)
\end{split}
\end{equation}
and
\begin{equation}\label{nom}
\begin{split}
\mathcal{P}_{i+\frac{1}{2}}\left[\Sigma_p,\Sigma_q\right] - \max\left(\alpha_i^p-\alpha^p_{i+1},0\right) &\,\myeq{(\ref{fact})} \mathcal{P}_{i+\frac{1}{2}}\left[\Sigma_p,\Sigma_q\right] - \alpha^p_i + \min\left(\alpha^p_i,\alpha^p_{i+1}\right)\\
&\myeq{(\ref{Condi})} - \mathcal{P}_{i+\frac{1}{2}}\left[\Sigma_p,\Sigma_p\right] + \min\left(\alpha^p_i,\alpha^p_{i+1}\right)
\end{split}
\end{equation}
Inserting (\ref{den}) and (\ref{nom}) into (\ref{r}), we get an equivalent definition of $r$, namely
\begin{equation}\label{r_equiv}
r = \frac{ \mathcal{P}_{i+\frac{1}{2}}\left[\Sigma_p,\Sigma_p\right] - \min\left(\alpha^p_i,\alpha^p_{i+1}\right) }{\max\left(\alpha^p_{i}-\alpha^q_{i+1},0\right) - \min\left(\alpha^p_i,\alpha^p_{i+1}\right)}
\end{equation}
which implies (\ref{Convex-pp}).
\item \underline{$r$ does not depend on $p$}: We are going to show that the quotients (\ref{r})-(\ref{r_equiv}) are in fact non-depending of $p$, namely the one induced by the choice $p = k$ coincide with the one induced by $p = l$, for any $k\neq l \in\lbrace 1,2\rbrace$.\\
By the equivalence between (\ref{r}) and (\ref{r_equiv}), it is enough to show that
\begin{equation}\label{independence}
\frac{ \mathcal{P}_{i+\frac{1}{2}}\left[\Sigma_k,\Sigma_l\right] - \max\left(\alpha_i^k-\alpha^k_{i+1},0\right)}{\min\left(\alpha^k_i,\alpha^l_{i+1}\right) - \max\left(\alpha_i^k-\alpha^k_{i+1},0\right) } = 
\frac{ \mathcal{P}_{i+\frac{1}{2}}\left[\Sigma_l,\Sigma_k\right] - \max\left(\alpha_i^l-\alpha^l_{i+1},0\right)}{\min\left(\alpha^l_i,\alpha^k_{i+1}\right) - \max\left(\alpha_i^l-\alpha^l_{i+1},0\right) }
\end{equation}
Subtraction of (\ref{Condi+1}) from (\ref{Condi}), when $p=k$ and $q = l$, yields
\begin{equation}\label{pro_change}
\mathcal{P}_{i+\frac{1}{2}}\left[\Sigma_k,\Sigma_l\right] = \mathcal{P}_{i+\frac{1}{2}}\left[\Sigma_l,\Sigma_k\right] + \alpha_i^k-\alpha^k_{i+1}
\end{equation}
Notice that,
\begin{equation}
\begin{split}	\label{max_change}
\alpha_i^k-\alpha^k_{i+1} - \max\left(\alpha_i^k-\alpha^k_{i+1},0\right) &\myeq{(\ref{vol})}\, \alpha_{i+1}^l-\alpha^l_{i} - \max\left(\alpha_{i+1}^l-\alpha^l_{i},0\right)\myeq{(\ref{fact})}\, \min\left(\alpha_{i+1}^l,\alpha^l_{i}\right) -\alpha^l_{i}\\
&\myeq{(\ref{fact})}\, - \max\left(\alpha_i^l-\alpha^l_{i+1},0\right)
\end{split}
\end{equation}
and
\begin{equation}\label{min_change}
\begin{split}
\min\left(\alpha^k_i,\alpha^l_{i+1}\right) - \alpha^k_i + \alpha^k_{i+1} &\myeq{(\ref{fact})} -\max\left(\alpha^k_i - \alpha^l_{i+1},0\right) + \alpha^k_{i+1} \myeq{(\ref{vol})} -\max\left(\alpha^k_{i+1} - \alpha^l_{i},0\right) + \alpha^k_{i+1}\\
&\myeq{(\ref{fact})} \min(\alpha^l_{i},\alpha^k_{i+1})
\end{split}
\end{equation}
Equations (\ref{pro_change}), (\ref{max_change}) and (\ref{min_change}) into the left hand side of (\ref{independence}), leads to the desired equality.
\end{enumerate}
\end{proof}

\section{The Continuous Limit}
\label{appendix:ContLimit}

In this section, we aim at deriving a set of PDEs for the simulation of multiphase flow phenomena. This can be achieved by deriving the continuous limit that the set of discrete ODEs \ref{semi-discrete_form} is converging to. As discussed for the relaxation term, the convergence of each single term involved in the system of ODEs is solver-dependent, in principle. One possibility to circumvent such difficulty is to fix a specific form of the RS, which allows for computations. We choose the assumption (\ref{eq:RSassumption}). Furthermore, we make also the following simplification: let us assume that the RS under use computes the contact-discontinuity speed $\sigma$ and pressure $p^*$ as follows
\begin{align}\label{eq:Interfacial}
\sigma\left(
\begin{bmatrix}
\rho\\
u\\
p
\end{bmatrix}_L,
\begin{bmatrix}
\rho\\
u\\
p
\end{bmatrix}_R
\right) 
&= \frac{Z_L u_L+Z_R u_R}{Z_L+Z_R} - \frac{P_R-P_L}{Z_L+Z_R}\\
p^*\left(
\begin{bmatrix}
\rho\\
u\\
p
\end{bmatrix}_L,
\begin{bmatrix}
\rho\\
u\\
p
\end{bmatrix}_R
\right)
&=\frac{Z_Rp_L+Z_Lp_R}{Z_L+Z_R}
-\frac{Z_LZ_R(u_R-u_L)}{Z_L+Z_R}
\end{align}
where $Z_k = \rho_kc_k$ $k=L,R$ denotes the acoustic impedances computed by the solver and $c_k$ is an approximation to the sound speed. Specifically, we always assume that the internal energy can be described in terms of the independent variables $\rho_k$ and $p_k$, i.e. $e_k = e_k(\rho_k,p_k)$ denotes the EOS, so that the sound speed
is denoted as
\begin{equation}
\label{SOS}
a_k := a_k\left(\rho_k,p_k\right):=\sqrt{\frac{p_k}{\rho_k^2\partial_{p_k}e_k}-\frac{\partial_{\rho_k} e_k}{\partial_{p_k }e_k}}.
\end{equation}
For the case of the acoustic solver (see. \cite{ToroRS}, page 299-300), we simply get $c_k = a_k$.\\
Conversely, for the case of the HLLC solver (see \cite{ToroRS,ToroHLLC}) one has that $c_k = u_k - S_k$, where $S_k$ denotes the fasted signal speed on the $k$-th side.\\
Notice that each of the aforementioned interfacial solvers can be written into the sum of a symmetric part and anti-symmetric part, namely:

\begin{equation*}
\sigma(L,R) = S(L,R) - AS(L,R),
\qquad
S(L,R):= \frac{Z_L u_L+Z_R u_R}{Z_L+Z_R},
\quad
AS(L,R) := \frac{P_R-P_L}{Z_L+Z_R}
\end{equation*}
so that $S(L,R) = S(R,L)$ and $AS(L,R) = -AS(L,R)$.\\

We split the analysis into several contributions

\subsection{Relaxation Terms}

Based on the assumption (\ref{eq:RSassumption}), we get that
\begin{align*}
\mathcal{E}_{relax}\left[F^{lag}\right]_i
&= 
\mathcal{E}\left[\frac{N_{int}}{\Delta x}\right]
\Bigg(
F^{lag}_{lk} - F^{lag}_{kl}
\Bigg) = 
\begin{bmatrix}
\sigma_{kl} - \sigma_{lk}\\
0\\
p^*_{lk} - p^*_{kl}\\
p^*_{lk}\sigma_{lk} - p^*_{kl}\sigma_{kl}
\end{bmatrix}\\
&=
\mathcal{E}\left[\frac{N_{int}}{\Delta x}\right]
\Bigg(
(\sigma_{lk} - \sigma_{kl})
\begin{bmatrix}
-1\\
0\\
0\\
\frac{1}{2}\left(p^*_{lk} + p^*_{kl}\right)
\end{bmatrix}
+
(p^*_{lk} - p^*_{kl})
\begin{bmatrix}
0\\
0\\
1\\
\frac{1}{2}\left(\sigma_{lk} + \sigma_{kl}\right)
\end{bmatrix}
\Bigg)
\end{align*}
Simple algebraic manipulations by using (\ref{eq:Interfacial}) and the symmetric-antisymmetric splitting show that
\begin{align*}
\sigma_{lk} + \sigma_{kl} &= 2 \frac{Z_ku_k + Z_lu_l}{Z_l+Z_k} =: 2u^{'}_I
\qquad
\qquad
\sigma_{lk} - \sigma_{kl} = -2 \frac{p_k-p_l}{Z_k + Z_l}
\\
p^*_{lk} + p^*_{kl} &= 2 \frac{Z_kp_k + Z_lp_l}{Z_l+Z_k} =: 2p^{'}_I
\qquad
\qquad
p^*_{lk} - p^*_{kl} = -2 Z_kZ_l\frac{u_k-u_l}{Z_k + Z_l}
\end{align*}
Plugging these latter into the relaxation form, one concludes
\begin{align*}
\mathcal{E}_{relax}\left[F^{lag}\right]_i
&=
\mathcal{E}\left[\frac{N_{int}}{\Delta x}\right]
\Bigg(
-2 \frac{p_k-p_l}{Z_k + Z_l}
\begin{bmatrix}
-1\\
0\\
0\\
\frac{Z_kp_k + Z_lp_l}{Z_l+Z_k}
\end{bmatrix}
-
2 Z_kZ_l\frac{u_k-u_l}{Z_k + Z_l}
\begin{bmatrix}
0\\
0\\
1\\
\frac{Z_ku_k + Z_lu_l}{Z_l+Z_k}
\end{bmatrix}
\Bigg)\\
&=
\mathcal{E}\left[\frac{N_{int}}{\Delta x}\right]
\Bigg(
-2 \frac{p_k-p_l}{Z_k + Z_l}
\begin{bmatrix}
-1\\
0\\
0\\
p^{'}_I
\end{bmatrix}
-
2 Z_kZ_l\frac{u_k-u_l}{Z_k + Z_l}
\begin{bmatrix}
0\\
0\\
1\\
u^{'}_I
\end{bmatrix}
\Bigg)
\end{align*}
By defining the parameters 
\begin{equation*}
\mu := \mathcal{E}\left[\frac{N_{int}}{\Delta x}\right] \frac{2}{Z_k+Z_l}
\qquad
\lambda := Z_kZ_l\mu
\end{equation*}
one gets that
\begin{align*}
\mathcal{E}_{relax}\left[F^{lag}\right]_i
&=
-\mu(p_k-p_l)
\begin{bmatrix}
-1\\
0\\
0\\
p^{'}_I
\end{bmatrix}
-
\lambda (u_k-u_l)
\begin{bmatrix}
0\\
0\\
1\\
u^{'}_I
\end{bmatrix}
= 
\begin{bmatrix}
\mu(p_k-p_l)\\
0\\
-\lambda (u_k-u_l)\\
-\mu(p_k-p_l)p^{'}_I -
\lambda (u_k-u_l) u^{'}_I
\end{bmatrix}
\end{align*}
By assuming that the relative number of interfaces $\mathcal{E}\left[\frac{N_{int}}{\Delta x}\right]$ remains bounded as $\Delta x\rightarrow 0$, the continuous limit for the relaxation term is derived.

\subsection{Conservative Terms}
\noindent
The convergence of conservative fluxes is readily provided: 
by the finite difference approximation, one gets
\begin{equation}
\frac{\mathcal{E}_{i+\frac{1}{2}}\left[X^{(k)}F\right] - \mathcal{E}_{i-\frac{1}{2}}\left[X^{(k)}F\right]}{\Delta x}\longrightarrow \frac{\partial}{\partial x} \mathbb{E}\left[ X^{(k)}\textbf{F}^{(k)}\right]
\end{equation}
Under the assumption that kinetic-fluctuations may be disregarded (see \ref{EA_flux}), one gets that
\begin{equation*}
\frac{\partial}{\partial x} \mathbb{E}\left[ X^{(k)}\textbf{F}^{(k)}\right]\approx \frac{\partial}{\partial x} \alpha_k\textbf{F}_{k}
\end{equation*}
so that the conservative terms of the continuous limits are proven.

\subsection{Non-Conservative Terms}
\noindent
Inserting the new set of probabilities the boundary terms can be split into
\begin{equation*}
\begin{split}
\frac{\mathcal{E}_{boundary}\left[F^{lag}\right]_i}{\Delta x} 
&= \frac{1}{\Delta x} 
\Bigg( 
r_{i+\frac{1}{2}}F_{disp,i+\frac{1}{2}}^{lag} + 
r_{i-\frac{1}{2}}F_{disp,i-\frac{1}{2}}^{lag}
+ 
(1-r_{i+\frac{1}{2}})F_{strat,i+\frac{1}{2}}^{lag} + 
(1-r_{i-\frac{1}{2}})F_{strat,i-\frac{1}{2}}^{lag} 
\Bigg)\\
& = 
\frac{1}{\Delta x}
\Bigg(
\underbrace{
F_{strat,i+\frac{1}{2}}^{lag} + F_{strat,i-\frac{1}{2}}^{lag} 
}
+
\underbrace{
r_{i+\frac{1}{2}}
\left(
F_{disp,i+\frac{1}{2}}^{lag} - F_{strat,i+\frac{1}{2}}^{lag}
\right)
+
r_{i-\frac{1}{2}}
\left(
F_{disp,i-\frac{1}{2}}^{lag} - F_{strat,i-\frac{1}{2}}^{lag}
\right)
}
\Bigg)\\
&\qquad \qquad\qquad 
=: F_{strat}^{lag} 
\qquad\qquad\qquad\qquad\qquad\qquad\qquad\qquad 
=: F_{disp}^{lag}
\end{split}
\end{equation*}
where
\begin{align*}
F^{lag}_{strat,i+\frac{1}{2}} & := \left(\beta^{(l,k)}_{i+\frac{1}{2}}\right)^{-}\max(\alpha_{i+1}^k - \alpha_{i}^k,0)F^{lag,(l,k)}_{i+\frac{1}{2}} - \left(\beta^{(k,l)}_{i+\frac{1}{2}}\right)^{-}\max(\alpha_i^k-\alpha_{i+1}^k,0)F^{lag,(k,l)}_{i+\frac{1}{2}}\\
F^{lag}_{strat,i-\frac{1}{2}} & := \left(\beta^{(l,k)}_{i-\frac{1}{2}}\right)^{+}\max(\alpha_{i}^k - \alpha_{i-1}^k,0)F^{lag,(l,k)}_{i-\frac{1}{2}} - \left(\beta^{(k,l)}_{i-\frac{1}{2}}\right)^{+}\max(\alpha_{i-1}^k-\alpha_{i}^k,0)F^{lag,(k,l)}_{i-\frac{1}{2}}\\
F^{lag}_{disp,i+\frac{1}{2}} & := \left(\beta^{(l,k)}_{i+\frac{1}{2}}\right)^{-}\min(\alpha_{i}^l, \alpha_{i+1}^k)F^{lag,(l,k)}_{i+\frac{1}{2}} - \left(\beta^{(k,l)}_{i+\frac{1}{2}}\right)^{-}\min(\alpha_i^k,\alpha_{i+1}^l)F^{lag,(k,l)}_{i+\frac{1}{2}}\\
F^{lag}_{disp,i-\frac{1}{2}} & := \left(\beta^{(l,k)}_{i-\frac{1}{2}}\right)^{+}\min(\alpha_{i-1}^l,\alpha_{i}^k)F^{lag,(l,k)}_{i-\frac{1}{2}} - \left(\beta^{(k,l)}_{i-\frac{1}{2}}\right)^{+}\min(\alpha_{i-1}^k,\alpha_{i}^l)F^{lag,(k,l)}_{i-\frac{1}{2}}\\
\end{align*}
Before detailing each term, we introduce the following convenient notation
\begin{align*}
f^+ := \max(f,0) = \frac{f + \vert f\vert}{2}
\qquad
f^- := \min(f,0) = \frac{f - \vert f\vert}{2}
\qquad
\delta_{i+\frac{1}{2}}^\pm \alpha^{k} := \left(\alpha_{i+1}^k-\alpha_i^k\right)^\pm
\end{align*}
So that also we rewrite the flux-indicators $\left(\beta^{(p,q)}_{i+\frac{1}{2}}\right)^\pm = \sigma_{i+\frac{1}{2}}^\pm(p,q)/\vert\sigma_{i+\frac{1}{2}}\vert$.

\subsubsection{Stratified-Flow terms}
 
This continuous limit was firstly derived in \cite{Saurel2003}, we recall it for the sake of completeness. 
By utilizing the aforementioned notation we get
\begin{align*}
F_{strat}^{lag}
&=
\frac{\sigma_{i-\frac{1}{2}}^+(l,k)}{\vert \sigma_{i-\frac{1}{2}}(l,k)\vert}\delta^+_{i-\frac{1}{2}}\alpha^k F^{lag}_{i-\frac{1}{2}}(l,k)
+
\frac{\sigma_{i-\frac{1}{2}}^+(k,l)}{\vert \sigma_{i-\frac{1}{2}}(k,l)\vert}\delta^-_{i-\frac{1}{2}}\alpha^k F^{lag}_{i-\frac{1}{2}}(k,l)\\
&\qquad
-
\frac{\sigma_{i+\frac{1}{2}}^-(l,k)}{\vert \sigma_{i+\frac{1}{2}}(l,k)\vert}\delta^+_{i+\frac{1}{2}}\alpha^k F^{lag}_{i+\frac{1}{2}}(l,k)
-
\frac{\sigma_{i+\frac{1}{2}}^-(k,l)}{\vert \sigma_{i+\frac{1}{2}}(k,l)\vert}\delta^-_{i+\frac{1}{2}}\alpha^k F^{lag}_{i+\frac{1}{2}}(k,l)
\end{align*}
As $\Delta x \rightarrow 0$, we perform the following approximations which hold under the assumption of smooth flow:
\begin{itemize}
\item $F^{lag}_{i+\frac{1}{2}}(l,k)=F^{lag}_{i-\frac{1}{2}}(l,k) = F^{lag}_{i}(l,k)$ as well as
$F^{lag}_{i+\frac{1}{2}}(k,l)=F^{lag}_{i-\frac{1}{2}}(k,l) = F^{lag}_{i}(k,l)$
\item $\sigma_{i+\frac{1}{2}}(k,l)=\sigma_{i-\frac{1}{2}}(k,l) =: \sigma_i(k,l)$ as well as $\sigma_{i+\frac{1}{2}}(l,k)=\sigma_{i-\frac{1}{2}}(l,k) =: \sigma_i(l,k)$
\end{itemize} 
so that one writes
\begin{align*}
\delta^+_i\alpha^k := \frac{\sigma^+(l,k)}{\vert\sigma(l,k)\vert}\delta^+_{i-\frac{1}{2}}\alpha^k - \frac{\sigma^-(l,k)}{\vert\sigma(l,k)\vert} \delta^+_{i+\frac{1}{2}}\alpha^k
\qquad
\delta^-_i\alpha^k := -\frac{\sigma^+(k,l)}{\vert\sigma(k,l)\vert}\delta^-_{i-\frac{1}{2}}\alpha^k + \frac{\sigma^-(k,l)}{\vert\sigma(k,l)\vert} \delta^-_{i+\frac{1}{2}}\alpha^k
\end{align*}
and
\begin{equation*}
\frac{1}{\Delta x} F^{lag}_{strat} \approx \frac{\delta^+_i\alpha^k F^{lag}_i(l,k) - \delta^-_i\alpha^k F^{lag}_i(k,l)}{\Delta x}
=
\frac{\delta_i^+\alpha^k}{\Delta x}
\begin{bmatrix}
-\sigma_i(l,k)\\
0\\
p^*_i(l,k)\\
p^*_i(l,k)\sigma_i(l,k)
\end{bmatrix}
-
\frac{\delta_i^-\alpha^k}{\Delta x}
\begin{bmatrix}
-\sigma_i(k,l)\\
0\\
p^*_i(k,l)\\
p^*_i(k,l)\sigma_i(k,l)
\end{bmatrix}
\end{equation*}
where we used assumption (\ref{eq:RSassumption}) and the interfacial quantities are computed as in (\ref{eq:Interfacial}).\\
By using the symmetric-antisymmetric splitting of interfacial quantities, one can rearrange each equation in the form
\begin{equation*}
\frac{\delta^+_i\alpha^k}{\Delta x}(S(l,k)-AS(l,k)) - \frac{\delta^-_i\alpha^k}{\Delta x}(S(k,l) - AS(k,l)) = \frac{\delta^+_i\alpha^k - \delta^-_i\alpha^k}{\Delta x}
\left( 
S (k,l) + 
\frac{\delta^+_i\alpha^k + \delta^-_i\alpha^k}{\delta^+_i\alpha^k - \delta^-_i\alpha^k}
AS(k,l) 
\right)
\end{equation*}
So that
\begin{equation}
\frac{1}{\Delta x} F^{lag}_{strat} \rightarrow \begin{bmatrix}
-u_I\\
0\\
p_I\\
p_Iu_I
\end{bmatrix}
\frac{\partial \alpha_k}{\partial x}
\end{equation}
where interfacial quantities are defined as 
\begin{equation*}
p_I := p_I^{'} + \mathrm{sign(\partial_x \alpha_k)} \frac{Z_kZ_l}{Z_k+Z_l}(u_l-u_k), \qquad u_I := u_I^{'} + \mathrm{sign(\partial_x \alpha_k)} \frac{1}{Z_k+Z_l}(p_l-p_k)
\end{equation*}

\subsubsection{Disperse-Flow Terms}

Here we assume that the variation with respect to the parameter $r$ is smooth, so that we conclude that
\begin{itemize}
\item $r_{i+\frac{1}{2}} = r_{i-\frac{1}{2}} =: r_i$.
\end{itemize}
\newpage
Under such assumption, the disperse term can be rearranged as
\begin{align*}
\frac{F^{lag}_{disp}}{\Delta x} &= -
\underbrace{ 
r_i \frac{\left(F^{lag}_{strat,i+\frac{1}{2}} + F^{lag}_{strat,i-\frac{1}{2}}\right)}{\Delta x}
} 
+ 
\frac{r_i}{\Delta x} 
\underbrace{
\left(F^{lag}_{disp,i+\frac{1}{2}} + F^{lag}_{disp,i-\frac{1}{2}}\right)
}\\
&
\qquad\qquad
\longrightarrow r
\begin{bmatrix}
-u_I\\
0\\
p_I\\
p_Iu_I
\end{bmatrix}
\frac{\partial \alpha^k}{\partial x}
\qquad\qquad\qquad\qquad
=: F_{disp}
\end{align*}
due to the discussion of previous subsection. We then focus on the convergence of the second term.\\
Under the hypotheses performed for the previous section, we can write
\begin{equation*}
F_{disp} = \alpha^{disp}_{lk} F_i^{lag}(l,k) + \alpha^{disp}_{lk} F^{lag}_i(k,l)
\end{equation*}
where
\begin{align*}
\alpha^{disp}_{lk} &= -\frac{\sigma^-_i(l,k)}{\vert \sigma^-_i(l,k)\vert}\min(\alpha_i^l,\alpha_{i+1}^k) + \frac{\sigma^+_i(l,k)}{\vert \sigma^+_i(l,k)\vert}\min(\alpha_{i-1}^l,\alpha_{i}^k)\\
\alpha^{disp}_{kl} &= -\frac{\sigma^-_i(k,l)}{\vert \sigma^-_i(k,l)\vert}\min(\alpha_i^k,\alpha_{i+1}^l) + \frac{\sigma^+_i(k,l)}{\vert \sigma^+_i(k,l)\vert}\min(\alpha_{i-1}^k,\alpha_{i}^l)
\end{align*}
Hence, by analogous splitting to the one performed above, one gets that
\begin{equation*}
\frac{F_{disp}}{\Delta x} \rightarrow \frac{\partial}{\partial x}
\Bigg(
\alpha_{disp}
\begin{bmatrix}
-u_I\\
0\\
p_I\\
p_Iu_I
\end{bmatrix}
\Bigg)
\end{equation*}
where $\alpha_{disp}$ denotes the volume fraction with lowest value.

\section{A solver-invariant equilibrium variety}
\label{appendix:variety}

In this section we are concerned with the proof of a result concerning the equilibrium variety of (\ref{eq:relaxODE}). For the sake of simplicity we will avoid the subscript $i$, meaning that all the following considerations hold cell-wise. This means $U_p := \left(\textbf{U}_{p}\right)_i$ for each $p\in\lbrace 1, 2\rbrace$. Furthermore, we will make use of the following notation: $F^*(U_L,UR)$, $\sigma(U_L,U_R)$, $U^*(U_L,U_R)$ denote the flux, the interface/contact discontinuity speed and the solution (i.e. the Godunov state) generated from the resolution of the RP
\begin{align*}
&\partial_t \textbf{U} + \partial_x \textbf{F}(\textbf{U}) = 0\\
&\textbf{U}(x,0) = \begin{cases}
\textbf{U}_L & x<0\\
\textbf{U}_R & x>0
\end{cases}
\end{align*}
by means of a prescribed RS. For the sake of brevity, we will also use $q^*_{pq}$ to denote the resulting quantity $q$ in the star region yielded by the resolution of the RP with initial data $U_L=U_p$ and $U_R=U_q$ as referring to the RP between the phases $p$ and $q$.\\
Without loss of generality, we consider the relaxation term (\ref{eq:relaxODE}) for phase $k\neq l \in\lbrace 1,2\rbrace$, which reads
\begin{equation}\label{eq:conds}
\begin{cases}
\sigma\left(U_k, U_l\right) - \sigma(U_l,U_k) = 0\\
F^*_I\left(U_l,U_k\right) - \sigma\left(U_l,U_k\right)U^*_I\left(U_l,U_k\right) - F^*_I\left(U_k,U_l\right) + \sigma\left(U_k,U_l\right) U^*_I\left(U_k,U_l\right) = 0
\end{cases}
\end{equation}
where the sub-index $I$ denotes the evaluation of the corresponding quantity close to the interface from the side of phase $k$.
By the evaluation of each Lagrangian flux to the interface and assumption (\ref{eq:RSassumption}), one has that
\begin{equation}
F^*_I\left(U_l,U_k\right) - \sigma\left(U_l,U_k\right)U^*_I\left(U_l,U_k\right)
= p^*_{lk}D^*_{lk} = p^*_{lk}\begin{bmatrix}
0\\
1\\
\sigma_{lk}
\end{bmatrix},
\qquad\qquad
\forall l\neq k.
\end{equation}
Hence, plugging this latter into (\ref{eq:conds}) one gets that the equilibrium variety is defined by the set of ODEs
\begin{equation}\label{eq:condProof}
0 = 
\begin{bmatrix}
\sigma_{kl} - \sigma_{lk}\\
0\\
p^*_{lk} - p^*_{kl}\\
p^*_{lk}\sigma_{lk} - p^*_{kl}\sigma_{kl}
\end{bmatrix}
=
(\sigma_{lk} - \sigma_{kl})
\begin{bmatrix}
-1\\
0\\
0\\
\frac{1}{2}\left(p^*_{lk} + p^*_{kl}\right)
\end{bmatrix}
+
(p^*_{lk} - p^*_{kl})
\begin{bmatrix}
0\\
0\\
1\\
\frac{1}{2}\left(\sigma_{lk} + \sigma_{kl}\right)
\end{bmatrix}
\end{equation}
Solving such system of ODEs implies the well-known conditions on relaxed states
\begin{equation*}
\sigma_{kl}=\sigma_{lk}=S^\infty
\qquad
p^*_{lk} = p^*_{kl} = p^*.
\end{equation*}

Notice that, the first equation in (\ref{eq:condProof}) is actually a trivial equation, $0 = 0$. This is indeed stating that \emph{no matter the values of $\textbf{U}^\infty$, conditions (\ref{eq:conds}) for mass are always fulfilled}. From the point of view of our ODE (\ref{eq:relaxODE}) this implies the following
\begin{equation}\label{eq:ConstantVF}
\frac{d}{dt} (\alpha_k\rho_k) = 0 \qquad \Rightarrow \qquad \alpha_k\rho_k = \textit{const}.
\end{equation}
over the relaxation step.\\
Therefore, by assumption on the equilibrium variety, the primitive variables vector of relaxed states can be rewritten as

\begin{equation}\label{Charcetrization}
V^\infty_k = \begin{bmatrix}
\alpha^\infty_k\\
\rho^\infty_k\\
S_\infty\\
p_\infty
\end{bmatrix} 
\qquad \qquad
V^\infty_l = \begin{bmatrix}
\alpha^\infty_l\\
\rho^\infty_l\\
S_\infty\\
p_\infty
\end{bmatrix}
\end{equation}
so that the relaxed volume fractions are given by
$\alpha_k^\infty = \alpha_k^0\rho_k^0 / \rho^\infty_k$, by (\ref{eq:ConstantVF}).
Therefore, a natural Maxwellian $M$ is defined as
\begin{equation*}
\textbf{u} = \begin{bmatrix}
\alpha^\infty_k\\
\rho^\infty_k\\
S^\infty\\
p^\infty\\
\alpha_l^\infty\\
\rho^\infty_l\\
\end{bmatrix}
\longmapsto M(\textbf{u}) = \begin{bmatrix}
\alpha^\infty_k\\
\alpha^\infty_k U^\infty_k\\
\alpha^\infty_l\\
\alpha^\infty_l U^\infty_l
\end{bmatrix}
\end{equation*}

\section{Relaxation Strategies}
\label{appendix:RelaxationStrategies}

One important feature of two-phase flow models is to correctly model the interaction between mixture phases. This has been studied for example in  \cite{Bdzil, Saurel&Abgrall}. Several strategies have been developed so far, and one robust approach is to model interaction between phases by means of relaxation procedure, typically involving stiff source terms. As firstly suggested by Abgrall and Saurel in \cite{Abgrall&Saurel}, if the relaxation term $\mathcal{E}_{relax}\left[F^{lag}\right]_i$ in (\ref{semi-discrete_form}) consists of moderate amount of bubbles, standard resolution of (\ref{semi-discrete_form}) can be applied. However, it is usual to associate such relaxation terms to a large values of source terms, i.e. large numbers of disperse particles are considered. Therefore, relaxation strategies that capture the equilibrium states have to be derived. For the seven-equation model, standard techniques are given in \cite{Saurel&Abgrall, Saurel2001, Lallemand}, in which velocity and pressure relaxation steps are split into subsequent operators.

In this work we propose two relaxation strategies that aims at deriving equilibrium states avoiding further splitting methods.

\subsection{A single continuous limit relaxation}
\label{subsec:relaxSingle}

A well-established procedure is to compute the equilibrium values of the unknown solution by firstly deriving a set of ODEs as limit of (\ref{eq:relaxODE}) as $\Delta x\rightarrow 0$. The resulting system of ODEs is then solved in time, determining the equilibrium states. Following this line we propose a unique relaxation: after application of the hyperbolic operator, by means of an approximation of an acoustic solver \cite{Murrone}, the continuous limit of the relaxation term reads \cite{Saurel2003} for each $k\neq l\in\lbrace 1,2 \rbrace$
\begin{equation}\label{Rel_full}
\begin{split}
\frac{d}{dt}& \alpha_k = \mu(p_k-p_l)\\
\frac{d}{dt}&(\alpha_k\rho_k)=0\\
\frac{d}{dt}&(\alpha_k\rho_k u_k) = \lambda (u_{l}-u_k) \\
\frac{d}{dt}&(\alpha_k\rho_kE_k) = \mu p^{'}_I(p_l-p_k) + \lambda u_i^{'}(u_{l}-u_k)
\end{split}
\end{equation}
where $p^{'}_I$, $u^{'}_I$ are given by
\begin{equation}
\label{meaninterfacials2}
p_I^{'} = \frac{Z_k p_l + Z_l p_k}{Z_k + Z_l},\qquad u_I^{'} = \frac{Z_k u_k + Z_l u_l}{Z_k + Z_l}
\end{equation}
where $Z_k = \rho_k a_k$ denotes the acoustic impedance of phase $k$.

It is not difficult to show that this system of ODEs results to have a single velocity and a single pressure, as $\lambda, \mu \rightarrow \infty$. We denote by $u^*$ and $p^*$ the relaxed velocity and relaxed pressure, respectively.\\
Notice that the conservation over the relaxation procedure of the quantity $\alpha_k\rho_k$ leads to the following reformulation of mass and momentum equations
\begin{equation}\label{MaMo-eq}
\frac{d}{dt}\rho_k = -\frac{\rho_k}{\alpha_k}\frac{d}{dt}\alpha_k,\qquad \alpha_k \rho_k \frac{d}{dt} u_k = \lambda (u_{l}-u_k).
\end{equation}
Summing over phase index $k$ the momentum equation and integrating over the relaxation step, we get
\begin{equation*}
(\alpha\rho)_{0,1} (u^* - u_{0,1}) + (\alpha\rho)_{0,2} (u^* - u_{0,2}) = 0
\end{equation*}
from which we deduce
\begin{equation}\label{u_mix}
u^* = \frac{(\alpha\rho u)_{0,1} + (\alpha\rho u)_{0,2}}{(\alpha\rho)_{0,1} + (\alpha\rho)_{0,2}}
\end{equation}
where the sub-index $0$ stands for the value resulting from the hyperbolic operator. Moreover, the energy equation can be rewritten as
\begin{equation}
\alpha_k\rho_k\left(u_k \frac{d}{dt}u_k + \frac{d}{dt}e_k\right) = - p^{'}_I\frac{d}{dt}\alpha_k + \alpha_k\rho_k u^{'}_I \frac{d}{dt}u_k
\end{equation}
which, by means of first equation in (\ref{MaMo-eq}), yields
\begin{equation}\label{energy_quasi}
\frac{d}{dt} e_k =  (u_I^{'} -u_k) \frac{d}{dt} u_k + \frac{p_I^{'}}{\rho_k^2}\frac{d}{dt}\rho_k = u_I^{'}\frac{d}{dt} u_k - \frac{d}{dt}\left(\frac{1}{2} u_k^2\right) - p_I^{'}\frac{d}{dt}\left(\frac{1}{\rho_k}\right) 
\end{equation}
Integration between the pre-relaxed time $t_0$ and the relaxed time $t^*$ yields
\begin{equation}
e_k^* - e_{k0} = \frac{1}{2}(u^*-u_{0,k})\left[2\overline{u}^{'}_I - (u^*+u_k) \right] - \overline{p}^{'}_I\left(\frac{1}{\rho^*_k} - \frac{1}{\rho_{0,k}}\right)
\end{equation}
where $\overline{u_I^{'}} := \frac{1}{ u^* -u_{k0}}\int_{t_0}^{t^*} u_I^{'}\frac{d}{dt} u_k\, dt$ and $\overline{p_I^{'}} := \frac{1}{\frac{1}{\rho_k^*} - \frac{1}{\rho_{k0}}} \int_{t_0}^{t^*} p_I^{'} \left(\frac{1}{\rho_k}\right)\, dt$. Following the work of \cite{Saurel2007b}, a possible choice that has been shown to be compatible with the entropy inequality and with energy conservation is
\begin{equation*}
\overline{p}^{'}_I(t)\approx p^{'}_I(t^*) = p^*, \qquad\qquad \overline{u}^{'}_I(t)\approx u^{'}_I(t^*) = u^*
\end{equation*}
By means of such an approximation, we are led to compute the root of the following non-linear function
\begin{equation}\label{F12}
F_k = F_k (\rho_k, p) := 2\rho_k\rho_{k0}(e_k - e_{k0}) - \rho_k\rho_{k0}(u^* - u_{k0})^2 - 2p (\rho_k - \rho_{k0})   
\end{equation}
where $e_k = e_k(\rho_k, p)$ and $u^*$ is computed according to (\ref{u_mix}). The multivariate function $F = (F_1, F_2)$ depends on $2 +1=3$ variables, namely $\rho_k$ and $p$, so one equation is missing. We complete the system by enforcing fulfillment of the saturation condition:
\begin{equation}\label{F3}
F_3 := \sum_k \alpha_k - 1 = 0
\end{equation}
where, by virtue of mass conservation, $\alpha_k = \frac{(\alpha_k\rho_k)_0}{\rho_k}$. Hence, the Jacobian matrix of $F = (F_1,F_2,F_3)^T$ reads
\begin{equation}
DF :=
\begin{bmatrix}
A_1 & 0 & B_1\\
0 & A_2 & B_2\\
C_1 & C_2 & 0
\end{bmatrix}
\end{equation}
with definitions
\begin{equation}
\begin{split}
A_k &:= \frac{\partial F_k}{\partial\rho_k} = 2\rho_{k0}(e_k-e_{k0}) + 2\rho_k\rho_{k0} \partial_{\rho_k}e_k - \rho_{k0}(u^* - u_{k0})^2 - 2p\\
B_k &:= \frac{\partial F_k}{p} = 2\rho_k\rho_{k0}\partial_{p}e_k - 2(\rho_{k}-\rho_{k0})\\
C_k &:= \frac{\partial F_3}{\partial\rho_k} = -\frac{(\alpha_k\rho_k)_0}{{\rho_k}^2}
\end{split}
\end{equation}
The computation of the root of the the multivariate function $F$ is accomplished by means of a standard Newton-Raphson method. The iterative scheme is stopped when the relative increment is sufficiently small and a robust initial guess has been shown to be $F_0 = (\rho_{10},\rho_{20}, p_I^{'}(t_0))^T$. Therefore, the approximation of the equilibrium states $(\rho^*_1,\rho^*_2,p^*)$ can be summarized into the following algorithm:
\begin{enumerate}
\item Compute the mixture velocity $u^*$ according to (\ref{u_mix});
\item Compute $p^*, \rho_k^*$ for each $k$ by finding the roots of the non linear function $F$ given in (\ref{F12})-(\ref{F3});
\item Update velocity, density and pressure of each phase by assigning $u^*, \rho_k^*,p^*$.
\item Reconstruct the vector of conserved variables and go to the following time step.
\end{enumerate}

Notice that such procedure is a generalization of the standard splitting strategy proposed in \cite{Lallemand, Saurel2001,Saurel2003}.

\subsection{A projection-relaxation strategy}
\label{subsec:relaxproj}

We propose a second relaxation strategy, that aims at avoiding the computation of the continuous limit of the source term $R$ in (\ref{eq:relaxODE}). This is accomplished making use of the approach developed by Murrone et al in \cite{Murrone}. Such a strategy starts with the introduction of a small parameter to model the speed of the relaxation. More precisely, we introduce the relaxation time $\varepsilon \rightarrow 0$, so that (\ref{eq:7EM}) becomes
 
\begin{equation}\label{eq:7EMeps}
\frac{d}{d t}\left( \alpha^{(k)}_i\textbf{U}_i^{(k)}\right) + \frac{1}{\Delta x} G_i(\textbf{U}_i) = \frac{\lambda_i}{\epsilon} \textbf{R}(\textbf{U}_i)
\end{equation}

As discussed in Appendix \ref{appendix:variety}, we define a set of relaxed states $\textbf{u}$ which, upon mapping to conserved variables $\textbf{U}^\infty = M^\textbf{U}(\textbf{u})$, defines a root of the function $R$, namely $R\left(M^\textbf{U}(\textbf{u})\right) = 0$. Based on the discussion exposed in Appendix \ref{appendix:variety}, there exist a natural parametrization in terms of the primitive variables, namely

\begin{equation}
\textbf{u} = \begin{bmatrix}
\alpha_1\\
\rho_1\\
u\\
p\\
\alpha_2\\
\rho_2
\end{bmatrix} \longmapsto  M^\textbf{V}(\textbf{u}) = \textbf{V}^\infty = \begin{bmatrix}
\alpha_1\\
\rho_1\\
u\\
p\\
\alpha_2\\
\rho_2\\
u\\
p
\end{bmatrix}
\end{equation}
Then, looking for a solution of the form $\textbf{W} = M^\textbf{W}(\textbf{u}) + \epsilon \textbf{Y}$, one assumes that there exists an expansion of the source term $R$ such that

\begin{equation}
R(\textbf{W}) = R(M(\textbf{u})) + \epsilon \textbf{J}_\textbf{W} R(M(\textbf{u})) \textbf{Y} + \mathcal{O}(\epsilon^2)
\end{equation}
where $\textbf{J}$ is the Jacobian of the source term $R$ in terms of the variables $\textbf{W}$ evaluated at $\textbf{W} = M^\textbf{W}(\textbf{u})$. Then (\ref{eq:7EMeps}) becomes
\begin{equation}
\frac{d }{d t}\left( \alpha^{(k)}_i\textbf{U}_i^{(k)}\right) + \frac{1}{\Delta x} G_i(\textbf{U}_i) = \textbf{J}_\textbf{U} R(M(\textbf{u}))\textbf{Y} + \mathcal{O}(\epsilon) 
\end{equation}

If we are able to find the projection matrix $\textbf{P}_\textbf{U}$ onto the $\ker \textbf{J}_\textbf{U} R (M(\textbf{u}))$, then neglecting the second order terms we get

\begin{equation}\label{eq:schemeImplicit}
\textbf{P}_\textbf{U}\Bigg(\frac{d}{d t} \left( \alpha^{(k)}_i\textbf{U}_i^{(k)}\right) + \frac{1}{\Delta x} G_i(\textbf{U}_i) \Bigg) = 0
\end{equation}

Equations (\ref{eq:schemeImplicit}) tell us the following: advancing the solution with the hyperbolic step followed by the multiplication of $\textbf{P}_\textbf{U}$ is yielding relaxed states. Notice that, this strategy has in principle just the cost of a matrix vector multiplication, in contrast to the rich variety of iterative processes that can arise form solving the continuous limit (\ref{eq:relaxODE}) as $\lambda_i \rightarrow \infty$ .\\
One can show that each solver that admits the splitting form (\ref{eq:RSassumption}) are associated to the same projection matrix $\Pi$ proposed in \cite{AbgrallPerrier}, in case of transonic flow regimes. Indeed, the Jacobian matrix in terms of the primitive variables of $M$ reads
\begin{equation}\label{eq:Mjac}
dM^\textbf{V}_\textbf{u} := \textbf{J}_\textbf{u}M^\textbf{V} = \begin{bmatrix}
1 & 0 & 0 & 0 & 0 & 0 \\
0 & 1 & 0 & 0 & 0 & 0 \\
0 & 0 & 1 & 0 & 0 & 0  \\
0 & 0 & 0 & 1 & 0 & 0  \\
0 & 0 & 0 & 0 & 1 & 0  \\
0 & 0 & 0 & 0 & 0 & 1  \\
0 & 0 & 1 & 0 & 0 & 0  \\
0 & 0 & 0 & 1 & 0 & 0  \\
\end{bmatrix}
\end{equation}

If the solver under consideration admits the representation (\ref{eq:RSassumption}), the relaxation term (\ref{eq:conds}) for phase $k$ reduces to
\begin{equation}
\label{eq:Rreduced}
R_k(\textbf{U}) := 
 \begin{bmatrix}
 \sigma_{kl} - \sigma_{lk}\\
 0\\
 p_{lk} - p_{kl}\\
 p_{lk}\sigma_{lk} - p_{kl}\sigma_{kl}
 \end{bmatrix} = 
(\sigma_{lk} - \sigma_{kl})
\begin{bmatrix}
-1\\
0\\
0\\
\frac{1}{2}(p_{lk} + p_{kl})
\end{bmatrix}
+ 
(p_{lk} - p_{kl}) \begin{bmatrix}
0\\
0\\
1\\
\frac{1}{2}(\sigma_{lk} + \sigma_{kl})
\end{bmatrix}
\end{equation}
We then deduce clearly that the range of $\textbf{J}_\textbf{U} R(M(\textbf{u}))$ is spanned by the vectors
\begin{equation}
\textbf{V}_1 = \begin{bmatrix}
1\\
0\\
0\\
-p\\
-1\\
0\\
0\\
p
\end{bmatrix} 	
\qquad\qquad
\textbf{V}_2 = \begin{bmatrix}
0\\
0\\
1\\
u\\
0\\
0\\
-1\\
-u
\end{bmatrix}
\end{equation}

However, notice that the Jacobian of the Maxwellian is written in terms of the primitive variables, therefore we need to transform $\textbf{V}_j$ in terms of the primitive variables. This can be done, by computing the linear transformation $T$ between conservative and primitive variables (see \cite{AbgrallPerrier}), such that straightforward computations lead to

\begin{equation}
T(M(\textbf{u}))\textbf{V}_1 =\begin{bmatrix}
1\\
-\frac{\rho_1}{\alpha_1}\\
0\\
-\frac{\rho_1 a^2_1}{\alpha_1}\\
-1\\
\frac{\rho_2}{\alpha_2}\\
0\\
\frac{\rho_2 a^2_2}{\alpha_2}
\end{bmatrix}
\qquad\qquad
T(M(\textbf{u}))\textbf{V}_2 =\begin{bmatrix}
0\\
0\\
\frac{1}{\alpha_1\rho_1}\\
0\\
0\\
0\\
-\frac{1}{\alpha_2\rho_2}\\
0
\end{bmatrix}
\end{equation}

Assembling the matrix $S = [dM(\textbf{u}),T(M(\textbf{u}))\textbf{V}_1,,T(M(\textbf{u}))\textbf{V}_2]$ and inverting it, yields the projection matrix
\begin{equation}
\label{eq:proj}
\Pi = \begin{bmatrix}

1 & 0 & 0 & \frac{\alpha_1\alpha_2}{d} & 0 & 0 & 0 & -\frac{\alpha_1\alpha_2}{d}\\

0 & 1 & 0 & -\frac{\alpha_2\rho_1}{d} & 0 & 0 & 0 & \frac{\alpha_2\rho_1}{d}\\

0 & 0 & \frac{m_1}{m_1 + m_2} & 0 & 0 & 0 & \frac{m_2}{m_1+m_2} & 0\\

0 & 0 & 0 & \frac{\alpha_1\rho_2 a_2^2}{d} & 0 & 0 & 0 & \frac{\alpha_2\rho_1 a_1^2}{d}\\

0 & 0 & 0 & -\frac{\alpha_1\alpha_2}{d} & 1 & 0 & 0 & \frac{\alpha_2\alpha_1}{d}\\

0 & 0 & 0 & \frac{\alpha_1\rho_2}{d} & 0 & 1 & 0 & -\frac{\alpha_1\rho_2}{d}\\ 

\end{bmatrix}
\end{equation}
where $m_k := \alpha_k\rho_k$, $a_k$ denotes the sound speed of phase $k$ and $d := \alpha_1\rho_2a_2^2 + \alpha_2\rho_1a_1^2$. We point out that the form of $\Pi$ is independent on the EOS for each phase, but has been derived by the assumption on the RS (\ref{eq:RSassumption}). Therefore, the result of \cite{AbgrallPerrier} can be extended to each solver that fulfills (\ref{eq:RSassumption}).\\
Finally, we design our alternative relaxation strategy as follows

\begin{enumerate}
\item From the values coming from the hyperbolic step $\textbf{U}^0$, compute the vector of primitive variables $\textbf{V}^0$ and the projection matrix $\Pi$.
\item Compute the vector of reduced variables
$$\textbf{u}^\infty :=
\begin{bmatrix}
\alpha_1^\infty\\
\rho_1^\infty\\
u^\infty\\
p^\infty\\
\alpha_2^\infty\\
\rho_2^\infty
\end{bmatrix}
=
 \Pi \textbf{V}^0$$
\item Build up the vector of conserved variables $\textbf{U}^\infty = M^\textbf{U} (\textbf{u}^\infty)$ and $\textbf{V}^\infty= M^\textbf{V}
(\textbf{u}^\infty)$.
\end{enumerate}

\bibliographystyle{plain}
\bibliography{./biblio}

\end{document}